\documentclass[11pt]{amsart}
\usepackage[foot]{amsaddr}
\usepackage{a4wide}
\usepackage{graphicx,subcaption,caption}
\usepackage{tikz,tkz-euclide}
\usepackage{bbm}
\usepackage{amsthm,amsmath,amssymb,txfonts,mathrsfs} 
\usetikzlibrary{calc,math}
\usepackage{hyperref} 

\newtheorem{theorem}{Theorem}[section]
\newtheorem{corollary}{Corollary}[theorem]
\newtheorem{proposition}{Proposition}[theorem]
\newtheorem{lemma}[theorem]{Lemma}

\newcommand{\lozd}{--++(1,1)--++(1,0)--++(-1,-1)--++(-1,0)
}

\renewcommand{\Re}{\mathop{\mathrm{Re}}}
\renewcommand{\Im}{\mathop{\mathrm{Im}}}
\newcommand{\Tr}{\mathop{\mathrm{Tr}}}
\newcommand{\Var}{\mathop{\mathrm{Var}}}

\newcommand{\eps}{\varepsilon}
\newcommand{\R}{\mathbb R}
\newcommand{\dv}{\partial}
\newcommand{\LL}{\mathcal L}
\newcommand{\C}{\mathbb C}
\newcommand{\A}{\mathcal A}
\newcommand{\F}{\mathcal F}
\newcommand{\dbb}{\overline{z}}
\newcommand{\Di}{\mathbb{D}}
\newcommand{\Om}{\Omega}
\newcommand{\unscaledtime}{s}
\title[]{Lozenge tilings of a hexagon and $q$-Racah ensembles}
\author[]{Maurice Duits, Erik Duse and Wenkui Liu}
\address{Department of Mathematics, Royal Institute of Technology, Lindstedtsvägen 25, SE 10044, Stockholm
Sweden.}
\email{duits@kth.se, duse@kth.se, wenkui@kth.se}
\thanks{ MD and WL were supported by the Swedish Research Council (VR), grant no. 2021-06015, and the European Research Council (ERC), Grant Agreement No. 101002013. ED was supported by the Swedish Research Council (VR), grant no. 2019-04152 and the Knut and Alice Wallenberg Foundation grant KAW 2015.0270.
Any correspondence should be addressed to MD}
\date{}

\begin{document}
	\maketitle

\begin{abstract}
We study the limiting behavior of random lozenge tilings of the hexagon with a $q$-Racah weight as the size of the hexagon grows large. Based on the asymptotic behavior of the recurrence coefficients of the $q$-Racah polynomials, we give a new proof for the fact that the height function for a random tiling concentrates near a deterministic limit shape and that the global fluctuations are described by the Gaussian Free Field. These results were recently proved using (dynamic) loop equation techniques. In this paper, we extend the recurrence coefficient approach that was developed for (dynamic) orthogonal polynomial ensembles to the setting of $q$-orthogonal polynomials. An interesting feature is that the complex structure is easily found from the limiting behavior of the (explicitly known) recurrence coefficients. A particular motivation for studying this model is that the variational characterization of the limiting height function has an inhomogeneous term. The study of the regularity properties of the minimizer for general variational problems with such inhomogeneous terms is a challenging open problem. In a general setup, we show that the variational problem gives rise to a natural complex structure associated with the same Beltrami equation as in the homogeneous situation. We also derive a relation between the complex structure and the complex slope.  In the case of the $q$-Racah weighting of lozenge tilings of the hexagon, our representation of the limit shape and their fluctuations in terms of the recurrence coefficients allows us to verify this relation explicitly. 
	\end{abstract}

  \tableofcontents
	
\section{Introduction}


Random tilings of planar domains form a rich class of models that have been studied intensively in the past decades. Different techniques have been developed to study their asymptotic behavior when the domain grows large (or the mesh becomes small). We refer to \cite{Vadim} for a recent treatise on the subject with many references. In certain special examples, remarkable connections with orthogonal polynomials have been found. For example, the correlation function for uniformly distributed lozenge tilings of the hexagon can be expressed in terms of  Hahn polynomials \cite{GorinHex,JohHex}. Uniformly distributed lozenge tilings of the Aztec diamond are related to Krawtchouk polynomials \cite{JohAztec}.   This connection bridges important models from statistical mechanics to classical objects of analysis. It also sets the stage for proving universality conjectures for these stochastic models, as these typically boil down to an asymptotic analysis  (as the degree tends to infinity) for the orthogonal polynomials and their properties. An important example for this paper is the method developed in  \cite{breuer2017central,duits2018global}  for proving Gaussian Free Field fluctuations via the recurrence coefficients of the polynomials. The purpose of this writing is to extend this approach to more general classes of models that include $q$-orthogonal polynomials. The running example will be the probability measure on lozenge tilings of the hexagon introduced in \cite{borodin2010q} related to $q$-Racah polynomials, which we will first discuss.

 
\subsubsection*{Lozenge tilings of a hegaxon and $q$-Racah polynomials}

Let $\mathbf{a},\mathbf{b},\mathbf{c}\in\mathbb{N}$ and consider the hexagon  with corners $(0,0), (\mathbf{b},0), (\mathbf{b}+\mathbf{c},\mathbf{c}), (\mathbf{b}+\mathbf{c},\mathbf{a}+\mathbf{c}), (\mathbf{c},\mathbf{a}+\mathbf{c}), (0,\mathbf{a})$ on the two-dimensional $(\unscaledtime,y)$ plane. We will be interested in tiling the hexagon with the following three types of lozenges: 
	
	\begin{center}
		Type  $I$ \begin{tikzpicture}[scale=0.4]
			\draw (0,0)--++(1,1)--++(0,1)--++(-1,-1)--++(0,-1);
		\end{tikzpicture} \quad Type  $II$ \begin{tikzpicture}[scale=0.4]
			\draw (0,0) --++(1,0)--++(0,1)--++(-1,0)--++(0,-1);
		\end{tikzpicture} \quad Type  $III$ \begin{tikzpicture}[scale=0.4]
			\draw (0,0) --++(1,1)--++(1,0)--++(-1,-1)--++(-1,0);
			
		\end{tikzpicture} .
  \label{fig:lozenges}
	\end{center}
	A tiling means a covering of the hexagon with lozenges so that no two tiles overlap. See Figure~\ref{fig:sample_color} for an example. There are many different ways of tiling the hexagon. We will choose one at random as follows: Let $\mathcal T$ be the space of all possible tilings, then we define a probability measure on $\mathcal T$ by saying that the probability of having a tiling $T\in \mathcal T$ is proportional to the product of the weights of the Type III lozenges in $T$,
	$$
	\mathbb P(T)\sim \prod_{ \tikz[scale=.3] { \draw (0,0) \lozd;} \in T} \mathfrak{w} ( \tikz[scale=.3] { \draw (0,0) \lozd;})
	$$
	where the weight is given by 
	\begin{equation} \label{eq:ellipticloz}
		\mathfrak{w}\left( \tikz[scale=.3] { \draw (0,0) \lozd; \filldraw (1,1) circle (3pt); 
			\draw (1,1) node[anchor=east]{$(\unscaledtime,y)$}; ;}\right)=\kappa \mathbf{q}^{\frac{-\unscaledtime-\mathbf{c}}{2}+y+\frac{1}{2}}-\frac{1}{\kappa \mathbf{q}^{\frac{-\unscaledtime-\mathbf{c}}{2}+y+\frac{1}{2}}},
	\end{equation}
	where $(\unscaledtime,y)$ are the coordinates of the  top left corner of \tikz[scale=.3] { \draw (0,0) \lozd; \filldraw (1,1) circle (3pt); 
			\draw (1,1) node[anchor=east]{$(\unscaledtime,y)$};;}.  This weight function was introduced by Borodin, Gorin and Rains in \cite{borodin2010q} and is a special case of a more general elliptic weight. Note that to obtain a proper probability measure, one needs conditions on the parameters $\kappa$ and ${\bf q}$ that ensure that $\mathbb P(T)\geq 0$ for all tilings $T\in \mathcal T$. This does not necessarily mean that $\mathfrak{w}$ has to take positive values. Since the total number of type III lozenges is fixed and determined by the dimensions of the hexagon, the weights \eqref{eq:ellipticloz} can be multiplied by a universal constant  (which can be factored out and absorbed into the normalizing constant).  We thus need only to require all weights to have the same sign. In \cite{borodin2010q}, the authors discuss the following three options: (a)  $\kappa \in \mathbb{R}$ with some restrictions, (b) both $\kappa$ and $\mathbf q$ are complex with modulus $1$ and (c) $\kappa \in i \mathbb R$.   Note that if ${\mathbf q}=1$, then the weight \eqref{eq:ellipticloz} is independent of the position $x$, and the probability measure is the uniform measure on all possible tilings of the hexagon.

The weight \eqref{eq:ellipticloz} is integrable in the sense that it can be solved in terms of $q$-Racah polynomials, as was shown in \cite{borodin2010q}. We start by drawing points on the part of the lattice $\mathbb Z \times (\mathbb Z +\tfrac12)$ such that there is a black point at all the lattice points that are covered by a type III lozenge and white points otherwise. See Figure~\ref{fig:sample_color} for an example. Then denote the coordinates of the white dots by $(\unscaledtime,y_j(\unscaledtime)+\frac12)$ for $j=1,2, \ldots,\mathbf{a}$ and $\unscaledtime \in \{1,\ldots,\mathbf{b}+\mathbf{c}-1\}$ (note that for each $\unscaledtime \in \{1,\ldots,\mathbf{b}+\mathbf{c}-1\}$ there are exactly $\mathbf{a}$ white points with horizontal coordinate $t$).  It was proved in \cite{borodin2010q} the $y_j(\unscaledtime)$ for $j=1,2,\ldots,\mathbf{a}$ have the joint probability distribution function \begin{equation} \label{eq:qOPErac}
		\sim \prod_{1 \leq i < j \leq \mathbf{a}}(\nu(y_i)-\nu(y_j))^2 w(y_j), 
	\end{equation}
 with $\nu$ and $w$ given by 
 \begin{equation}\label{eq:qRac}
	\nu(y) = \mathbf{q}^{-y}+ \gamma \delta \mathbf{q}^{y+1} \qquad \text{and} \qquad w(y)= \frac{\left(\alpha \mathbf{q},\beta\delta \mathbf{q},\gamma \mathbf{q},\gamma\delta \mathbf{q};\mathbf{q}\right)_y\left(1-\gamma\delta \mathbf{q}^{2y+1}\right)}{\left(\mathbf{q} ,\alpha^{-1}\gamma\delta \mathbf{q},\beta^{-1}\gamma \mathbf{q},\delta \mathbf{q};\mathbf{q}\right)_y\left(\alpha\beta \mathbf{q}\right)^y\left(1-\gamma\delta \mathbf{q}\right)}
\end{equation}
 The precise choice of parameters  on $n,\alpha, \beta,\gamma$ and $\delta$ are expressed in terms of  ${\bf a},{\bf b},{\bf c},t,{\bf q}$ and $\kappa$ and will be discussed in detail in Section~\ref{sec:tiling}.  The weight $w$ in \eqref{eq:qRac} is exactly the orthogonality weight for the $q$-Racah polynomials, and the correlation functions for the point process can be expressed in terms of these polynomials \cite{borodin2010q}.  Also, in this writing, the connection with $q$-Racah polynomials will be central, although we will not use the correlation kernel. In \cite{borodin2010q} the authors also computed the joint probability distribution function  for 
 $\{y_j(\unscaledtime)\}_{j=1,s=1}^{{\bf a},{\bf b+\bf c-}1}$. Since its description is rather technical, we will postpone the details to Section~\ref{sec:tiling}


 \subsubsection*{Large hexagons}

  A tiling of the hexagon by lozenges can be viewed three-dimensional as cubes (or boxes) stacked in the corner of a room. This also allows us to view a random tiling as a random surface. We will be interested in the asymptotic behavior of this random surface when the size of the hexagon grows large. More precisely, for $b,c\in\mathbb{R_+},n\in\mathbb{N}_+$ and $q>0$ or $q$ on the complex unit circle, we set 
 \begin{align*}
     \mathbf{a}=n, \quad \mathbf{b}=nb, \quad \mathbf{c}=nc, \quad  \mathbf{q} = q^{\frac{1}{n}}. 
 \end{align*}	
 and study the limit $n \to \infty$. Note that this means that we take $\mathbf{q}\to 1$ simultaneously as the size of the hexagon goes to infinity.  
 In Figure~\ref{fig:example} (also later in Figure~\ref{fig:morexamples}), we have plotted a sample of a random lozenge tiling for some special choice of parameters. The Figure(s) illustrates a common picture in tiling models: the randomness is contained in a region called the liquid region, but the corners are frozen. Moreover, a limit shape seems to appear for the random surface. 

 To make this more concrete, we define the height function $h_n$ as 
 \begin{equation} \label{eq:heightfunction}
     h_n(t,x)= \#\left\{ j \mid y_j(\lfloor nt\rfloor) \leq  nx \right\},
 \end{equation}
 where $t$ and $x$ are rescaled variables inside the rescaled hexagon.  In other words, $h_n$ counts the number of the white points below and including the point $(\lfloor nt\rfloor,nx)$ in Figure~\ref{fig:sample_color}. Then, as $n \to \infty$, we have 
\begin{equation} \label{eq:limitshapeintro}
    \frac{1}{n} h_n(t,x) \to h(t,x),
\end{equation}
almost surely. The graph of $h(t,x)$ is the limit shape near which the random tilings concentrate.  The limit shape $h$ was rigorously computed using loop equations in \cite{dimitrov2019log}. Moreover, for fixed $t$, they also proved the fluctuations of $h_n(t,\cdot)-\mathbb E h_n(t,\cdot)$ along vertical one-dimensional sections. Recently, the full two-dimensional fluctuation field was proved to be described by the Gaussian Free Field (GFF) \cite{gorin2022dynamical} using time-dependent extensions of loop equation techniques. More precisely, after a change of coordinate, they prove that $h_n-\mathbb E h_n$ converges to the Gaussian Free Field in the upper half-plane.  
  
  One of the main motivations of the current writing is to show that we can derive the limit shape and the Gaussian Free  Field fluctuations by extending established principles for orthogonal polynomials in \cite{breuer2017central,duits2018global,kuijlaars1999asymptotic} to the $q$-orthogonal setting.    More precisely, we will show that the limit shape and its fluctuations can be computed from the recurrence coefficients for the $q$-Racah polynomials. In our opinion, this is a rather direct way of studying the particular tiling model that fully exploits its integrable structure. Moreover, the conformal structure, particularly the diffeomorphism describing the pullback in the Gaussian Free Field for the height fluctuations, admits a very simple expression that can simply be read off from the asymptotic behavior of the recurrence coefficients.


  \subsubsection*{Linear statistics and recurrence coefficients}

The probability measure in \eqref{eq:qOPErac} together with \eqref{eq:qRac} is a special case of a more general class of models, and we will prove various results on their asymptotic behavior based on general principles.  Indeed, we study point processes on $\mathbb Z$ where one considers $n \in \mathbb N$ random points $y_1,\ldots,y_{n}$ with a joint probability distribution function proportional to \begin{equation} \label{eq:qOPE}
		\sim \prod_{1 \leq i < j \leq n}(\nu(y_i)-\nu(y_j))^2 w(y_j), 
	\end{equation}
	where  $\nu: \mathbb Z \to \mathbb R$ is strictly monotone and $w: \mathbb Z \to [0,\infty)$ such that 
	\begin{equation} \label{eq:finitemoments}
		\sum_{y \in \mathbb Z}|\nu(y) |^k w(y)< \infty,
	\end{equation}  
	for all $k \in \mathbb N$.   If $w$ has a finite support, then \eqref{eq:finitemoments} is redundant. But in the case of infinite support, it ensures that all moments of $w$ exist and, in particular, \eqref{eq:qOPE} can be normalized to a probability measure. Also, if $w$ has a finite support, we need at least $n$ points in that support for \eqref{eq:qOPE} to be a well-defined probability measure. We will also study time-dependent extensions of \eqref{eq:qOPE}, but as their definition is a bit more technical we postpone the discussion to Section 2.  
 
 In  case $\nu(y)=y$, \eqref{eq:qOPE} is the discrete orthogonal polynomial ensemble \cite{konig} associated with $w$.  Such ensembles appear naturally in various contexts, and their name is explained by the fact that various quantities, such as the correlation functions, can be expressed in terms of the orthogonal polynomials with respect to the weight $w$. The process defined by \eqref{eq:qOPE} for general $\nu$  can also be analyzed using extensions of these orthogonal polynomials, such as  $q$-analogues.   Let $p_k$ be the unique polynomial of degree $k$ and positive leading coefficient such that 
	\begin{align}\label{eq:orthogonality}
		\sum_{y\in \mathbb Z}p_l(\nu(y))p_k(\nu(y))w(y)=\delta_{l,k},
	\end{align}
	for $l,k=0,1,2\ldots$. We also set  $r_k(y)=p_k(\nu(y))$, which is a polynomial in $\nu$. For $\nu(y)=y$, the polynomials $r_k=p_k$ are the  orthogonal polynomials with respect to $w$. For $\nu$ and $w$ as in \eqref{eq:qRac}, the polynomials are the $q$-Racah polynomials.  From the orthogonality relation one can show that the  polynomials $r_k(y)$ satisfy a three-term recurrence. That is, there exists two sequences $\{a_j\}_{j=0}^\infty$ and $\{b_j\}_{j=0}^\infty$ with $a_j>0$ and $b_j\in \mathbb R$ such that  
	\begin{align}\label{recurrence}
		\nu(y)r_j(y)=a_{j}r_{j+1}(y)+b_{j}r_j(y)+a_{j-1}r_{j-1}(y), 
	\end{align}
	for $j=0,1,2,\ldots$ (where we also set $a_{-1}=0$ for convenience).   For $\nu(y)=y$, this is the classical three-term recurrence for orthogonal polynomials on the real line, and it is not difficult to see (and well-known) that this extends to the case of general $\nu(y)$. The coefficients $a_j$ and $b_j$ are  called the recurrence coefficients and are important objects as they determine the polynomials. These coefficients are explicitly known in the case of classical polynomials, such as the $q$-Racah polynomials. As we will see below, they also play an important central role for the associated point process \eqref{eq:qOPE}.

	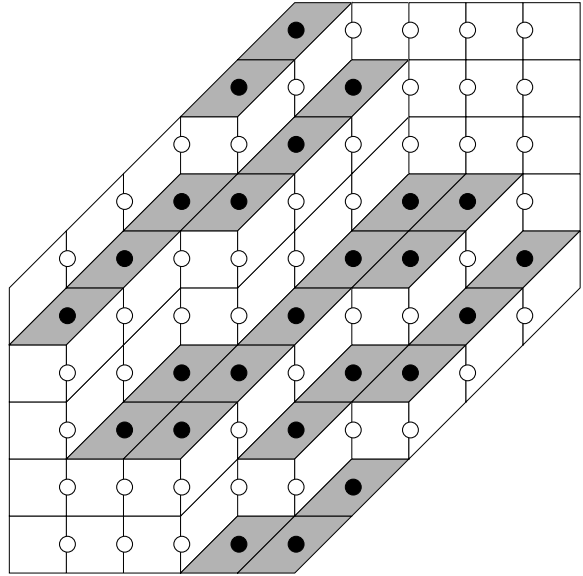
\begin{figure}
		\centering
		\begin{tikzpicture}[scale=1]
			\tikzmath{\x=.76;}
			\node[inner sep=0pt] (small) at (-.48,0)
			{\includegraphics[width=.798\textwidth]{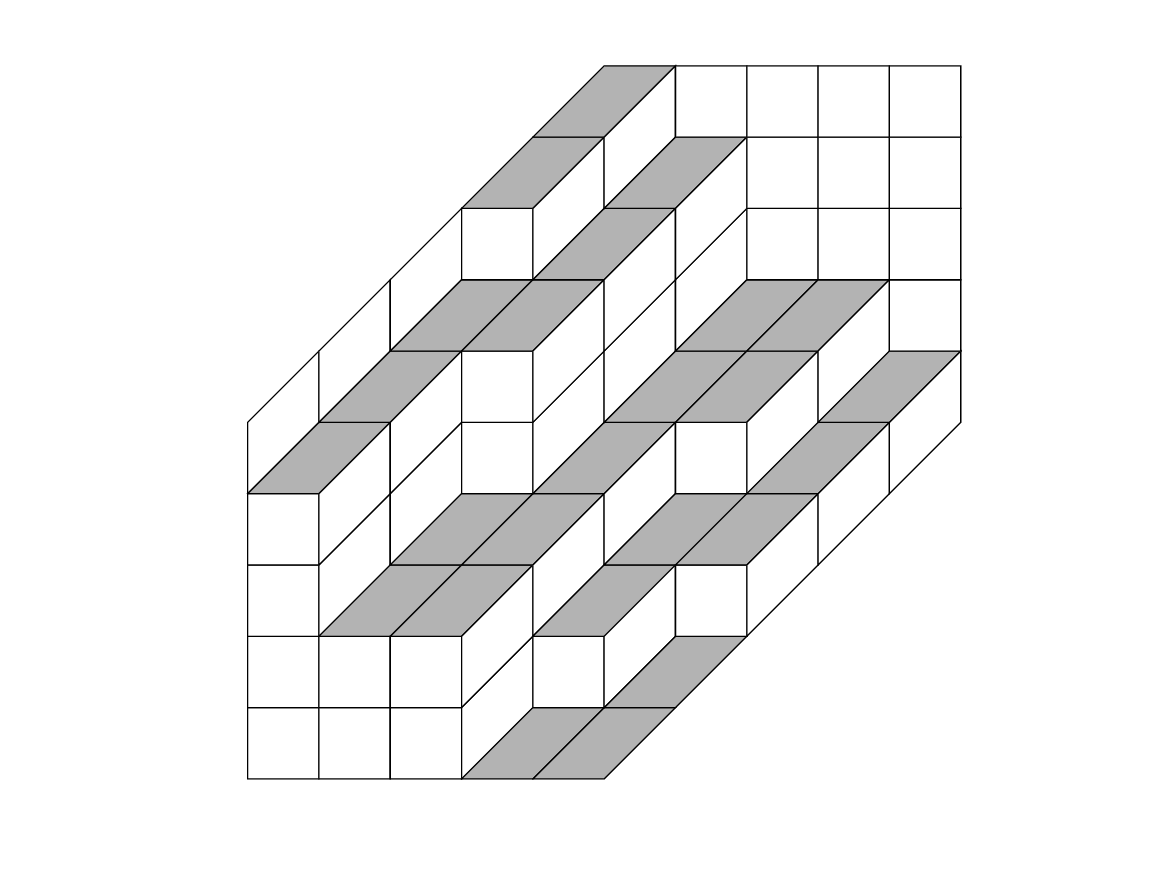}};
			\draw[fill=white] (-1,-3.25+\x) circle (3pt);
			\draw[fill=white]  (-1,-3.25+2*\x) circle (3pt);
			\draw[fill=white]  (-1,-3.25+4*\x) circle (3pt);
			\draw[fill=white]  (-1,-3.25+5*\x) circle (3pt);
			\draw[fill=white]  (-1,-3.25+7*\x) circle (3pt);
			\filldraw (-1,-3.25+0*\x) circle (3pt);
			\filldraw (-1,-3.25+3*\x) circle (3pt);
			\filldraw (-1,-3.25+6*\x) circle (3pt);
			\filldraw (-1,-3.25+8*\x) circle (3pt);

			\draw[fill=white] (-1-\x,-3.25+0*\x) circle (3pt);
			\draw[fill=white]  (-1-\x,-3.25+1*\x) circle (3pt);
			\draw[fill=white]  (-1-\x,-3.25+4*\x) circle (3pt);
			\draw[fill=white]  (-1-\x,-3.25+5*\x) circle (3pt);
			\draw[fill=white]  (-1-\x,-3.25+7*\x) circle (3pt);
			\filldraw (-1-\x,-3.25+2*\x) circle (3pt);
			\filldraw (-1-\x,-3.25+3*\x) circle (3pt);
			\filldraw (-1-\x,-3.25+6*\x) circle (3pt);
			
			\draw[fill=white] (-1-2*\x,-3.25+0*\x) circle (3pt);
			\draw[fill=white]  (-1-2*\x,-3.25+1*\x) circle (3pt);
			\draw[fill=white]  (-1-2*\x,-3.25+3*\x) circle (3pt);
			\draw[fill=white]  (-1-2*\x,-3.25+4*\x) circle (3pt);
			\draw[fill=white]  (-1-2*\x,-3.25+6*\x) circle (3pt);
			\filldraw (-1-2*\x,-3.25+2*\x) circle (3pt);
			\filldraw (-1-2*\x,-3.25+5*\x) circle (3pt);
			
			\draw[fill=white] (-1-3*\x,-3.25+0*\x) circle (3pt);
			\draw[fill=white]  (-1-3*\x,-3.25+1*\x) circle (3pt);
			\draw[fill=white]  (-1-3*\x,-3.25+2*\x) circle (3pt);
			\draw[fill=white]  (-1-3*\x,-3.25+3*\x) circle (3pt);
			\draw[fill=white]  (-1-3*\x,-3.25+5*\x) circle (3pt);
			\filldraw (-1-3*\x,-3.25+4*\x) circle (3pt);
			

			\draw[fill=white] (-1+\x,-3.25+1*\x) circle (3pt);
			\draw[fill=white]  (-1+\x,-3.25+3*\x) circle (3pt);
			\draw[fill=white]  (-1+\x,-3.25+5*\x) circle (3pt);
			\draw[fill=white]  (-1+\x,-3.25+6*\x) circle (3pt);
			\draw[fill=white]  (-1+\x,-3.25+8*\x) circle (3pt);
			\filldraw (-1+\x,-3.25+0*\x) circle (3pt);
			\filldraw (-1+\x,-3.25+2*\x) circle (3pt);
			\filldraw (-1+\x,-3.25+4*\x) circle (3pt);
			\filldraw (-1+\x,-3.25+7*\x) circle (3pt);
			\filldraw (-1+\x,-3.25+9*\x) circle (3pt);

			\draw[fill=white] (-1+2*\x,-3.25+2*\x) circle (3pt);
			\draw[fill=white]  (-1+2*\x,-3.25+4*\x) circle (3pt);
			\draw[fill=white]  (-1+2*\x,-3.25+6*\x) circle (3pt);
			\draw[fill=white]  (-1+2*\x,-3.25+7*\x) circle (3pt);
			\draw[fill=white]  (-1+2*\x,-3.25+9*\x) circle (3pt);
			\filldraw (-1+2*\x,-3.25+1*\x) circle (3pt);
			\filldraw (-1+2*\x,-3.25+3*\x) circle (3pt);
			\filldraw (-1+2*\x,-3.25+5*\x) circle (3pt);
			\filldraw (-1+2*\x,-3.25+8*\x) circle (3pt);
			
			\draw[fill=white] (-1+3*\x,-3.25+2*\x) circle (3pt);
			\draw[fill=white]  (-1+3*\x,-3.25+4*\x) circle (3pt);
			\draw[fill=white]  (-1+3*\x,-3.25+8*\x) circle (3pt);
			\draw[fill=white]  (-1+3*\x,-3.25+7*\x) circle (3pt);
			\draw[fill=white]  (-1+3*\x,-3.25+9*\x) circle (3pt);
			
			\filldraw (-1+3*\x,-3.25+3*\x) circle (3pt);
			\filldraw (-1+3*\x,-3.25+5*\x) circle (3pt);
			\filldraw (-1+3*\x,-3.25+6*\x) circle (3pt);

			\draw[fill=white] (-1+4*\x,-3.25+3*\x) circle (3pt);
			\draw[fill=white]  (-1+4*\x,-3.25+5*\x) circle (3pt);
			\draw[fill=white]  (-1+4*\x,-3.25+8*\x) circle (3pt);
			\draw[fill=white]  (-1+4*\x,-3.25+7*\x) circle (3pt);
			\draw[fill=white]  (-1+4*\x,-3.25+9*\x) circle (3pt);
			
			\filldraw (-1+4*\x,-3.25+4*\x) circle (3pt);
			\filldraw (-1+4*\x,-3.25+6*\x) circle (3pt);
			
			\draw[fill=white] (-1+5*\x,-3.25+4*\x) circle (3pt);
			\draw[fill=white]  (-1+5*\x,-3.25+6*\x) circle (3pt);
			\draw[fill=white]  (-1+5*\x,-3.25+8*\x) circle (3pt);
			\draw[fill=white]  (-1+5*\x,-3.25+7*\x) circle (3pt);
			\draw[fill=white]  (-1+5*\x,-3.25+9*\x) circle (3pt);
			
			\filldraw (-1+5*\x,-3.25+5*\x) circle (3pt);

		\end{tikzpicture}
		\caption{Example of a tiling of the hexagon with $\mathbf{a}=\mathbf{b}=\mathbf{c}=5$.}
		\label{fig:sample_color}
	\end{figure}
	
We will study the asymptotic behavior of orthogonal polynomials ensemble as $n\to \infty$ using   linear statistics with sufficiently smooth test functions for point processes of the form \eqref{eq:qOPE} and time-dependent generalizations. More precisely, for $y_j$ randomly from \eqref{eq:qOPE}  and a function $f$ we will consider the random variable
 $$
 X_n(f)=\sum_{j=1}^{n} f(y_j/n),
 $$
 where we anticipate the necessary rescaling with $n$ as $n \to \infty$ (recall that we need the rescale the mesh for the lozenge tilings of the hexagon above).  While taking this limit, we will also allow $\nu=\nu_n$
 and $w=w_n$ to vary with $n$. Note that this also implies that the recurrence coefficients depend on $n$ and we will indicate this by using the notation $a_{j,n}$ and $b_{j,n}$.
 Following the reasoning in \cite{breuer2017central,duits2018global}, we will show (see Section 2 and 4) that the cumulants for these linear statistics can be expressed in terms of the recurrence coefficients and this will be the starting point of our analysis.

The setup above is too general to prove useful results, and we will need to make further assumptions. For the loop equation approach \cite{dimitrov2019log,gorin2022dynamical}, one imposes certain regularity conditions, and the typical assumption is that $w_n$ converges to a sufficiently smooth function as $n \to \infty$. Here, we will instead make natural assumptions on the recurrence coefficients $a_{j,n}$ and $b_{j,n}$.  Roughly speaking, we will assume there exist continuous functions $a$ and $ b$ such that,
\begin{equation}\label{eq:roughlimits}
	\begin{cases}
		a_{j,n}=a(j/n)+o(1),\\
		b_{j,n}=b(j/n)+o(1)\
	\end{cases}
\end{equation}
as $n \to \infty$, and limiting function $\mu(x)$ such that
\begin{equation} 
\label{eq:roughlimitsmu}
    \nu_n(nx) \to \mu(x).
\end{equation}
The conditions \eqref{eq:roughlimits} and \eqref{eq:roughlimitsmu} are natural. Condition \eqref{eq:roughlimits}  means that, since $n$ is large, the recurrence coefficients are slowly varying. Many examples of orthogonal polynomials satisfy \eqref{eq:roughlimits} and \eqref{eq:roughlimitsmu}, including the classical families of orthogonal polynomials. We refer to \cite{breuer2017central, kuijlaars1999asymptotic} for extensive discussions and examples. In the context of the present paper, it is interesting to mention that the $q$-analogues of the classical orthogonal polynomials (including the ${q}$-Racah) also satisfy \eqref{eq:roughlimits}, but typically only when we take ${\bf q}\to 1$ simultaneously as $n \to \infty$. For fixed $0<{\bf q}<1$, the conditions will not hold and the analysis in this paper does not apply.  Note that from the running example this is not surprising, since for $0<{\bf q}<1$ the weighting \eqref{eq:ellipticloz} is not expected to have Gaussian Free Field fluctuations.

The conditions \eqref{eq:roughlimits} and \eqref{eq:roughlimitsmu}, are known to have important implications for the asymptotic behavior of the polynomials \cite{kuijlaars1999asymptotic}. In this paper, we will derive both the almost sure convergence of the linear statistic $X_n(f)$  and a Central Limit Theorem for the corresponding point process. Section~\ref{sec:generalOPE} states three theorems concerning the almost sure limit of $\frac 1n X_n(f)$ as $n \to \infty$. While postponing the precise details to that section, we mention that, in case the function $a$ is strictly positive, we find  that
 \begin{align} \label{eq:roughresult1}
		\frac 1n X_n(f) \to   \frac{1}{\pi }\iint\frac{f(x)  \ d\xi \  d\mu(x)}{\sqrt{4a(\xi)^2-(\mu(x)-b(\xi))^2}}, 
	\end{align} 
 as $n \to \infty$. Although not relevant in our context, we mention that the measure on the right-hand side is also the limiting zero distribution of the $n$-th orthogonal polynomial as shown in \cite{kuijlaars1999asymptotic} (strictly speaking, the result was only proved for $\nu_n(y)=y$ under condition \eqref{eq:roughlimits}, but the authors also discuss the extension to general $\nu_n$ in an example).  
  For the case $\nu_n(y)=y$, the limiting density of zeros is known to be the same as the almost sure limit of the empirical measure of the random point configuration from \eqref{eq:qOPE}, for example, \cite{hardy2018polynomial}. It is not hard to generalize this to the general case.  In this paper, we will give a direct proof of \eqref{eq:roughresult1} that does not use the zeros of the polynomials. To the best of our knowledge,  this proof has not appeared in the literature before. We will also expand the scope of the test functions $f$ that are allowed.

 In Section~\ref{sec:generalOPE}, we also give Central Limit Theorems for $X_n(f)$ by following the strategy in  \cite{breuer2017central}. The result is that 
 $$
 X_n(f) -\mathbb E X_n(f) \to N(0,\sigma_f^2)
 $$
 where $\sigma_f^2$ is an explicit expression (cf. \eqref{eq:limitingvariance1} and \eqref{eq:limitingvariance2}) involving $a(1), b(1)$ and $\mu$.  Interestingly, the only condition that we need is that \eqref{eq:roughlimits} 
 holds for $j/n \to 1$ as $n \to \infty$. We do, however, need a Lipschitz condition for the convergence in \eqref{eq:roughlimitsmu}. In Theorem~\ref{thm:fluctuation_theorem} we also present the extension of this Central Limit Theorem to dynamic extensions of point processes of the type \eqref{eq:qOPE} following \cite{duits2018global}. 
	
	In Section~\ref{sec:tiling}, we return to the case of lozenge tilings of the hexagon with weight \eqref{eq:ellipticloz}. We show that the conditions for Theorems~\ref{thm:LLNlpfunctions} and \ref{thm:fluctuation_theorem} are easily verified by looking up the recurrence coefficients for $q$-Racah polynomials, which are explicit.  From there, we compute the limit shape and show that the fluctuations of the height function are given by the Gaussian Free Field. We stress that the diffeomorphism that maps the rough disordered region to a strip in the complex plane has a very simple expression in terms of the limits of the recurrence coefficients.


 \subsubsection*{Variational problem}
	
The  weight \eqref{eq:ellipticloz}  
falls outside the class of uniform or periodic weights but still has an elegant integrable structure. This makes it an attractive model to test extensions of established results for the uniform or periodic weights. For instance, for uniformly distributed random tilings, and more generally for tilings with periodic weight structure, there have been important works on characterizing the limit shape in terms of a variational problem, in particular \cite{CKP, CLP,DeS10, KeOk07,KOS,Ku}. Furthermore, the regularity of the minimizers for lozenge tilings, as well as general dimer models with periodic weight structure, has been classified in \cite{ADPZ}. In addition, \cite{ADPZ} shows that the \emph{liquid regions} are endowed with a unique complex structure. It is a conjecture of Kenyon and Okounkov that the global fluctuations around the limit shape in the liquid region are governed by the Gaussian Free Field (which is proved for the uniform distribution on the hexagon in  \cite{duits2018global,petrov2015asymptotics}).  However, to observe the Gaussian Free Field, one must first identify the correct complex structure and make a change of coordinates by solving a Beltrami equation, typically pulling back the liquid region by a diffeomorphism to the unit disc in the case when the liquid region is a simply connected domain. 
	
	By including the parameters $q$ and $\kappa$, one breaks the uniformity, and the complexity increases significantly. As conjectured in \cite{borodin2010q}, the limit shape is still characterized by a variational problem, and this is proven in \cite{MPT22Var}. The difference with the uniform measure is that the energy function now includes an additional linear term. A careful treatment of such a variational problem and a study of the regularity of the minimizer is an open problem that is significantly more complicated than the one for the uniform measure. We will discuss a general setup for this in Section~\ref{sec:variational} and will extend certain results in \cite{ADPZ}. In particular, in a joint work with Astala, we show how the inhomogeneous Euler-Lagrange equation can be reduced to a first-order quasilinear inhomogeneous Beltrami equation. We will also identify the naturally associated complex structure and change of variables associated with the variational problem. In addition, we show how the complex structure can be expressed by the complex slope solving a complex Burger equation, connecting the complex structure to the results in \cite{dimitrov2019log,gorin2022dynamical}. That the complex structure associated with the variational problem is the same as the one giving the Gaussian Free Field for the regular hexagon is then verified in Section~\ref{sec:tiling}, and we conjecture that they are always the same. 
 
 In greater detail, in Section~\ref{sec:tiling}, we return to the variational problem for weight \eqref{eq:ellipticloz} for the hexagon. We recall that the diffeomorphism that gives the correct coordinates for the Gaussian Free Field fluctuation is a very simple function of the limiting behavior of the recurrence coefficients. The limit shape is computed using   \eqref{eq:roughresult1}, but from the principles for the variational problem, it should also be computable directly from the diffeomorphism giving the Gaussian Free Field. We verify that both ways indeed provide the same result (as it should). Note that this is somewhat remarkable. Indeed, the diffeomorphism only depends on the limiting behavior of the recurrence coefficient $a_{j,n}$ with $j/n \to 1$ as $n \to \infty$, where the formula \eqref{eq:roughresult1} depends on all $a_{j,n}$ with $j/n \to \xi \in (0,1)$. We believe that the reason for this is that the recurrence coefficients for the $q$-Racah polynomials themselves satisfy difference equations (in both degree and time). It would be very interesting to see if these difference equations can be used to derive the variational problem.

	\begin{figure}[t]
		\begin{center}
			\includegraphics[scale=0.25]{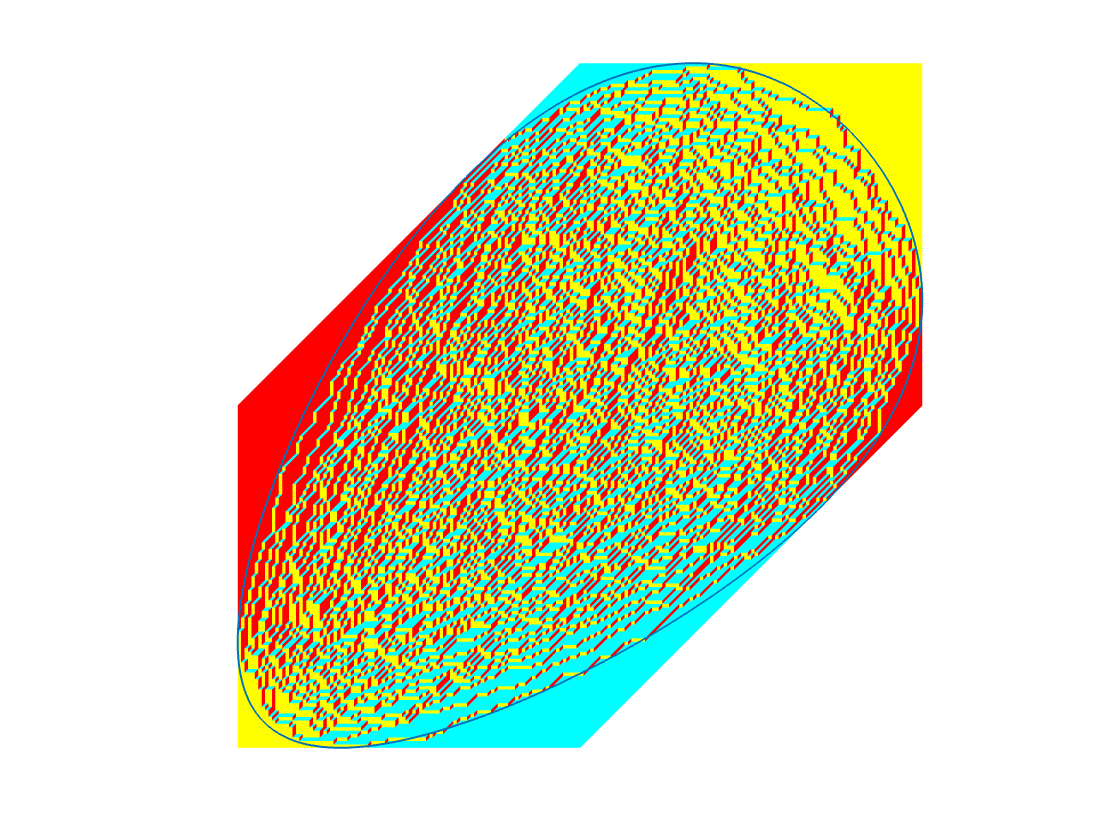}
		\end{center}
		\caption{An example of a lozenge tiling with $q$-Racah weight with parameters $b=1$, $c=1$, $q=0.5, \kappa=2.9i$ and size $n=100$.}
		\label{fig:example}
	\end{figure}


	\subsubsection*{Overview of the paper}


In Section~\ref{sec:generalOPE}, we state our main results for the asymptotic behavior of linear statistics for point processes of the type \eqref{eq:qOPE} and extensions.  A Law of Large Numbers is formulated in Theorems~\ref{limit_density_poly}, \ref{thm:LLNcontinuousfunctions} and \ref{thm:LLNlpfunctions}. Then, we state a CLT for the linear statistics in Theorem~\ref{CLT_q_polynomial} and a time-dependent extension in Theorem~\ref{thm:fluctuation_theorem}. It is imperative to underscore that the asymptotic behavior of the recurrence coefficients inherently drives all these theorems.  The proofs of these theorems are presented in Section~\ref{sec:proofsGen}, but we discuss the example of the $q$-Racah ensembles first in Section~\ref{sec:q-Racah}. Our results include the real, imaginary, and trigonometric cases. The first two cases are dealt with simultaneously in Theorem~\ref{limit_shape_qracah}. The trigonometric cases require some special attention, and the limit density is computed separately in Theorem~\ref{limit_shape_qracah tri}. Then, in Section~\ref{sec:variational}, we discuss the variational problem for the tilings of the hexagon and propose a road map for studying the regularity properties. Finally,  in Section~\ref{sec:tiling}, we compute the limiting behavior of the height function for the lozenge tilings of the hexagon with weight \eqref{eq:ellipticloz}. In particular, we compute the limiting height function and the Gaussian Free Field fluctuations. Then, we verify the relation between the complex slope and complex structure as predicted in Section~\ref{sec:variational}.


	\section{Limit shapes and fluctuations for $q$-Orthogonal Polynomial Ensembles}\label{sec:generalOPE}
	
	In this section, we discuss general  $q$-orthogonal polynomial ensembles and present our main results on the limiting behavior of their linear statistics.

	
	\subsection{Scaled point process and large $n$ behavior}
	
We will be interested in the large $n$ behavior of the point process defined by \eqref{eq:qOPE}. To achieve a meaningful limit, we will allow the parameters in the model to depend on $n$. The assumptions we will pose imply that the point process will essentially have a global scale of order $\sim n$ and it does make sense to rescale the process by $n$. That is, set $y=xn$ and consider the point process on $\frac 1n \mathbb Z$ with the probability distribution function proportional to 
 \begin{align}\label{bi-ope}
		\sim \prod_{1\leq i<j\leq n}\left(\nu_n(nx_i)-\nu_n(nx_j)\right)^2 w_n( n x_1)\dots w_n(n x_n).
	\end{align}
	 Note that we allow that $\nu=\nu_n$ and $w=w_n$ vary with $n$ as $n \to \infty$, which also implies that $a_j=a_{j,n}$ and $b_j=b_{j,n}$ both depend on $n$. In fact, after setting 
	$$
	\mu_n(x)=\nu_n(nx), \qquad x \in \frac{1}{n}\mathbb{Z}
	$$
	we will assume that there exists a function $\mu$ on $\mathbb{R}$ such that
	\begin{equation} \label{eq:limitmun}
		\lim_{n \to \infty} \mu_n(x)=\mu(x).
	\end{equation}
The precise mode of convergence in \eqref{eq:limitmun} will be specified in our main theorems. In a nutshell, for the limiting behavior of the empirical measure, we will require uniform convergence, but for the fluctuations, we will need a stronger Lipschitz condition.
	
 We study the large $n$ limit of the point process \eqref{bi-ope} by analyzing the asymptotic behavior of the linear statistics

 \begin{equation} \label{linearstat}
	X_n(f)=\sum_{j=1}^{n} f(x_j),
	\end{equation}
	for functions $f\colon \mathbb R \to \mathbb R$.


 \subsection{Limiting distribution of points}

	The first result we discuss is the fact that the linear statistic has an almost sure limit, implying that there is a deterministic density describing the limiting distribution of points as $n \to \infty$.

	\begin{theorem}\label{limit_density_poly}
		Consider a discrete point process $\{x_j\}_{j=0}^{n-1}$  with its joint distribution proportional to \eqref{bi-ope}.  
		Suppose there exists real-valued continuous functions,  $a(\xi), b(\xi)$, on $[0,1]$, where $a\geq 0$ and $a\not\equiv 0$, such that the recurrence coefficients \eqref{recurrence} associated with the orthonormal polynomials satisfy 
		\begin{align} \label{ass:slowly_varying_recurrence}
			\lim_{\frac{j}{n}\to\xi}a_{j,n}=a(\xi), \quad \lim_{\frac{j}{n}\to\xi}b_{j,n}=b(\xi), \quad \forall \xi\in[0,1],
		\end{align} 
  where the limit $n \to \infty$ is taken such that such that $j/n \to \xi$ simultaneously. 
		Then, for any polynomial $p$ we have that the linear statistics \eqref{linearstat} satisfies
		\begin{align}  \label{eq:meanpolyGen}
			\frac{1}{n}X_n(p(\mu_n))\to\frac{1}{\pi}\int_0^1\int_{b(\xi)-2a(\xi)}^{b(\xi)+2a(\xi)}\frac{ p (x)}{\sqrt{4a(\xi)^2-(x-b(\xi))^2}}dxd\xi+\int_{a^{-1}(\{0\})}p(b(\xi))d\xi
		\end{align} 
		as $n \to \infty$ almost surely.
	\end{theorem}

	The proof of this theorem is given in Section~\ref{sec:proofsGen}. Note that for $\mu_n(x)=x$, this result is already known in the literature. Indeed, in \cite{kuijlaars1999asymptotic}, the authors prove the exact same limit result for the asymptotic zero distribution of the orthogonal polynomials. That the empirical measure for the orthogonal polynomial ensembles and the zero distribution of the $n$-orthogonal polynomial have the same limit (if it exists) is a well-known result and is, for example, discussed in \cite{hardy2018polynomial}. The extension to general $\mu_n$ is straightforward, but we will present a direct proof, both for completeness and because, to the best of our knowledge, the simple arguments we provide have not been published before.
  	
	Note that in Theorem~\ref{limit_density_poly}, we have not specified any condition on $\mu_n$.  It may converge as in \eqref{eq:limitmun}, but it does not have to. All that matters is that the recurrence coefficients are slowly varying. It is, however, a bit restrictive to only allow linear statistics with functions $f$ that are polynomials in $\mu_n(x)$. It is desirable to allow for more general functions. Note that if the rescaled point process lives on a compact set (so $w$ has finite support), then it is natural to extend the latter theorem to continuous functions. However, if $w$ has infinite support then we need an additional condition that ensures that the points will accumulate on a compact subset (after rescaling) with high probability.

	\begin{theorem} \label{thm:LLNcontinuousfunctions} 
		Consider a discrete point process $\{x_j\}_{j=0}^{n-1}$  with its joint distribution proportional to \eqref{bi-ope}. Let the assumptions in Theorem~\ref{limit_density_poly} be satisfied.  Assume further that there exists a bounded $K\subset \mathbb{R}$ and $\epsilon>0$ with $\mu(K)$ being compact such that \begin{enumerate}
			\item 
			for all $k\in\mathbb{N}$, 
			\begin{align}  \label{eq:mu_mean_small}
				\mathbb{E}\left[X_n\left( \left|\mu_n\right|^k\chi_{K^c} \right)\right]  & = O_n(n^{-\epsilon}), 
			\end{align} 
			where $\chi_{K^c}(x)=1$ for $x\in K^c$ and zero otherwise.  \item   $\mu_n$  has a uniform limit $\mu$, i.e. 
			\begin{align}
				\lim_{n\to\infty}\sup_{x\in \frac{1}{n}\mathbb{Z}\cap K}\left|\mu_n(x)-\mu(x)\right| = 0 \label{eq:uniform_mu},
			\end{align}
			and $\mu$ is strictly monotone on $K$. 
		\end{enumerate}
		Then, for any $f$ such that $f\circ\mu^{-1}$ is continuous with at most polynomial growth at infinity, we have 
		\begin{align} \label{eq:lln_gen}
			\frac{1}{n}X_n(f)\to\frac{1}{\pi}\int_0^1\int_{\mu^{-1}(b(\xi)-2a(\xi))}^{\mu^{-1}(b(\xi)+2a(\xi))}f(x)\frac{d\mu(x)d\xi}{\sqrt{4a(\xi)^2-(\mu(x)-b(\xi))^2}}+\int_{a^{-1}(\{0\})}f\circ\mu^{-1}(b(\xi))d\xi,
		\end{align} 
		as $n \to \infty$ almost surely.
	\end{theorem}
	
	Note that if the supports $S_n$ of $x\mapsto w(xn)$ are finite and contained in a compact interval, then condition (1) of Theorem~\ref{thm:LLNcontinuousfunctions} is redundant. 
	
Although the extension to continuous functions is rather direct, the extension to piecewise continuous functions (note that the height function for the lozenge tilings, as discussed in the introduction, is a linear statistic with respect to a step function) does not hold without further conditions. 
\begin{theorem} \label{thm:LLNlpfunctions}
Consider a discrete point process $\{x_j\}_{j=0}^{n-1}$  with its joint distribution proportional to \eqref{bi-ope}. Let all the assumptions in Theorem~\ref{thm:LLNcontinuousfunctions} be satisfied and $a(\xi)$ and $b(\xi)$ be the limiting continuous functions in \eqref{ass:slowly_varying_recurrence}. Additionally, assume that there exists a $\gamma>2$ such that 
\begin{equation} \label{eq:conda}
\int_0^1 |a(\xi)|^{-\frac 1 \gamma } d \xi <\infty.
\end{equation}
Consider Riemann integrable functions $f$ and $f \circ \mu^{-1}$ on $K$ and $\mu(K)$ respectively, such that $f\circ \mu^{-1}$ has at most polynomial growth at infinity. Then we have 
\begin{align} \label{eq:lln_integr}
			\frac{1}{n}X_n(f)\to\frac{1}{\pi}\int_0^1\int_{\mu^{-1}(b(\xi)-2a(\xi))}^{\mu^{-1}(b(\xi)+2a(\xi))}f(x)\frac{d\mu(x)d\xi}{\sqrt{4a(\xi)^2-(\mu(x)-b(\xi))^2}}
		\end{align} 
		as $n \to \infty$ almost surely.
\end{theorem}
 The condition \eqref{eq:conda} implies that $a^{-1}(\{0\})$ is a set of measure $0$. In many examples of interest,  $a(\xi) $ has at most finitely many zeros and their order is low. In such cases, condition \eqref{eq:conda} is satisfied. At first glance, condition \eqref{eq:conda}  may look a bit odd, but it is the natural condition that makes the right-hand side of \eqref{eq:lln_integr} a continuous linear function on the space $L^p(\mu(K),dx)$ (up to composition with $\mu^{-1}$), as we will see in the proof.

Note that by changing the order of integration in \eqref{eq:lln_gen}, one obtains a formula for the limiting mean density of points. To this end,  define 
	\begin{equation} \label{eq:defIxi}
	I(x)=\left\{ \xi \in (0,1) \mid b(\xi)-2 a(\xi) \leq \mu(x) \leq b(\xi)+2 a(\xi)\right\}.
	\end{equation}
	Then, we can rewrite the right-hand side of \eqref{eq:lln_gen} as 
	\begin{align*}
		\int f(x) \rho(x) dx,
	\end{align*}
	where the density $\rho(x)$ is defined as 
	\begin{align}
		\rho(x)dx& =  \frac{1}{\pi }\int_{I(x)}\frac{d\xi  d\mu(x)}{\sqrt{4a(\xi)^2-(\mu(x)-b(\xi))^2}}, \label{eq:limit_shape_density}
	\end{align} 
	on the set
	$$
	x \in \mu^{-1} \left(\min _{\xi \in [0,1]}\left(b(\xi)-2 a(\xi)\right),\max _{\xi \in [0,1]} \left(b(\xi)+2 a(\xi)\right)\right).
	$$
	Remarkably, in various interesting cases, one can compute the integral over $\xi$ and find an explicit expression for the density.

	
 \subsection{Global fluctuations}
	In the next step, we study the fluctuations for the linear statistics $X_n(f)$ around its mean $\mathbb E [X_n(f)]$. In order to control the slope of the test function, we introduce the Lipschitz constant for any given function on a real domain $f\colon I\to \mathbb{R} $ and subset $E\subset \mathbb{R}$, 
	\begin{align}\label{def: lipschitz_seminorm}
		L_E(f)\coloneqq & \sup_{x\neq y \in E \cap I}\left| \frac{f(x)-f(y)}{x-y} \right| .
	\end{align}
	If $E=I$ we drop the subscript $E$, i.e., $L(f)$. The following theorem is an extension of Theorem 2.3 in  \cite{breuer2017central}. The main difference is in the presence of $\mu_n$. In particular, we need a non-trivial condition on the convergence of $\mu_n$ to $\mu$.
 \begin{theorem}\label{CLT_q_polynomial}
		Consider a discrete point process $\{x_j\}_{j=1}^{n}$ with its joint distribution proportional to \eqref{bi-ope}. 
		Assume there exist a function $\mu$ and a subset $K\subset \mathbb{R}$ such that $\mu(K)$ is compact, $\sup_{x\in K}|\mu(x)|<\infty $, $ L_K(\mu)<\infty $ (as defined in \eqref{def: lipschitz_seminorm}), and
		\begin{align}  
			\mathbb{E}\left[X_n\left( \left|\mu_n\right|^k\chi_{K^c} \right)\right] & = o_n(n^{-1}),\label{eq:mu_mean_small1}
		\end{align} 
  as $n \to \infty$, and
		\begin{align} \label{ass: mu smooth}
			\lim_{n\to\infty}\sup_{x\in \frac{1}{n}\mathbb{Z}\cap K}\left|\mu_n(x)-\mu(x)\right| = 0 , \quad \lim_{n\to\infty}L_K(\mu\circ\mu_n^{-1}-id) = 0 ,
		\end{align} 
		where $id(x)=x$. Assume further that there exist two real numbers $a(1), b(1)$ such that, for any fixed $j$,
		\begin{align}
			a_{j+n,n}\to a(1), \quad b_{j+n,n}\to b(1) \quad \text{as } n\to\infty .
		\end{align}
		Consider a  test function $f$ such that $f\circ \mu^{-1}$ is continuously differentiable on $\mu(K)$ and $f\circ \mu^{-1}$ grows at most polynomially at infinity.  Then, as $n\to \infty$, we have \begin{align}\label{eq:limitingvariance1}
			X_n(f)-\mathbb{E}[X_n(f)] & \to\mathcal{N}(0,\sum_{k\geq 1}k\left|\widehat{f\circ\mu^{-1}}_k\right|^2), \end{align}
   in distribution, where 
   \begin{align}
			\widehat{f\circ\mu^{-1}}_k & \coloneqq \frac{1}{2\pi}\int_0^{2\pi}f\circ\mu^{-1}(2a(1) \cos\theta+b(1))e^{-ik\theta}d\theta \label{eq:limitingvariance2}
		\end{align}
		and $\mu^{-1}(x)$ is the inverse function of $\mu(x)$.
	\end{theorem}
	
	Note that the conditions for the Central Limit Theorem are weaker than the conditions from the Law of Large Numbers. We only require \eqref{ass:slowly_varying_recurrence} to hold only for $\xi=1$ for Theorem~\ref{CLT_q_polynomial} and not necessarily for all $0\leq \xi \leq 1$.  However, Theorems~\ref{thm:LLNcontinuousfunctions} and~\ref{thm:LLNlpfunctions} need  \eqref{ass:slowly_varying_recurrence} to be satisfied for all  $0\leq \xi \leq 1$.

	What is different from the case of Orthogonal Polynomial Ensembles is the presence of $\mu$. Note that if $\mu$ is differentiable then  $f \circ \mu^{-1} \in C^1(I)$ means that 
    \begin{align}
        f'(x)=\mu'(x) g'(\mu(x)), \label{eq: singular test fun restriction}
    \end{align}
    for some $g \in C^1(I)$.  In general, if $\mu$ is differentiable with a strictly positive (or negative) derivative, then the above is valid for all $f \in C^1 (I)$. However, if $\mu'$ vanishes at certain points, then so does $f$. It means that we are restricting our class of test functions to $C^1$ functions whose derivative vanishes at the zeros of $\mu'$. This will happen in some degenerate examples.


\subsection{Extended ensembles}\label{subsec:dynamic}
The classical orthogonal polynomial ensembles have natural time-dependent extensions (for a general discussion on extended processes, we refer to \cite{duits2018global,JohLec} and the references therein). For such extensions, the global fluctuations can be analyzed in terms of the recurrence coefficients as above and the spectrum of the underlying generator that drives the dynamics \cite{duits2018global}. As we will see for the tiling model of the hexagon from the Introduction, also the $q$-Racah ensembles have such a natural  extension and this is why we will discuss the generalization of \cite{duits2018global} to a class of  processes that extend \eqref{eq:qOPE}.

Let $N$ be a fixed integer, and $$t_0<t_1<t_2<\dots<t_N$$ be a real-valued sequence. Let $\mu_n(\cdot,\cdot)$ be a given real-valued function on $\{t_0,t_1,\dots, t_N\}\times \frac{1}{n}\mathbb{Z}$ and, for  $m=1,2,\dots, N$,  let $w_{n,t_m}(nx)$ be an orthogonal weight on $\frac{1}{n}\mathbb{Z}$. For $m=1,\ldots, N$ and $k\in\mathbb N$, let $r_{k,t_m}(nx)$ be defined as the polynomial of order $k$ in $\mu_n(t_m,x)$ with positive leading coefficient such that
 \begin{align}
     \sum_{x\in\frac{1}{n}\mathbb{Z}}r_{j,t_m}(nx)r_{k,t_m}(nx)w_{n,t_m}(nx)=\delta_{j,k}.
 \end{align}
To define the extended process we need transition functions	\begin{align}\label{transition}
		U_{t_m}(x_i,x_j) \coloneqq \sum_{k\geq 0}\frac{c_{k,t_{m+1}}}{c_{k,t_m}}r_{k,t_m}(nx_i)r_{k,t_{m+1}}(nx_j), 
	\end{align}	
	where $c_{k,t_m}>0$. The associated extended point  process$\{(t_m,x_j(t_m))\}_{m=1,j=0}^{N,n-1}$  is then induced by the probability  measure  proportional to	\begin{multline}\label{ensemble}
		\sim \det\left(c_{k,t_0}r_{k,t_0}(nx_j(t_0))\right)_{j,k=1}^{n} \prod_{m=0}^{N-1}\det\left(U_{t_m}\left(nx_{i}(t_m),nx_{j}(t_{m+1})\right)\right)_{i,j=1}^{n} \\
  \times \det\left(\frac{r_{k,t_N}(nx_j(t_N))}{c_{k,t_N}}\right)_{j,k=0}^{n-1}\prod_{m=1}^{N}\prod_{j=1}^{n} w_{n,t_m}(nx_j(t_m)).
	\end{multline}
It is an interesting standard exercise, based on a repeated application of the Cauchy-Binet theorem, the orthogonality of the polynomials and the Vandermonde determinant,  that the marginal densities for the point process $\{(t_m,x_j)\}_{j=1}^n$ for fixed $m$ are of the type \eqref{eq:qOPE} with $\nu_m$ and weight $w_{n,t_m}$. 
  
  The linear statistics for this point process are the random variables
	\begin{align}
		X_n(f) = \sum_{m=1}^{N}\sum_{j=1}^{n} f(t_m,x_j(t_m))
	\end{align}
 for a test function $f$. As is the case for the one-dimensional marginals, the limiting behavior of these linear statistics can be studied using the  coefficients in the recurrence relations
\begin{align}\label{recurrence:time}
		\mu_n(t_m,x)r_{j,t_m}(nx)=a_{j,n}^{(t_m)}r_{j+1,t_m}(nx)+b_{j,n}^{(t_m)}r_{j,t_m}(nx)+a_{j-1,n}^{(t_m)}r_{j-1,t_m}(nx),
	\end{align}
 but we will also need a condition on the limiting behavior of the coefficients $c_{k,t_m}$ in the transition functions \eqref{transition}.

	First, note that since the marginal density of \eqref{ensemble} at any fixed $t_m$ is the same as that of \eqref{bi-ope}, Theorem~\ref{thm:LLNcontinuousfunctions}  still holds for any fixed  $t_m$  and thus also gives the limiting distribution of points for the extended point process. The following is an extension of Theorem \ref{CLT_q_polynomial} and describes the two-dimensional fluctuations.
	
 \begin{theorem}\label{thm:fluctuation_theorem}
		Fix $N\in\mathbb{N}$ and consider a probability measure on $\{t_0, t_1,\dots, t_N\}\times \frac{1}{n}\mathbb{Z}$ of the form \eqref{ensemble} with 
	$$t_0<t_1<t_2<\dots<t_N.$$
		Suppose there exist a function $\mu$ on $\{t_0, t_1,\dots, t_N\}\times \mathbb R$ and compact sets $K^{(t_m)}\subset \mathbb{R}$ such that
  \begin{enumerate}
      \item 
  $\mu(t_m,K^{(t_m)})$ is compact for all $m=1,2, \dots, N$ and thus, in particular,  $\sup_{ x\in K^{(t_m)} }\left|  \mu(t_m, x)\right| <  \infty.$
  \item we have 
		\begin{align}  
	\mathbb{E}\left[X_n\left( \left|\mu_n\right|^k\chi_{K^c} \right)\right] & = o(n^{-1}), \label{ass: moment}
\end{align} 
as $n \to \infty$, where $K^c\coloneqq {K^{(t_1)}}^c\times {K^{(t_2)}}^c\times \dots \times{ K^{(t_N)}}^c \subset \mathbb{R}^N$.
\item for   $m=1,\dots,N$, we have 
\begin{align}
	\lim_{n\to\infty}\sup_{x\in \frac{1}{n}\mathbb{Z}\cap K^{(t_m)}}\left|  \mu_n(t_m,x)-\mu(t_m,x)\right|= & 0, \label{ass:mu_sup}
 \end{align}
 \item for  $m=1,\dots,N$, we have 
 \begin{align}
	\lim_{n\to\infty}L_{K^{(t_m)}}\left(  \mu(t_m, \mu_n^{-1}(t_m, \cdot))-id(\cdot)\right)= & 0, \label{ass:mu_lip}
\end{align}
where $id(x)=x$
\item for  $m=1,\dots,N$, we have 
\begin{align}\label{ass:mu}
	L_{K^{(t_m)}}\left(  \mu(t_m, \cdot)\right) <  \infty.
\end{align} 
\end{enumerate}
Suppose the recurrence coefficients \eqref{recurrence:time} are such that, for any fixed $j\in\mathbb{N}$ and $m=1,2,\dots,N$, we have
\begin{align}\label{ass:cltjaccobi}
	a_{j+n,n}^{(t_m)}\to a(1;t_m) \quad b_{j+n,n}^{(t_m)}\to b(1;t_m)\quad \text{as } n\to\infty ,
\end{align}
for some $b(1;t_m)\in\mathbb{R}$ and $a(1;t_m)>0$, and, finally, for fixed $k,l \in \mathbb N$ we have
\begin{align}\label{ass:cltc}
	\lim_{n\to\infty}\frac{c_{n+k,t_m}}{c_{n+l,t_m}}=e^{\tau_m(l-k)}, 
\end{align}
for some $\tau_1<\tau_2<\dots<\tau_N$. Then for any $f\colon\{t_1,\dots,t_N\}\times\mathbb{R}\to\mathbb{R}$ such that, for each $m$, the restriction $x\mapsto f(t_m,\mu^{-1,(t_m)}(x))$ is continuously differentiable on $\mu(t_m,K^{(m)})$ with at most polynomial growth at infinity,   we have
\begin{align}
	X_n(f)-\mathbb{E}[X_n(f)]\to\mathcal{N}\left(0, \sum_{r_1,r_2=1}^N\sum_{k=1}^\infty ke^{-|\tau_{r_1}-\tau_{r_2}|k}\widehat{f\circ \mu^{-1}}^{(r_1)}_k\widehat{f\circ \mu^{-1}}^{(r_2)}_k\right)\label{eq: CLT Gaussian},
\end{align} 
where 
\begin{align}
	\widehat{f\circ \mu^{-1}}^{(r)}_k=\frac{1}{\pi}\int_0^\pi f\left(t_r,\mu^{-1,(t_r)}\left(b(1;t_r)+2a(1;t_r)\cos\theta\right)\right)\cos(k\theta)d\theta . \label{eq: CLT Gaussian Variance}
\end{align}
and  $\mu^{-1,(t_m)}$ is the inverse of $x\mapsto\mu(t_m,x)$. 
	\end{theorem}

The proof of Theorem~\ref{thm:fluctuation_theorem} will be given in Section~\ref{sec:proofsGen}.


 \section{A first example: $q$-Racah ensemble} \label{sec:q-Racah} In this section, we will  discuss the $q$-Racah ensemble and illustrate some results from Section~\ref{sec:generalOPE}. The computations here will also be relevant for our running example of lozenge tilings of the hexagon with the weight \eqref{eq:ellipticloz}.

\subsection{$q$-Racah ensemble} 
	Our running example will be that of the $q$-Racah ensemble. For $M\in \mathbb N$ set  $-\log_{\bf q} \gamma_n=M+1$. The $q$-Racah ensemble of size $n<M$ is then the probability density on $y_1,\ldots,y_n$ proportional to \eqref{eq:qOPE} with 
	\begin{align} \label{eq:nuRacah}
		\nu_n(y) = \mathbf{q}^{-y}+\delta \gamma_n\mathbf{q}^{y+1},
	\end{align}
		and weight function
\begin{align}
	w(y)= \frac{\left(\alpha \mathbf{q} ,\beta\delta \mathbf{q},\gamma_n \mathbf{q},\gamma_n\delta \mathbf{q};\mathbf{q}\right)_y\left(1-\gamma_n\delta \mathbf{q}^{2y+1}\right)}{\left(\mathbf{q} ,\alpha^{-1}\gamma_n\delta \mathbf{q},\beta^{-1}\gamma_n \mathbf{q},\delta \mathbf{q};\mathbf{q}\right)_y\left(\alpha\beta \mathbf{q}\right)^y\left(1-\gamma_n\delta \mathbf{q}\right)}, \qquad y=0,\ldots,M, \label{eq:q-racah weight}
\end{align} 
where $(x_1,\cdots,x_j;\mathbf{q})_k=(x_1;\mathbf{q})_k\cdots(x_j;\mathbf{q})_k$ and $(x;\mathbf{q})_k=(1-x)(1-x\mathbf{q})\cdots (1-x\mathbf{q}^{k-1})$ is the $q$-Pochhammer symbol. One needs to specify the parameters $\alpha,\beta,\delta,\gamma_n$ such that $w(y)$ is non-negative for all $y$ to become a weight. We do not intend to deal with the most general choice of parameters that do this but discuss a few special choices. 
One readily verifies that $w(y)$ is positive in the following cases:

\begin{align*}
	\text{case 1:} & \qquad 0<\mathbf{q}<1,\qquad \alpha, \beta>0, \qquad \delta\geq 0 ,\qquad \beta\delta<1,\qquad \beta\geq \gamma_n, \qquad \alpha\geq \gamma_n\\
	\text{case 2:} & \qquad
	1<\mathbf{q},\qquad \alpha, \beta>0,\qquad \delta\geq 0 ,\qquad \alpha^{-1}\delta<1,\qquad \beta\leq \gamma_n, \qquad \alpha\leq \gamma_n\\
	\text{case 3:} & \qquad
	0<\mathbf{q}<1,\qquad \alpha, \beta>0, \qquad \delta< 0 ,\qquad \beta\geq \gamma_n, \qquad \alpha\geq \gamma_n\\
	\text{case 4:} & \qquad
	1<\mathbf{q},\qquad \alpha, \beta>0,\qquad \delta< 0 ,\qquad  \beta\leq \gamma_n, \qquad \alpha\leq \gamma_n.
\end{align*}
Note that the cases $0<q<1$ and $q>1$ are dual. Important is that in all four cases we have the following:
$$
\log_{\mathbf q} \alpha \beta <2\log_{\mathbf q} \gamma_n=-2(M+1),
$$
and this holds both for $0<q<1$ and $q>1$. 

Another interesting choice of parameters is known in the literature as the trigonometric case. Here $\mathbf{q},\alpha,\beta,\delta,\gamma_n$ are on the unit circle. In this case, there are some natural conditions on the parameters for which the orthogonality weight $w(y)$ becomes non-negative. Details can, for example, be found in Proposition 6.1. in \cite{van1998multivariable}. Also note that in this case $\nu_n(y)=2\sqrt{\gamma_n\delta\mathbf{q}}\cos\left(\arg \left(\sqrt{\gamma_n\delta\mathbf{q}}\right)y\right)$ is a monotone function after factoring out  $\sqrt{\gamma_n\delta\mathbf{q}}$.

We recall that the $q$-Racah ensembles are our main interest, as they are directly related to the tiling model from the Introduction. Another important feature of the $q$-Racah polynomials is that they are on top of the Askey scheme. That means other classical families and their $q$-analogues can be obtained by taking special limits of the $q$-Racah polynomials. For instance, taking $\delta=0$, $w$ becomes the $q$-Hahn polynomial weight. If one further sets $\alpha=\mathbf q^{\tilde{\alpha}}$, $\beta= \mathbf q^{\tilde{\beta}}$, where $\tilde{\alpha},\tilde{\beta}\in\mathbb{N}$, one gets 
$$
\lim\limits_{\mathbf q\to1}w(y) = \binom{\tilde{\alpha}+y}{y}\binom{\tilde{\beta}+M-y}{M-y},
$$ 
to be the weight for Hahn polynomials. Similarly, setting  $\beta=-\alpha^{-1} p \mathbf q^{-1}$ and $\delta=0$, we have 
$$
\lim\limits_{\mathbf q\to1} \lim\limits_{\alpha\to0}w(y) = \binom{M}{y}p^{-y},
$$ 
to be the weight for Krawtchouk polynomials. For more about $q$-analogues of orthogonal polynomials, one may refer to \cite{koekoek2010hypergeometric}.\\

To apply the results of Section~\ref{sec:generalOPE}, we first need to recall the recurrence relation for the $q$-Racah polynomials.

Denote $R_j(\nu_n(y))$ to be the $q$-Racah orthogonal polynomial of order $j$ which is given by the $q$-hypergeometric function ${}_4\phi_3$ via the $q$-Pochhammer symbol $(\cdot;\mathbf{q})_k$ as follows (for more details we refer to \cite{koekoek2010hypergeometric}), 
\begin{align}\label{eq:q-racah poly}
	R_j(\nu_n(y)) & = {}_4\phi_3\bigg(\begin{matrix}\mathbf{q}^{-j}, \alpha\beta \mathbf{q}^{j+1},\mathbf{q}^{-y},\delta \gamma_n \mathbf{q}^{y+1}\\\alpha \mathbf{q},\delta \beta \mathbf{q},\gamma_n \mathbf{q}\end{matrix}; \mathbf{q},\mathbf{q} \bigg).
\end{align}
Instead of $R_j$ we will work with the monic polynomials $p_j$ defined by   $$p_j(\nu_n(y))=R_j(\nu_n(y))\frac{(\alpha \mathbf{q},\delta \beta \mathbf{q}, \gamma_n \mathbf{q};\mathbf{q})_j}{( \alpha\beta \mathbf{q}^{j+1};\mathbf{q})_j}.$$ Then, the three-term recurrence relation can be written as 
\begin{align}
	\nu_n(y)p_j(\nu_n(y)) & =p_{j+1}(\nu_n(y))+(1+\delta\gamma_n \mathbf{q}-A_j-C_j)p_j(\nu_n(y))+A_{j-1}C_jp_{j-1}(\nu_n(y)), 
\end{align}
where 
\begin{align}
	A_j & =\frac{(1-\gamma_n \mathbf{q}^{j+1})(1-\alpha \mathbf{q}^{j+1})(1- \alpha\beta \mathbf{q}^{j+1})(1-\delta \beta \mathbf{q}^{j+1})}{(1-\alpha\beta \mathbf{q}^{2j+1})(1-\alpha\beta \mathbf{q}^{2j+2})},\label{eq:recurrence A}\\
	C_j & =\frac{\mathbf{q}\left(1-\mathbf{q}^j\right) \left(1-\beta  \mathbf{q}^j\right) \left(\delta -\alpha  \mathbf{q}^j\right) \left(\gamma_n -\alpha  \beta  \mathbf{q}^j\right)}{(1-\alpha\beta \mathbf{q}^{2j})(1- \alpha\beta \mathbf{q}^{2j+1})} \label{eq:recurrence C}. 
\end{align}
In the previous sections, we have used recurrences for the normalized polynomials and not the monic polynomials. The recurrence relation  for the normalized polynomials $r_j(y)$ can be easily computed from the recurrence relation for the monic polynomials, giving
\begin{align}\label{eq: q-racah recurrence}
	\nu_n(y)r_j(y) & =\sqrt{A_jC_{j+1}}r_{j+1}(y)+(1+\delta\gamma_n \mathbf{q}-A_j-C_j)r_j(y)+\sqrt{A_{j-1}C_j}r_{j-1}(y).
\end{align}
Now that we have the recurrence relations spelled out explicitly. We are ready to apply the results from Section~\ref{sec:generalOPE}.


\subsection{Limiting density}

We will discuss the limits of the $q$-Racah ensembles as $n \to \infty$ under the following assumptions:
\begin{align*}
    \mathbf{q}&=q^{\frac{1}{n}} \qquad \text{for $0<q<1$ or $q>1$,}\\
    x&=\frac{y}{n}, \qquad \mu_n(x) = \nu_n(y), \qquad \text{for $x\in\frac{1}{n}\mathbb{Z}$.}\\
    &\gamma_n=q^{-\frac{M+1}{n}},
\end{align*}
Then we take the limit $n \to \infty$ and $M\to \infty$ simultaneously such that $$\lim_{n \to \infty}\gamma_n=\gamma, \text{ with }-\log_q \gamma>1.$$
Taking $n\to\infty$ one can easily see the existence of the limits of $A_j$ and $C_j$ : $$a(\xi)=\lim\limits_{j/n\to\xi} \sqrt{A_{j}C_{j+1}}\,\,\, \text{and}\,\,\,b(\xi)=1+\delta\gamma-\lim\limits_{j/n\to\xi} (A_{j}+C_{j}),$$ 
and here 
$$
0 < \xi < -\log_q \gamma.
$$
Note that we are particularly interested in the values $0 < \xi \leq 1$. But in the discussion below, $\xi>1$ will also play a role.

Finally, $$\mu(x):=\lim\limits_{n\to\infty}\mu_n(x) = q^{-x}+\gamma\delta q^{x}.$$ 
Theorem~\ref{thm:LLNlpfunctions} gives the limiting density of the points in the $q$-Racah ensemble.  We also find the limiting shape density $\rho$, defined via \eqref{eq:limit_shape_density}, as an integral. We can further compute that integral explicitly for the scaled $q$-Racah case, which we will discuss now.

  Central to the result is the equation
    \begin{equation}\label{eq: 4a2-mub2}
      4a(\xi)^2-(\mu(x)-b(\xi))^2=0.
  \end{equation}
  Note that for each $\xi$, there are two solutions for $x$. However, we will be interested in finding solutions $\xi$ in terms of $x$. It turns out that there are at most two real solutions $\xi_\pm(x)$ that we order as $\xi_-(x) \leq \xi_+(x)$ (when they exist). This may not be obvious from \eqref{eq: 4a2-mub2} but can be seen  as follows: With 
    \begin{equation}\label{eq:def:y}
        y(\xi)=q^{-\xi}+\alpha\beta q^{\xi},
    \end{equation} we will see in the proof below that  \eqref{eq: 4a2-mub2} can be rewritten as a quadratic equation in $y(\xi)$, i.e.,
  \begin{align} \label{eq:def:y:ext}
      (4a(\xi)^2-(\mu(x)-b(\xi))^2)(y(\xi)^2-4\alpha\beta) = -\left(\frac{\mu'(x)}{\log(q)}\right)^2y(\xi)^2+h(x)y(\xi)+p(x),
  \end{align}
  where $h(x)$ and $p(x)$ can be computed explicitly. As long as $\mu'(x)\neq 0$, there are two solutions for $y(\xi)$. If, furthermore,  $y(\xi)$  is monotone, we will have two solutions for $\xi$. The latter poses conditions on the parameters of the $q$-Racah ensemble, which leads to the following result:
	\begin{theorem} \label{limit_shape_qracah}
		Assume that the parameters $\alpha, \beta$ are either $\alpha\beta\leq0$ or $\log_q(\alpha\beta)< 2\log_q \gamma$.  If $\mu'(x) \neq 0$, then
  there are two real solutions $\xi_\pm(x)$ to \eqref{eq: 4a2-mub2}, ordered as $\xi_-(x)\leq \xi_+(x)$. 
	Then, as $n \to \infty,$ we have that the limiting density for the rescaled $q$-Racah ensemble is given by
        \begin{align}\label{eq:limitshape}
			\rho(x)= \begin{cases} 1 & \quad \xi_-(x)<\xi_+(x)\leq 1\\ \frac{1}{\pi}\arccos\left(\frac{2y(1)-y(\xi_-(x))-y(\xi_+(x))}{\left|y(\xi_-(x))-y(\xi_+(x))\right|}\right)& \quad \xi_-(x)<1<\xi_+(x) \\ 0 &\quad 1\leq \xi_-(x) <\xi_+(x).
			\end{cases}
		\end{align}
  with $y(\xi)$  as defined in \eqref{eq:def:y}.

	\end{theorem}

 \begin{figure}[t]
		\begin{subfigure}[b]{0.3\textwidth}
			\includegraphics[scale=0.2]{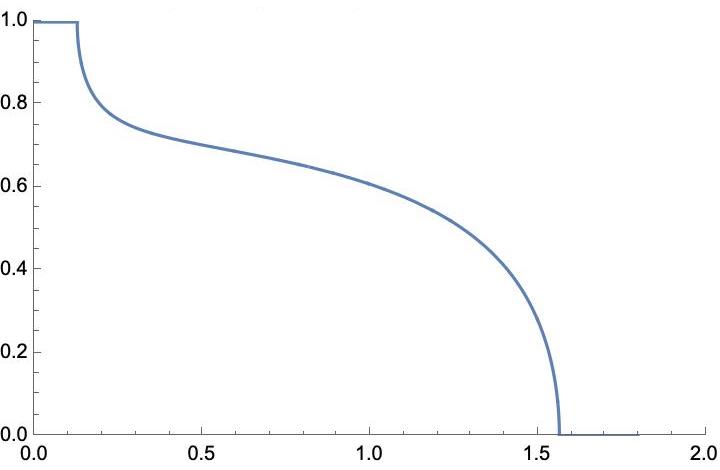}
			\caption{
$\alpha=e^{-2},\beta=e^{-2},\gamma=e^{-1.8},$\\ $\delta=e^2$}
		\end{subfigure}
		\begin{subfigure}[b]{0.3\textwidth}
			\includegraphics[scale=0.2]{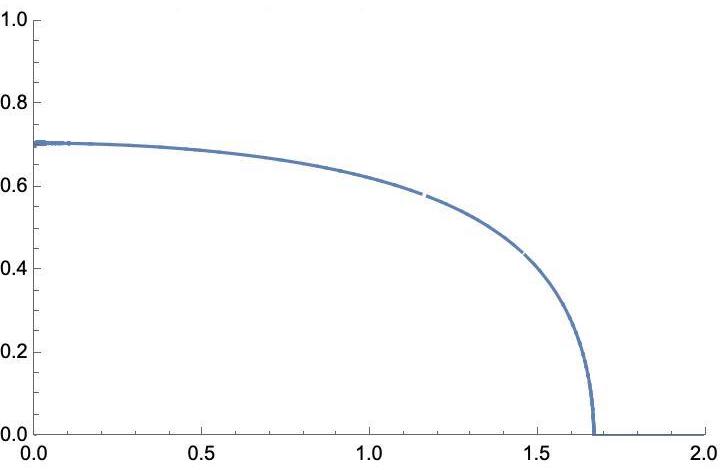}
			\caption{
$\alpha=e^{-2},\beta=e^{-2},\gamma=e^{-2},$\\ $\delta=e^2$}
		\end{subfigure}
  \begin{subfigure}[b]{0.3\textwidth}
			\includegraphics[scale=0.2]{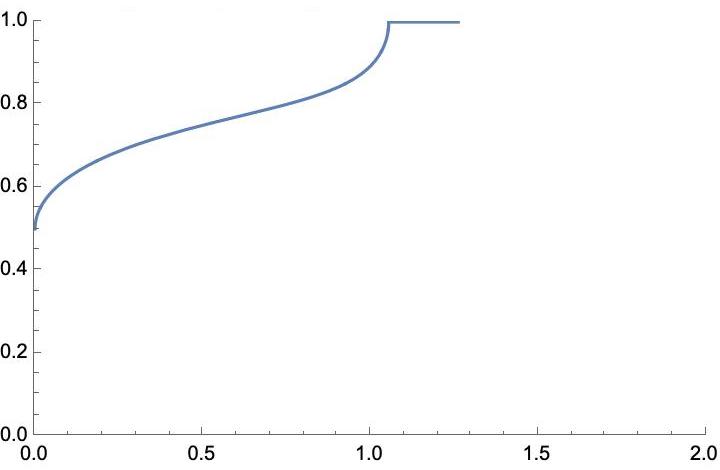}
			\caption{
 $\alpha=4,\beta=4, \gamma=e^{0.18},$\\ $\delta=-841$}
		\end{subfigure}
		\begin{subfigure}[b]{0.3\textwidth}
			\includegraphics[scale=0.2]{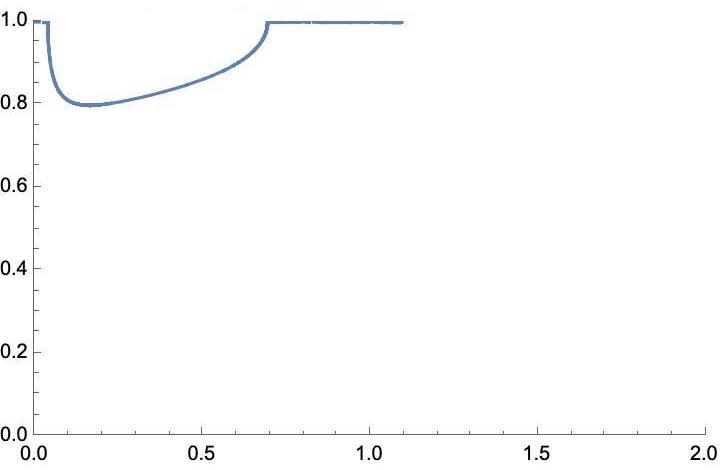}
			\caption{
   $\alpha=4,\beta=4,\gamma=2^{1.1},$\\$\delta=-841$}
		\end{subfigure}
		\begin{subfigure}[b]{0.3\textwidth}
			\includegraphics[scale=0.2]{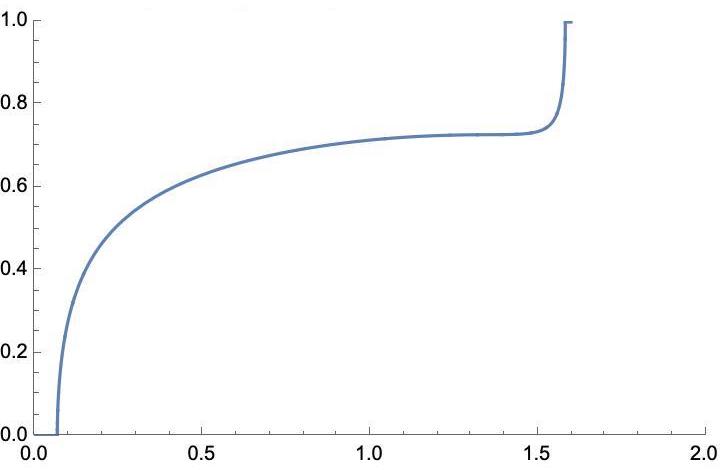}
			\caption{
$\alpha=4,\beta=4,\gamma=2^{1.6},$\\ $\delta=-841$}
		\end{subfigure}
		\begin{subfigure}[b]{0.3\textwidth}
			\includegraphics[scale=0.2]{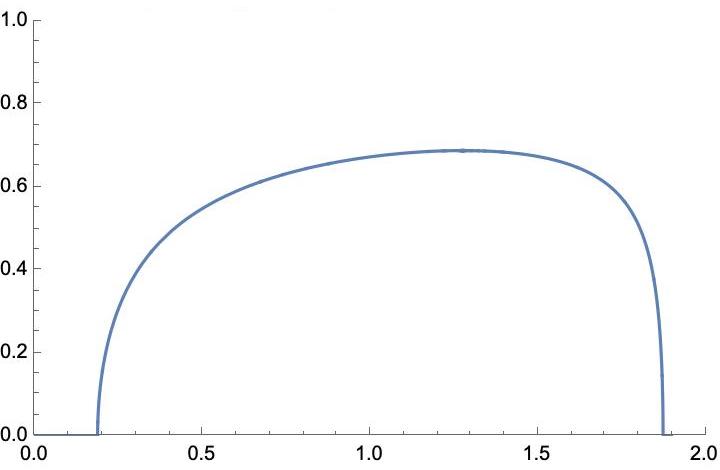}
			\caption{
   $\alpha=4,\beta=4$,$\gamma=2^{1.9}$,\\$\delta=-841$}
		\end{subfigure}
		\caption{Plot of the densities for different choices of parameters.  }
        \label{fig:q-racah density}
	\end{figure}

The proof of this theorem is based on a change of variables, and we postpone this computation to Appendix~\ref{proof:qracahLimitShape}.

We will proceed by making some remarks. First of all, Theorem~\ref{limit_shape_qracah} relies on a change of variables from $\xi$ to $y(\xi)$. For this argument to work, we only need $y(\xi)$ to be monotone for $x \in [0,1]$. Thus $\log_q\alpha \beta<-2\log_q\gamma$ can be replaced by $\log_q\alpha \beta<-2$. However, for simplicity and to ensure that $\xi_\pm(x)$  are well-defined, we work with $\log_q\alpha \beta<-2\log_q\gamma$.

 Figure~\ref{fig:integral} illustrates how to find the $\xi_-(x)$ and $\xi_+(x)$ given the curves by $x=b(\xi)-2a(\xi)$ and $x=b(\xi)+2a(\xi)$. Note that from the recurrence coefficients in \eqref{eq:recurrence A} and \eqref{eq:recurrence C}, it is easy to see that $a(0)=0$ and $a(-\log_q(\gamma))=0$, where $-\log_q \gamma > 1$,  and there are no zeros in between $0$ and $-\log_q(\gamma)$. Hence, the curves enclose a simply connected domain.

 From Theorem~\ref{thm:LLNcontinuousfunctions} we know that the support of the density for $\rho(x)$ is
  \begin{align}
      \mu^{-1} \left(\min\limits_{\xi \in [0,1]}\left(b(\xi)-2 a(\xi)\right),\max \limits_{\xi \in [0,1]} \left(b(\xi)+2 a(\xi)\right)\right). \label{set:continuous}
  \end{align}
For any $x$ in the support, we thus have $\xi_-(x)\leq 1$ (this is also clear from Figure~\ref{fig:integral}).  Then for $x \in [b(1)-2a(1),b(1)+2a(1)]$ we have  $\xi_-(x)<1< \xi_+(1)$. For $x$ in the support but outside the interval $[b(1)-2a(1),b(1)+2a(1)]$ we have $\xi_+(x)<1$.  Note that in the latter region, the limiting density is $1$ meaning the lattice is densely packed with points. Such a region is called saturated in the literature. In the interval $[b(1)-2a(1),b(1)+2a(1)]$, the density is less than one, and one readily verifies that the density is analytic in the interior. At its endpoints $x=b(1)\pm 2 a(1)$, we have  $\xi_-(x)=1$ or $\xi_+(x)=1$ and we will investigate briefly how the density behaves near such points.  

Assume that $q>1$. By a  Taylor approximation around the boundary point $x_0$ such that   $\xi_+(x_0)=1>\xi_-(x_0)$, we find  
\begin{align*}
    \rho(x)=\frac{1}{\pi}\arccos \left(-1-\frac{2y'(1)\xi_+'(x_0)}{y(\xi_-(x_0)-y(1)} (x-x_0) + O(|x-x_0|^2)\right).
\end{align*}
Note that $y'(1)=(-q^{-1}+\alpha\beta q^{1})\log(q)\neq 0$. Also note that $\xi_+'(x_0)=\frac{\mu'(x_0)}{b'(1)+2a'(1)}$, which can be seen by taking the derivative on both sides of $\mu(x)=b(\xi_+(x))+2a(\xi_+(x))$. From here, we see that if $\mu'(x_0)\neq 0$, the density $\rho(x)$ vanishes as a square root at $x_0$. Similar arguments show that if $x_0$ is such that $\xi_+(x_0)>\xi_-(x_0)=1$, then $1-\rho(x)$ has a square root vanishing near $x_0$. This describes the generic behavior near the endpoints $x_0=b(1)\pm 2a(1)$.
For numerical illustrations, we refer to cases (A), (D), (E), and (F), at $x_0=1.67$ in case (B) and at $x_0=1.056$ in case (C) in Figure~\ref{fig:q-racah density}. 
    
\begin{figure}[t]
		\centering
		\includegraphics[scale=0.4]{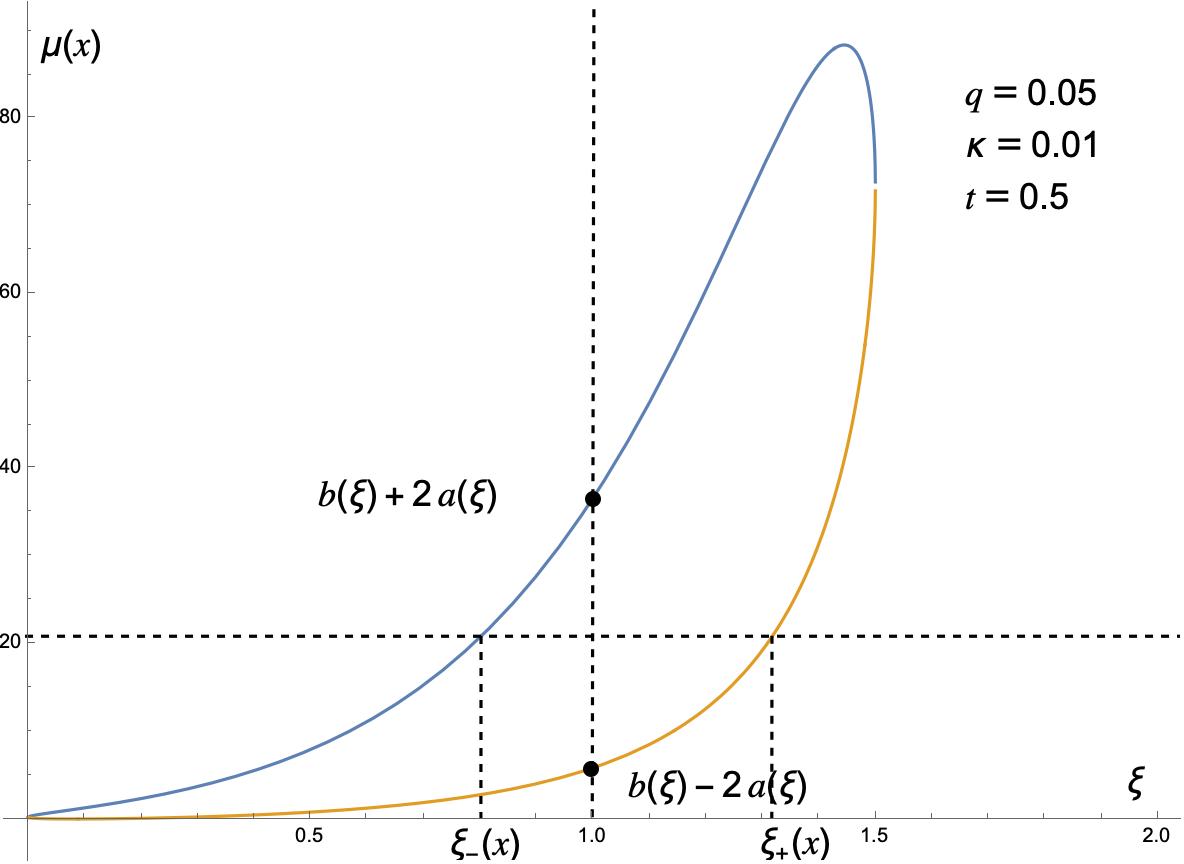}
		\caption{Interval of integration for $q$-Racah with $\alpha,\beta= q^{-2}$, $\gamma= q^{-t-1}$, $\delta=\kappa^2$}
		\label{fig:integral}
\end{figure}
  
Now we turn to special points $x_0$ such that $\xi_+(x_0)=\xi_-(x_0)=1$ and $\mu'(x_0)\neq 0$. At such points, the discriminant of \eqref{eq:def:y:ext} vanishes. The discriminant equals \begin{align}
    16 q^{-4 x} \left(q^x-1\right) \left(\alpha  q^x-1\right) \left(\gamma  q^x-1\right) \left(\delta  q^x-1\right) \left(\gamma  q^x-\beta \right) \left(\beta  \delta  q^x-1\right) \left(\gamma  \delta  q^x-1\right) \left(\gamma  \delta  q^x-\alpha \right), \label{eq: discriminant}
  \end{align}
  and thus, by our choice of parameters,  vanishes only at $x=0$ or $x=-\log_q (\gamma)$.	
  Consider the case $x_0=0$ and the discriminant $\Delta(x)$ as \eqref{eq: discriminant} such that $\lim\limits_{x\to 0}\frac{\Delta(x)}{x}=c>0$. In that case, one has 
\begin{align*}
    \rho(x)=\frac{1}{\pi}\arccos\left(c_1\sqrt{x}+c_2 x^{\frac{3}{2}} + O(x^{\frac{5}{2}})\right), 
\end{align*}
where $c_1=\frac{y'(1)\mu'(0)^32b'(1)}{(\log(q)^2)\sqrt{c}(b'(1)^2-4a'(1)^2)}$ and some $c_2\in\mathbb{R}$. Since we assume $\mu'(0)\neq 0$, we also have $b'(1)^2-4a'(1)^2\neq 0$. 
 If $b'(1)\neq 0$, the limiting density satisfies $\rho(0)=\frac12$ and $\rho(x)-\frac{1}{2}$ vanishes as a square root near $x=0$. See case (C) in Figure~\ref{fig:q-racah density} for an example.
 However, for $b'(1)=0$, and hence $c_1=0$, then one needs to compute $c_2=-\frac{\mu'(0)^4y''(1)}{2\log(q)^2a'(1)\sqrt{c}}\neq 0$ and we see that $\rho(x)-\frac 12$ vanishes with exponent $3/2$.

Finally, we comment on the case $\mu'(x_0)=0$, which we will call the singular case. Since $\mu$ is assumed to be monotone (and $\mu''(x_0)\neq 0$ if $\mu'(x_0)=0$), such points can only be at endpoints of the interval, $x_0=0$ or $x_0=\log_q\gamma$. For this to happen at $x_0=0$, we need $\gamma \delta =1$, and at $x_0=-\log_q\gamma$ we need $\gamma=\delta$.  For example, consider the configuration 
\begin{align*}
    \alpha=\gamma ,\quad \beta=\gamma ,\quad \delta=\frac{1}{\gamma} ,\quad q>1, 
\end{align*}
one can approximate the density around $0$ to be
\begin{align*}
    \rho(x) = \frac{1}{\pi}\arccos\left(\frac{2 \gamma ^2 q^2-(\gamma +1)^2 q+2}{(\gamma -1)^2 q}+\frac{(\gamma +1)^2 (\gamma  q-1)^2 \log ^2(q)}{(\gamma -1)^4 q}x^2+O(x^3)\right).
\end{align*}
Thus the limiting density $\rho(0)$ takes a value possibly different from $0,1$ and $\frac{1}{2}$ and $\rho(x)-\rho(0)$ has a square vanishing near $x=0$. 
It is also important to note that  $\lim\limits_{x\to 0 }\xi_-(x)=0$ and $\lim\limits_{x\to 0 }\xi_+(x)=-\log_q(\gamma)$.  This example is illustrated in the case (B) of Figure~\ref{fig:q-racah density}.

Note for case 1 (listed just below \eqref{eq:q-racah weight}) the limiting density result for $q$-Racah ensemble was first obtained by Dimitrov and Knizel in \cite{dimitrov2019log}. They express the density $$\rho(x)=\frac{1}{\pi}\arccos\left(\frac{R(q^{-x})}{2\sqrt{\Phi^-(q^{-1})\Phi^+(q^{-1})}}\right)$$ in terms of quantities in the $q$-analog of Nekrasov's equation and loops equation (see (7.13) in \cite{dimitrov2019log}). Instead, we characterize the density in terms of recurrence coefficients of the $q$-Racah polynomials. With some explicit computations, one can verify that
 \begin{align*}
     R(q^{-x})= & \frac{q^{-2 x}}{2}\left(\frac{\mu'(x)}{\log(q)}\right)^2\left(2y(1)-y(\xi_-(x))-y(\xi_+(x))\right), \\
     \Phi^-(q^{-1})\Phi^+(q^{-1})= & \frac{q^{-4 x}}{16}\left(\frac{\mu'(x)}{\log(q)}\right)^4\left|y(\xi_-(x))-y(\xi_+(x))\right|.
 \end{align*}
 Thus, our expression agrees with theirs.

 
 \subsection{Trigonometric $q$-Racah ensembles}

 In the following, we will illustrate a trigonometric case. Let $n,M\in\mathbb{N}$ with $n\leq M$. Consider the probability density of $\{y_j\}_{j=1}^n$ follows a $q$-Racah polynomial ensemble defined by \eqref{eq:nuRacah} and \eqref{eq:q-racah weight}. Let $g_{\gamma_n}<0$ such that $-ng_{\gamma_n} =M+1\in\mathbb{N}_+$. Consider the $q$-Racah polynomials defined by \eqref{eq:q-racah poly} with $M$ being the highest order. The recurrence coefficients are given by \eqref{eq:recurrence A}, \eqref{eq:recurrence C} and \eqref{eq: q-racah recurrence}.
In the trigonometric case, the scaling of the parameters are as follows
\begin{align}
	\mathbf{q}= q^{\frac{1}{n}}, \qquad q=e^{ig_q}, \qquad \alpha=q^{g_\alpha}, \qquad \beta=q^{g_\beta}, \qquad \delta=q^{g_\delta}, \qquad \gamma_n=q^{g_{\gamma_n}}.
\end{align}	
 Note that $\alpha,\beta,\delta$ does not depend on $n$. In the rest of this subsection, we only consider the trigonometric case with the following restrictions

\begin{multline}
    0<g_q<\pi, \qquad -\frac{\pi}{g_q}\leq g_{\gamma_n}\leq -1, \qquad \frac{2\pi j}{g_q}\leq g_\alpha\leq g_{\gamma_n} +\frac{2\pi (j+1)}{g_q},\\
    \frac{2\pi j}{g_q}\leq g_\beta \leq g_{\gamma_n} +\frac{2\pi (j+1)}{g_q},\qquad
    \frac{2\pi k}{g_q}\leq g_\alpha+g_\beta \leq 2g_{\gamma_n} +\frac{2\pi (k+1)}{g_q},\\
     \frac{2\pi l}{g_q}  \leq g_\beta+g_\delta \leq \frac{2\pi (l+1)}{g_q}+g_\gamma, \qquad \frac{2\pi m}{g_q}  \leq g_\alpha-g_\delta \leq \frac{2\pi (m+1)}{g_q}+g_\gamma,
\end{multline}
for some $j,k,l,m,\in\mathbb{Z}$ such that $|l-m|$ is an odd integer. One can verify that the $q$-Racah weight $w(y)$ \eqref{eq:q-racah weight} is non-negative for all $y=0,1,\cdots, -ng_{\gamma_n}$. Proposition 6.1. in \cite{van1998multivariable} offers a general criterion for trigonometric $q$-Racah polynomials where $w(y)$ is non-negative.

We make the following scaling assumptions 
 \begin{align*}
      x=\frac{y}{n}, \qquad \text{ for }x\in\frac{1}{n}\mathbb{Z},
 \end{align*}
 and then take the limits $n\to\infty$ and $M\to\infty$ simultaneously such that 
 \begin{align*}
\lim_{n\to\infty}g_{\gamma_n} =g_\gamma, \qquad \text{ for some } g_\gamma\in\left(-\frac{\pi}{g_q},0\right).
 \end{align*} 
 Now, we are ready to state the following result.
 \begin{theorem}[Trigonometric case]\label{limit_shape_qracah tri}
Assume there exist some $k_1,k_2\in\mathbb{Z}$ such that
\begin{align}
 \frac{2k_1\pi}{g_q}< g_\alpha+g_\beta<\frac{2\pi (k_1+1)}{g_q}+2g_\gamma, \qquad \frac{2k_2\pi}{g_q}\leq   g_\gamma+g_\delta\leq\frac{2\pi (k_2+1)}{g_q}+2g_\gamma.
\end{align}
If $\mu'(x) \neq 0$, then
there are two real solutions $\xi_\pm(x)$ to $4a(\xi)^2-(\mu(x)-b(\xi))^2=0$, where $a(\xi),b(\xi)$ are the limits of the recurrence coefficients, ordered as $\xi_-(x)\leq \xi_+(x)$. 
Then, as $n \to \infty,$ we have that the limiting density for the rescaled $q$-Racah ensembles is given by
\begin{align}\label{eq:limitshapetri}
	\rho(x)= \begin{cases} 1 & \quad \xi_-(x)<\xi_+(x)\leq 1\\ \frac{1}{\pi}\arccos\left(\frac{2\tilde{y}(1)-\tilde{y}(\xi_-(x))-\tilde{y}(\xi_+(x))}{\left|\tilde{y}(\xi_-(x))-\tilde{y}(\xi_+(x))\right|}\right)& \quad \xi_-(x)<1<\xi_+(x) \\ 0 &\quad 1\leq \xi_-(x) <\xi_+(x).
	\end{cases}
\end{align}
with $\tilde{y}(\xi) \coloneqq 2\cos\left( g_q\left(\frac{g_\alpha+g_\beta}{2}+\xi\right) \right)$.
\end{theorem}
The proof of this result is similar to the proof of Theorem~\ref{limit_shape_qracah} but with an appropriate adjustment of the change of variables. We give the details of this computation in Appendix~\ref{limit_shape_qracah tri}.

	\section{Proof of general limit theorems} \label{sec:proofsGen}
 In this section, we prove the main results of Section~\ref{sec:generalOPE}.

\subsection{Preliminaries}
	 For any random variable $X$, its cumulants $\mathbf{C}_k$ are given by the expansion of its Laplace transform, if it exists,
	\begin{align*}
		\log\mathbb{E}\left[e^{\lambda X}\right]=\sum_{k=1}^\infty\lambda^k\frac{\mathbf{C}_k}{k!}
	\end{align*}
	The $k$-th moment is fully characterized by the first $k$ cumulants, for example, $\mathbf{C}_1=\mathbb{E}[X]$ and  $\mathbf{C}_2=\Var(X)$. In particular, a random variable is Gaussian $\mathcal{N}(\mu,\sigma^2)$ if and only if the first cumulant is $\mu$ and the second $\sigma^2$ and all cumulants of higher order are zero.

The important feature of orthogonal polynomial ensembles is that the cumulants for linear statistics with polynomial test functions can be expressed in terms of the Jacobi matrix associated with the orthogonality weight. This idea was originally introduced in \cite{breuer2017central} and further extended to extended processes in \cite{duits2018global}. It is not hard to see that the same is true for the generalization to \eqref{eq:qOPE}, but now for polynomials in $\mu_n$, as we will now explain. We will first discuss the one-dimensional ensembles of type \eqref{eq:qOPE} and then the extended processes of type \eqref{ensemble} later. We will only state the main observations, as the proofs follow almost verbatim. 

Let $J_n$ be the Jacobi matrix 
\begin{align}\label{Jacob}
		J_n \coloneqq \begin{pmatrix}b_{0,n} & a_{0,n} &  \\a_{0,n} & b_{1,n} & a_{1,n} & \\ & a_{1,n} & b_{2,n} & a_{2,n} & \\ &  & a_{2,n} & b_{3,n} & a_{3,n} & \\& & &\ddots &\ddots & \ddots \end{pmatrix},
	\end{align}
 where $a_{j,n}$ and $b_{j,n}$ are the recurrence coefficients from \eqref{recurrence} for the polynomials satisfying \eqref{eq:orthogonality}. We will also need the projection matrix $P_n$ that is defined by $(P_n)_{ij}=1$ for $1\leq i=j\leq n$ and $(P_n)_{ij}=0$ otherwise. The complementary projection is defined by $Q_n=I-P_n$.

 The heart of the matter is that for $X_n(p\circ \mu_n)$ , where $X_n$ is the linear statistic in \eqref{linearstat} with $f=p\circ \mu_n$ and $p$ polynomial, we have the following expression:
\begin{align}\label{eq:moment_gen_pre}
		\mathbb{E}[e^{\lambda X_n(p\circ \mu_n)}] =  \det \left(Q_n+P_ne^{\lambda p(J_n)}P_n\right).
\end{align}
Indeed, a proof of this claim follows immediately after replacing $x$ by $\mu_n(x)$ everywhere in the proof of \cite[Lemma 4.1]{duits2018global}.
 Then, by taking the log-derivatives with respect to $\lambda$, each cumulant can be expressed in terms of $J_n$ and $P_n$.  To start with, we have 
	\begin{align}\label{mean:1d}
		\mathbb{E}[X_n(p\circ \mu_n)] = & \Tr\left(P_np(J_n)P_n\right) \\
		\Var[X_n(p\circ\mu_n)]= & \frac{1}{2}\Tr\left(P_np(J_n)Q_np(J_n)P_n\right),\label{var:1d}
	\end{align}
and the general expression for the $k$-th cumulant is given by 
\begin{align} \label{cumulant:1d}
		\mathbf{C}_k(X_n(p\circ \mu_n)) =\sum_{j=1}^k\frac{(-1)^{j+1}}{j}
		\sum_{\substack{l_1+\ldots+l_j=k \\ l_i\geq 1}} \frac{1}{l_1!\cdots l_j !}\Tr P_n p(J_n)^{l_1}P_n\cdots P_n p(J_n)^{l_j}P_n .
\end{align}
Here, all indices are integers. Note that from \eqref{mean:1d} we see that the $\mathbb E[X_n(p\circ \mu_n)]$ depends on the recurrence coefficients $a_{j,n}$ and $b_{j,n}$ with $j/n \approx \xi$ with $\xi \in[0,1]$. This explains the conditions in Theorems~\ref{limit_density_poly}, \ref{thm:LLNcontinuousfunctions} and \ref{thm:LLNlpfunctions}. Moreover, the right-hand side of \eqref{var:1d}, and thus $\Var X_n(p\circ \mu_n)$, depends on only finitely many coefficients $a_{j,n}$ and $b_{j,n}$ with $j\approx n$. It is not obvious from \eqref{cumulant:1d}, but the same is true for the higher cumulants: the right-hand side of \eqref{cumulant:1d}, and thus each $\mathbf C_k(p\circ \mu_n)$, depends on only finitely many recurrence coefficients with indices around $n$. Indeed, this is proven even when $J_n$ is replaced by an arbitrary banded matrix in \cite{breuer2016universality}. 

It is important to note that \eqref{cumulant:1d} gives linear statistics with very special test functions. Indeed, the $\mu_n$ depends on $n$  (not only due to rescaling but also through possible parameters), and it is not obvious if one can extend the results to general $f$. Especially for the fluctuations of the linear statistic, we need good control over the convergence of $\mu_n$ to $\mu$, and this is the motivation for conditions \eqref{ass: mu smooth}.

The linear statistics for the extended process \eqref{ensemble} have similar expressions for the cumulants. Note that in that case we have $N$ families of recurrence coefficients $(a_{k,n}^{(m)})_k$ and $(b_{k,n}^{(m)})_k$ for $m=1,\ldots,N$, and we get $N$ Jacobi matrices $J^{(m)}_n$ for $m=1,\ldots,N$.  In addition, we have the coefficients $c_{k,m}$ and those we use to modify the Jacobi matrices as follows $$(\mathcal J_n^{(m)})_{k,l}=\frac{c_{k,m}}{c_{l,m}}({J}_n^{(m)})_{k,l}.$$
Then for any function $p(x,m)$ such that $x \to p(x,m)$ is a polynomial in $x$, we have
\begin{align} \label{cumulant:Nd}
		\mathbf{C}_k(X_n(p\circ \mu_n)) =\sum_{j=1}^k\frac{(-1)^{j+1}}{j}
		\sum_{\substack{l_1+\ldots +l_j=k \\ l_i\geq 1}
  }\sum_{\substack{r_{i}^{(m)} \geq 0\\ \sum_{m=1}^N r_{i}^{(m)}=l_i}} \frac{\Tr\left(\prod_{i=1}^j P_n\left(\prod_{m=1}^Np(m,\mathcal J_n^{(m)})^{r_i^{(m)}}\right)P_n\right)}{\prod_{m=1,\dots N} r_1^{(m)}! r_2^{(m)}! \cdots r_j^{(m)}!}.
\end{align}
Here, all indices are integers. Note that \eqref{cumulant:1d} is indeed a special case of \eqref{cumulant:Nd} after setting $N=1$.  Again, the proof of \eqref{cumulant:Nd} is, almost per verbatim, the same as  in \cite{duits2018global}, and is based on the following expression:
\begin{align}\label{eq:moment_gen}
		\mathbb{E}[e^{\lambda X_n(p\circ \mu_n)}] =  \det \left(Q_n+P_ne^{\lambda p(1,J_n^{(1)})}e^{\lambda p(2,J_n^{(2)})}\dots e^{\lambda p(N,J_n^{(N)})}P_n\right).
\end{align}
Then, \eqref{cumulant:Nd} follows from an expansion in $\lambda$ around $\lambda =0$ as given in \cite[Lemma 5.2]{duits2018global}.


	\subsection{Proof of Theorem~\ref{limit_density_poly}}
	
	  We say that a sequence $\{A^{(n)}\}_n$ of $\mathbb N \times \mathbb N$ matrices is \emph{slowly varying along the diagonals}, if there exists a function $a(\xi,z)$ analytic in $z$ on the unit circle and continuous in $\xi$ such that for any $j,k\in\mathbb{N}$ and any $\xi$, we have 
	\begin{align}\label{slowlivarying}
		\lim\limits_{\frac{j}{n},\frac{k}{n}\to\xi}\left(A^{(n)}\right)_{j,k} = \frac{1}{2\pi i}\oint a(\xi,z)\frac{dz}{z^{j-k+1}}.
	\end{align}
 It is important to note that the three diagonal Jacobi operator $J$, under the assumption \eqref{ass:slowly_varying_recurrence}, is a slowly varying matrix with  $a(\xi,z)=a(\xi)z+b(\xi)+a(\xi)z^{-1}$. The proof of Theorem~\ref{limit_density_poly} is, in fact, a special case of a more general result for matrices that are slowly varying along the diagonals. 
 
 We start with a convenient consequence of the slowly varying assumption, which is the uniform boundedness in the following sense. 
 
 \begin{lemma}\label{lemma:bounded}
     Let $A^{(n)}$ be the slowly varying operator satisfying \eqref{slowlivarying}. Assume its limit $a(\cdot,z)$ is continuous on $[0,1]$. Let $j\in\mathbb{N}$ such that $\lim\limits_{n\to\infty}\frac{j}{n}=\xi\in[0,1]$. Let $\Delta\in\mathbb{Z}$ be independent of $n$. Then we have 
     \begin{align}
    \sup\limits_{n\in\mathbb{N}}\sup\limits_{j\leq n+\sqrt{n}}\left|\left(A^{(n)}\right)_{j+\Delta,j}\right|<\infty .\label{eq:DCT bound}
     \end{align}
 \end{lemma}
 
	\begin{proof}
	    Suppose, on the contrary, \eqref{eq:DCT bound} does not hold. Then there exist two subsequences of $\mathbb{N}$, say $\{j_m\}_m$ and $\{n_m\}_m$ with $j_m\leq n_m+\sqrt{n_m}$ and $n_m\to\infty$ as $m\to\infty$ such that for all $m\in\mathbb{N}$ 
     \begin{align*}
\left|\left(A^{(n_m)}\right)_{j_m+\Delta,j_m}\right|>m.
     \end{align*}
    Here we also have $0 \leq \frac{j_m}{n_m}\leq 1+\frac{1}{\sqrt{n_m}}\leq 2$. Then, by compactness of the closed interval $[0,2]$, there exist another sub-sequence $\{m_k\}_k$ of $\mathbb{N}$ such that the limit exists as $\lim\limits_{k\to\infty}\frac{j_{m_k}}{n_{m_k}} = \xi\in[0,1]$. Now we have 
    \begin{align*}  
        \lim\limits_{\frac{j_{m_k}}{n_{m_k}}\to\xi} \left|\left(A^{(n_{m_k})}\right)_{j_{m_k}+\Delta,j_{m_k}}\right| = +\infty.
    \end{align*}
    Recall that $a(\cdot,z)$ is continuous on $[0,1]$ and hence bounded. This contradicts assumption \eqref{slowlivarying}. 
	\end{proof}
 
Recall from \eqref{mean:1d} that we have, for any polynomial $p$, that $\mathbb{E}[X_n(p\circ\mu_n)]=\Tr(P_np(J_n)P_n)$, where $p(J_n)$ forms a sequence of slowly varying linear operator. The following two results thus prove the convergence of the mean $\frac 1n \mathbb{E}[X_n(p\circ\mu_n)]$. We first deal with the case when $p(\mu)=\mu$ and then treat the case for general polynomials $p$.

	\begin{lemma}\label{lemma:diagonal}
		Let $\{A^{(n)}\}_n$ be a sequence of slowly varying operator satisfying \eqref{slowlivarying}. Then 
		\begin{align}
			\lim_{n\to\infty}\frac{1}{n}\Tr\left(P_nA^{(n)}P_n\right) = \frac{1}{2\pi i}\int^{1}_{0}\oint a(\xi,z)\frac{dz}{z}d\xi .
		\end{align}
	\end{lemma}

	\begin{proof}
 Note that by the definition of the floor function, we can rewrite the trace into the following
    \begin{align}
        \frac{1}{n}\Tr\left(P_nA^{(n)}P_n\right) = \int_0^1A^{(n)}_{\lfloor n\xi\rfloor, \lfloor n\xi\rfloor}d\xi
    \end{align}
    For any fixed $\xi\in[0,1]$ we have by the slowly varying assumption \eqref{slowlivarying}, 
    \begin{align*}
        \lim\limits_{n\to\infty} A^{(n)}_{\lfloor n\xi \rfloor, \lfloor n\xi\mathbb \rfloor} =  \frac{1}{2\pi i}\oint a(\xi,z)\frac{dz}{z}
    \end{align*}
    Now Lemma~\ref{lemma:bounded} shows that $\sup\limits_{n\in\mathbb{N}}\sup\limits_{\xi\in[0,1]}\left|A^{(n)}_{\lfloor n\xi\rfloor, \lfloor n\xi\rfloor}\right|<\infty$. Hence, the statement follows by the dominated convergence theorem.
	\end{proof}

	\begin{proposition}
		Let $\{A^{(n)}\}_n$ be a sequence of slowly varying operator satisfying \eqref{slowlivarying}. Suppose $A^{(n)}$ is a band matrix with a bandwidth independent of $n$. We have for any polynomial $p$, 
		\begin{align} \label{eq:tracepolyA}
			\lim_{n\to\infty}\frac{1}{n}\Tr\left(P_np(A^{(n)})P_n\right) = \frac{1}{2\pi i}\int^{1}_{0}\oint p(a(\xi,z))\frac{dz}{z}d\xi.
		\end{align}
	\end{proposition}
	\begin{proof}
		First, consider a monomial $p(x)=x^k$ for any integer $k$.  We have 
		\begin{align*}
			\left(\left(A^{(n)}\right)^k\right)_{m,m} = & \sum_{l_1,l_2,\dots,l_{k-1}}\left(A^{(n)}\right)_{m,l_1}\left(A^{(n)}\right)_{l_1,l_2}\dots \left(A^{(n)}\right)_{l_{k-2},l_{k-1}}\left(A^{(n)}\right)_{l_{k-1},m}.
		\end{align*}
		Since $A^{(n)}$ has fixed bandwidth, say $b$ which is independent of $n$, we have $(A^{(n)})_{j,k}=0$ for $|j-k|>b$. This means the non-zero entries in the sum must have indices $l_j$ such that $|l_j-m|\leq k b$ for all $j=1,2,\dots,k-1$. Hence, the total number of non-zero entries is independent of $n$. Consider $m$ such that $\frac{m}{n}\to\xi\in[0,1]$. Then we have $\frac{l_j}{n}\to\xi$ for those non-zero entries. By the slowly varying assumption \eqref{slowlivarying}, we have the following limits for the non-zero entries,
        \begin{align*}
            \lim_{\frac{m}{n}\to\xi} \left(A^{(n)}\right)_{m,l_1}= & \oint a(\xi,z)\frac{dz}{z^{m-l_1+1}} \\
            \lim_{\frac{m}{n}\to\xi}\left(A^{(n)}\right)_{l_{j},l_{j+1}} = & \oint a(\xi,z)\frac{dz}{z^{l_j-l_{j+1}+1}} \quad \text{ for all } j=1,2 \dots, k-2 \\
            \lim_{\frac{m}{n}\to\xi} \left(A^{(n)}\right)_{l_{k-1},m} = &  \oint a(\xi,z)\frac{dz}{z^{l_{k-1}-m+1}}.
        \end{align*}  
On the other hand, we expand $a(\xi,z)=\sum_{l'}\frac{1}{2\pi i}\oint a(\xi, z_{k})\frac{z^{l'}dz_{k}}{z_{k}^{l'+1}}$ and get
		\begin{align*}
			\frac{1}{2\pi i}\oint a(\xi, z_1)^k\frac{dz_1}{z_1} =  \sum_{l_1',l_2',\dots,l_{k-1}'}\frac{1}{(2\pi i)^k}\oint\dots\oint \frac{a(\xi,z_1)a(\xi,z_2)\dots a(\xi,z_{k-1})a(\xi,z_{k})}{z_1^{1-l_1'-l_2'-\dots - l_{k-1}'}z_2^{l_1'+1}\dots z_{k-1}^{l_{k-2}'+1}z_{k}^{l_{k-1}'+1}}dz_1dz_2\dots dz_k
		\end{align*}
		where we can reorder the indexes in the following way, i.e. setting  $l_{k-1}'=l_{k-1}-m$, and then recursively setting  $l_j'=l_j-l_{j+1}$ for all $j=1,\dots, k-2$. Clearly, $1-l_1'-l_2'-\dots  -l_{k-1}'=m-l_1+1$ and thus we have for all $\xi\in[0,1]$
		\begin{align*}
			\lim_{\frac{m}{n}\to\xi}\left(\left(A^{(n)}\right)^k\right)_{m,m}=\frac{1}{2\pi i}\oint a(\xi, z)^k\frac{dz}{z}
		\end{align*}
		the limit is uniform in $n$. This actually shows that the diagonal of $A^k$ is also slowly varying. Thus we can apply Lemma~\ref{lemma:diagonal} for $\text{diag}((A^{(n)})^k)$ and we have, 
		\begin{align*}
			\lim_{n\to\infty}\frac{1}{n}\Tr\left(P_n\left(A^{(n)}\right)^kP_n\right) = \frac{1}{2\pi i}\int^{1}_{0}\oint a(\xi,z)^k\frac{dz}{z}d\xi
		\end{align*}
		for any $k$. By linearity of trace and integral, the limit above also works for any $p$ being any polynomial. That is 
		\begin{align*}
			\lim_{n\to\infty}\frac{1}{n}\Tr\left(P_np(A^{(n)})P_n\right) = \frac{1}{2\pi i}\int^{1}_{0}\oint p(a(\xi,z))\frac{dz}{z}d\xi,
		\end{align*}
  and this proves the statement.
	\end{proof}

\begin{corollary}\label{mean_convergence}
	With $a(\xi,z)=a(\xi)z+b(\xi)+\frac{a(\xi)}{z}$, we have for any polynomial $p$
		\begin{align}
			\lim_{n\to\infty}\frac{1}{n}\Tr\left(P_np(A^{(n)})P_n\right) 
			= & \frac{1}{\pi }\int^{1}_{0}\int_{0}^{\pi} p(b(\xi)+2a(\xi)cos\theta)d\theta d\xi \label{eq:lln_2nd}\\
			= & \frac{1}{\pi }\int_{[0,1]\setminus a^{-1}(\{0\})}\int_{b(\xi)-2a(\xi)}^{b(\xi)+2a(\xi)} \frac{p(x)}{\sqrt{4a(\xi)^2-(x-b(\xi))^2}}dxd\xi \nonumber\\
            & +\int_{a^{-1}(\{0\})}p(b(\xi))d\xi .
		\end{align}
	\end{corollary}
	\begin{proof} The statement follows by using $z=e^{i\theta}$ and the change of variable $x=b(\xi)+2a(\xi)cos\theta$.  \end{proof}
	
	\begin{proof}[Proof of Theorem~\ref{limit_density_poly}]
		Replace $A^{(n)}$ by the Jacobi operator $J_n$. By Corollary~\ref{mean_convergence} we have
		\begin{align}\label{mean:conv}
			\mathbb{E}\left[\frac{1}{n}X_n(p\circ \mu_n)\right] \to  \frac{1}{\pi }\int^{1}_{0}\int_{b(\xi)-2a(\xi)}^{b(\xi)+2a(\xi)} \frac{p(x)}{\sqrt{4a(\xi)^2-(x-b(\xi))^2}}dxd\xi+\int_{a^{-1}(\{0\})}p(b(\xi))d\xi,
		\end{align}
  as $n \to \infty$. 
  The almost sure convergence can be obtained by estimating the variance  (as in  \cite{breuer2014nevai,hardy2018polynomial}).
		By Markov's inequality, we have for any $\delta>0$ 
		\begin{align*}
			\mathbb{P}\left(\left|\frac{1}{n}X_n(p\circ \mu_n)-\mathbb{E}\left[\frac{1}{n}X_n(p\circ \mu_n)\right]\right|>\delta\right)\leq \frac{1}{\delta^2n^2} \Var X_n(p\circ \mu_n) .
		\end{align*}
		Then, we need to estimate \eqref{var:1d}. Since $p$ is a polynomial and $J_n$ is a tri-diagonal matrix, $p(J_n)$ is a band matrix with bandwidth equal to the degree of $p$, which is finite and independent with $n$. Thus, $P_np(J_n)Q_np(J_n)P_n$ is of rank at most two times the degree of $p$. Also, by the slowly varying assumption, we have entries of $p(J_n)$ around $n$-th row and $n$-th column are uniformly bounded by a constant independent with $n$. That is there exists a constant $C>0$ such that
  \begin{align}
\left|\Tr\left(P_np(J_n)Q_np(J_n)P_n\right)\right| \leq \sum_{l_1=n-k+1}^{n}\sum_{l_2=n+1}^{n+k}\sum_{l_3=n-k+1}^{n}\left|\left(p(J_n)\right)_{l_1,l_2}\left(p(J_n)\right)_{l_2,l_3}\right|<C,
  \end{align}
where $k$ is the degree of the polynomial $p$. Therefore, 
		\begin{align*}
			\mathbb{P}\left(\left|\frac{1}{n}X_n(p\circ \mu_n)-\mathbb{E}\left[\frac{1}{n}X_n(p\circ \mu_n)\right]\right|>\delta\right)\leq \frac{C}{2\delta^2n^2}.
		\end{align*}
		Then combine this with \eqref{mean:conv} and Borel-Cantelli lemma, and we have the almost sure convergence.
  \end{proof}


	\subsection{Proof of Theorem~\ref{thm:LLNcontinuousfunctions}}
  For the proof of Theorem~\ref{thm:LLNcontinuousfunctions}, we need the following lemma. 
 
  \begin{lemma}
   Assume \eqref{eq:mu_mean_small}. Then we have almost surely 
		\begin{equation}
			\lim_{n\to\infty}\frac{1}{n}X_n(\left|p\circ \mu_n\right|\chi_{K^c} )= 0. \label{almost sure:poly_K}
		\end{equation}
 \end{lemma}
  \begin{proof} 
  Note that
		\begin{align*}
			\Var X_n(\left|p\circ \mu_n\right|\chi_{K^c} ) 
			\leq & 2n\mathbb{E}[X_n(\left|p\circ \mu_n\right|^2\chi_{K^c} )] - 2\mathbb{E}[X_n(\left|p\circ \mu_n\right|\chi_{K^c} )] ^2
		\end{align*} 
		so by assumption \eqref{eq:mu_mean_small} we have $\Var X_n(\left|p\circ \mu_n\right|\chi_{K^c} ) \leq Cn^{1-\epsilon}$ for some constant $C>0$ and $\epsilon>0$. Then, using the Markov inequality, we have 
		\begin{align*}
			\mathbb{P}\bigg[\bigg\vert\frac{1}{n}X_n(\left|p\circ \mu_n\right|\chi_{K^c} )-\mathbb{E}\bigg[\frac{1}{n}X_n(\left|p\circ \mu_n\right|\chi_{K^c} )\bigg]\bigg\vert>\delta\bigg]\leq \frac{C}{n^{1+\epsilon}}
		\end{align*}
		By the Borel-Contelli lemma, we obtain the statement.
	\end{proof}

	\begin{proof}[Proof of Theorem~\ref{thm:LLNcontinuousfunctions}]
		Note that polynomials are dense in $C(\mu(K))$. By the Stone-Weierstrass theorem for any function $f\circ\mu^{-1}$ continuous on compact subset $\mu(K)$ and any $\epsilon>0$ we can find a $p$ such that 
		\begin{align}\label{almost sure:X_n continuous}
			\sup_{x\in \mu(K)}|f\circ\mu^{-1}(x)-p(x)|<\epsilon.
		\end{align}
		Since $f$ is bounded by polynomial and by \eqref{almost sure:poly_K} we have 
		\begin{align}\label{almost sure:sigma continuous}
			\left|\frac{1}{n}X_n(f-p\circ\mu)\right|\leq \sup_{x\in K}|f(x)-p\circ\mu(x)|+o_n(1)\leq \epsilon+o_n(1) 
		\end{align}
		
		Using the triangle inequality, we obtain
		\begin{align*}
			& \left|\frac{1}{n}X_n(f)-\frac{1}{\pi }\int^{1}_{0}\int_{\mu^{-1}(b(\xi)-2a(\xi))}^{\mu^{-1}(b(\xi)+2a(\xi))} \frac{f(x)}{\sqrt{4a(\xi)^2-(\mu(x)-b(\xi))^2}}d\mu(x)d\xi-\int_{a^{-1}(\{0\})}f\circ\mu^{-1}(b(\xi))d\xi\right|\\
			\leq &\left|\frac{1}{n}X_n(f-p\circ \mu)\right|+\left|\frac{1}{n}X_n(p\circ \mu-p\circ \mu_n)\right| \\
			& +\left|\frac{1}{n}X_n(p\circ \mu_n)-\frac{1}{\pi }\int^{1}_{0}\int_{b(\xi)-2a(\xi)}^{b(\xi)+2a(\xi)} \frac{p(x)}{\sqrt{4a(\xi)^2-(x-b(\xi))^2}}dxd\xi\right|\\
			& +\left|\frac{1}{\pi }\int^{1}_{0}\int_{\mu^{-1}(b(\xi)-2a(\xi))}^{\mu^{-1}(b(\xi)+2a(\xi))} \frac{f(x)-p\circ\mu^{-1}(x)}{\sqrt{4a(\xi)^2-(\mu(x)-b(\xi))^2}}d\mu(x)d\xi\right|+\left|\int_{a^{-1}(\{0\})}f\circ\mu^{-1}(b(\xi))-p(b(\xi))d\xi\right|.
		\end{align*}
		Combining with the bounds \eqref{almost sure:X_n continuous} and \eqref{almost sure:sigma continuous} and by the uniform convergence of $\mu_n$ and almost sure convergence of linear statistics of polynomials  \eqref{eq:meanpolyGen}, we have  for all  $\epsilon>0$
		\begin{align*}
			\limsup_n	\left|\frac{1}{n}X_n(f)-\frac{1}{\pi }\int^{1}_{0}\int_{\mu^{-1}(b(\xi)-2a(\xi))}^{\mu^{-1}(b(\xi)+2a(\xi))} \frac{f(x)d\mu(x)d\xi}{\sqrt{4a(\xi)^2-(\mu(x)-b(\xi))^2}}-\int_{a^{-1}(\{0\})}f\circ\mu^{-1}(b(\xi))d\xi\right|\leq 3\epsilon
		\end{align*}
		almost surely. This shows the almost sure convergence of $\frac{1}{n}X_n(f)$ for any $f$ such that $f \circ \mu^{-1}$ is continuous. 
   \end{proof}


	\subsection{Proof of Theorem~\ref{thm:LLNlpfunctions}}

	\begin{proof}[Proof of Theorem~\ref{thm:LLNlpfunctions}]
	We will approximate the function $f$ by a function $p\circ \mu$ as follows:	First note that  $\mu$ monotone and thus Borel measurable and $d\mu^{-1}$ is a Radon measure on $\mu(I)\subset{\mathbb{R}}$.  Then since $f$ and $f\circ \mu^{-1}$ are Riemann integrable we also find that $f\in L^{\gamma }(\mu(K), d x+|d \mu^{-1}(x)|)$, with $\gamma>2$ as in \eqref{eq:conda}. Then continuous functions are dense in $L^{\gamma }(\mu(K), d x+|d \mu^{-1}(x)|)$, as a consequence of the Urysohn's lemma. For any $\epsilon>0$, we can thus find a continuous function $p$ such that 
\begin{align}\label{asconv:epsilonapprox}
			\int_{\mu(K)} \left| f\circ\mu^{-1}(x)-p(x) \right|^{\gamma} \left( d x+d |\mu^{-1}|(x) \right)<\epsilon .
		\end{align}
		
		Now by  H\"older's inequality with $\gamma>2,\eta<2$ such that  $\frac{1}{\gamma}+\frac{1}{\eta}=1$ we have that 
		\begin{multline*}
			\left|\frac{1}{\pi }\int^{1}_{0}\int_{\mu^{-1}(b(\xi)-2a(\xi))}^{\mu^{-1}(b(\xi)+2a(\xi))} \frac{f(x)-p\circ\mu(x)}{\sqrt{4a(\xi)^2-(\mu(x)-b(\xi))^2}}d\mu(x)d\xi\right|\\
			\leq  \|f\circ \mu^{-1}-p(x)\|_{L^\gamma(\mu(K),dx)}\frac{1}{\pi}\int_0^1 \left(\int_{b(\xi)-2a(\xi)}^{b(\xi)+2a(\xi)} \frac{1}{(4a(\xi)^2-(x-b(\xi))^2)^{\frac{\eta}{2}}}dx\right)^{\frac{1}{\eta}}d\xi .
		\end{multline*}
 Since
		\begin{multline*}
			\frac{1}{\pi}\int_0^1 \left(\int_{b(\xi)-2a(\xi)}^{b(\xi)+2a(\xi)} \frac{1}{(4a(\xi)^2-(x-b(\xi))^2)^{\frac{\eta}{2}}}dx\right)^{\frac{1}{\eta}}d\xi
			=\frac{1}{\pi}\int_0^1 |2a(\xi)|^{-\frac{1}{\gamma}}\left(\int_{-1}^{1} \frac{1}{(1-x^2)^{\frac{\eta}{2}}}dx\right)^{\frac{1}{\eta}}d\xi\\
			 = \frac{1}{\pi}\left(\frac{\sqrt{\pi } \Gamma \left(1-\frac{\eta }{2}\right)}{\Gamma \left(\frac{3}{2}-\frac{\eta }{2}\right)}\right)^{\frac{1}{\eta}}\int_0^1 |2a(\xi)|^{-\frac{1}{\gamma}}d\xi,
		\end{multline*}
		we find by  \eqref{eq:conda} we have the handy bound, 
		\begin{align}\label{alomst sure:limitbound}
			\left|\frac{1}{\pi }\int^{1}_{0}\int_{b(\xi)-2a(\xi)}^{b(\xi)+2a(\xi)} \frac{f\circ \mu^{-1}(x)-p(x)}{\sqrt{4a(\xi)^2-(x-b(\xi))^2}}dxd\xi\right|\leq  C_{\gamma} \|f-p\circ \mu \|_{L^{\gamma}(K,d\mu)},
		\end{align}
  for some constant $C_{\gamma}$.
		Now using the triangle inequality and \eqref{almost sure:poly_K}, we obtain
		\begin{multline*}
			 \left|\frac{1}{n}X_n(f)-\frac{1}{\pi }\int^{1}_{0}\int_{\mu^{-1}(b(\xi)-2a(\xi))}^{\mu^{-1}(b(\xi)+2a(\xi))} \frac{f(x)}{\sqrt{4a(\xi)^2-(\mu(x)-b(\xi))^2}}d\mu(x)d\xi\right|
			\leq \\\left|\frac{1}{n}X_n((f-p\circ \mu)\chi_K)\right|+\left|\frac{1}{n}X_n((p\circ \mu-p\circ \mu_n)\chi_K)\right| \\
			  +\left|\frac{1}{n}X_n(p\circ \mu_n)-\frac{1}{\pi }\int^{1}_{0}\int_{b(\xi)-2a(\xi)}^{b(\xi)+2a(\xi)} \frac{p(x)}{\sqrt{4a(\xi)^2-(x-b(\xi))^2}}dxd\xi\right|\\
			 +\left|\frac{1}{\pi }\int^{1}_{0}\int_{\mu^{-1}(b(\xi)-2a(\xi))}^{\mu^{-1}(b(\xi)+2a(\xi))} \frac{f(x)-p\circ\mu(x)}{\sqrt{4a(\xi)^2-(\mu(x)-b(\xi))^2}}d\mu(x)d\xi\right| +o_n(1) .
		\end{multline*} 
Considering that the points in $\frac{1}{n}\mathbb{Z}$ are equally spaced and $f$ is Riemann integrable on $K$ , we have $$\limsup_{n \to \infty}\frac{1}{n}\left|X_n((f-p\circ\mu)\chi_K)\right|\leq \limsup_{n \to \infty}\frac{1}{n}\sum_{x\in K \cap \frac{1}{n}\mathbb Z}\left|f(x)-p(\mu(x))\right|\leq\int_{\mu(K)}|f\circ\mu^{-1}(x)-p(x)|d|\mu^{-1}|(x)<\epsilon.$$	
  By the uniform convergence of $\mu_n$ we also have $\left|\frac{1}{n}X_n((p\circ \mu-p\circ \mu_n)\chi_K)\right|=o_n(1)$.  Now combine Theorem~\ref{thm:LLNcontinuousfunctions}  and \eqref{alomst sure:limitbound}, and we have for any $ \epsilon>0$, 
	\begin{align*}
	\limsup_n	\left|\frac{1}{n}X_n(f)-\frac{1}{\pi }\int^{1}_{0}\int_{b(\xi)-2a(\xi)}^{b(\xi)+2a(\xi)} \frac{f\circ\mu^{-1}(x)}{\sqrt{4a(\xi)^2-(x-b(\xi))^2}}dxd\xi\right|\leq (1+C_\gamma) \epsilon
	\end{align*}
  almost surely. This shows the almost sure convergence of $\frac{1}{n}X_n(f)$.  \end{proof}

 
\subsection{Proof of Theorems~\ref{CLT_q_polynomial} and \ref{thm:fluctuation_theorem}}
We now prove Theorems~\ref{CLT_q_polynomial}  and \ref{thm:fluctuation_theorem}. Since Theorem~\ref{CLT_q_polynomial} is a special case of Theorem~\ref{thm:fluctuation_theorem} (set $N=1$), it suffices to prove the latter. The proofs will follow the same structure as \cite{duits2018global}. 

In this subsection, we only prove the case for the time sequence $$t_m=m,\qquad \text{ for all } m=0,\dots, N$$ to shorten the notations in the proof. However, by relabelling the time sequence, the general statement of Theorem~\ref{thm:fluctuation_theorem} follows directly from the same arguments. 

We start by introducing some notation (recall \eqref{cumulant:Nd}): For $\mathbb N \times \mathbb N$ matrices  $A_1,\ldots,A_N$ and $k \in \mathbb N$ we define 
\begin{align} \label{cumulant:opA}
		{C}_k(A_1,\ldots,A_N) =\sum_{j=1}^k\frac{(-1)^{j+1}}{j}
		\sum_{\substack{l_1+\ldots +l_j=k \\ l_i\geq 1}
  }\sum_{\substack{r_{i,m} \geq 0\\ \sum_{m=1}^N r_{i}^{(m)}=l_i}} \frac{\Tr\left(\prod_{i=1}^j P_n\left(\prod_{m=1}^NA_m^{r_i^{(m)}}\right)P_n\right)}{\prod_{m=1,\dots N} r_1^{(m)}! r_2^{(m)}! \cdots r_j^{(m)}!}.
\end{align}
Then, the following is the first key observation in \cite{duits2018global}.
	\begin{lemma}[Corollary 5.6 in \cite{duits2018global}]\label{cumulant_comparision}
		For any $A_1, \dots, A_N$ and $T_1, \dots, T_N$ banded semi-infinite matrices with bandwidth $b$ we have the following bound for all $k>1$
		\begin{align}
			& |C_k(A_1, \dots, A_N)-C_k(T_1, \dots, T_N)| \nonumber\\
			& \leq 4c_0b(k+1)\left(2 \max\left(\sum_{m=1}^N\|A_m\|_\infty,\sum_{m=1}^N\|T_m\|_\infty\right)\right)^{k-1}\sum_{m=1}^N\|R_{n,2b(k+1)}(A_m-T_m)R_{n,2b(k+1)}\|_\infty
		\end{align}
		where $c_0=\frac{2e}{(2-\sqrt{e})^2}$, $R_{n,l}=P_{n+l}-P_{n-l}$ and $\|\cdot\|_\infty$ denotes the operator norm. 
	\end{lemma}
	The lemma is a comparison result : if the  matrices $A_j$ and $T_j$ are asymptotically the same around the $nn$ entry as $n \to \infty$, more precisely $\lim_{n\to\infty}\|R_{n,2b(k+1)}(A_m-T_m)R_{n,2b(k+1)}\|_\infty=0$, then $C_k(A_1, \dots, A_N)$ and $C_k(T_1, \dots, T_N)$ have the same limit (if it exists).   So, if one is able to compute the limiting behavior for a favorite set of matrices, then one immediately obtains the same limit for any set of matrices that can be compared to that set. Given \eqref{cumulant:Nd}, we want to apply this principle to $A_m=p(m,J_m)$, which by the assumptions can be compared to Toeplitz matrices. For Toeplitz matrices, we have the following result:
	\begin{lemma}[Corollary 5.8 in \cite{duits2018global}] \label{cumulanttoep} Let 
		\begin{align*}
			(\tilde{T}_{m})_{k,l} = \begin{cases}
				e^{\tau_m(l-k)}b(1;m) \quad & l=k\\
				e^{\tau_m(l-k)}a(1;m) \quad & |l-k|=1\\
				0 \quad & otherwise,
			\end{cases}
		\end{align*} 
  Then
  \begin{align}\label{thm:Toeplitz_clt}
			\lim_{n\to\infty} C_2(T_1,\ldots,T_N)=\sum_{r_1,r_2=1}^N\sum_{k=1}^\infty ke^{-|\tau_{r_1}-\tau_{r_2}|k}\hat{p}^{(r_1)}_k\hat{p}^{(r_2)}_{-k}
		\end{align}
		where 
		\begin{align*}
			\hat{p}^{(r)}_k=\frac{1}{\pi}\int_0^\pi p(r,b(1;r)+2a(1;r)\cos\theta)\cos(k\theta)d\theta
		\end{align*}
  and $\lim_{n \to \infty} C_k(T_1,\ldots,T_N)=0$ for $k\geq 3$.
	\end{lemma}
	
	\begin{corollary}
		For $\lambda>0$ small,
		\begin{align}\label{clt:polyfourier}
			\lim_{n\to\infty}\mathbb{E}[e^{i \lambda X_n(p\circ \mu_n)}]e^{-i \lambda \mathbb{E}[X_n(p\circ \mu_n)]}=e^{-\frac{\lambda^2}{2}\sigma^2(p)}
		\end{align}
		where 
		\begin{align} \label{eq:sigmap2}
			\sigma^2(p) & = \sum_{r_1,r_2=1}^N\sum_{k=1}^\infty ke^{-|\tau_{r_1}-\tau_{r_2}|k}\widehat{p}^{(r_1)}_k\widehat{p}^{(r_2)}_{-k}
		\end{align}
	\end{corollary}
	\begin{proof} 
		Let $A_m= p(m,J_m)$ and recall \eqref{cumulant:Nd}. Then $\lim_{n\to\infty}\|R_{n,2b(k+1)}(A_m-T_m)R_{n,2b(k+1)}\|_\infty=0$, by the assumptions \eqref{ass:cltjaccobi} and \eqref{ass:cltc}. The combination of Lemmas~\ref{cumulant_comparision} and \ref{cumulanttoep} give that $\mathbf C_k(X_n(p\circ \mu_n)) \to 0$ for $k\geq 3$ and $\mathbf C_k(X_n(p\circ \mu_n))=\sigma^2(p)$ as in \eqref{eq:sigmap2}. \end{proof}

	In the rest of this section, we extend $f$ from being a polynomial in $\mu_n $ to a $C^1$-function in $\mu_n$. First, we estimate the variance in the limit. Though we require the function to be $C^1$, the following lemma works for Lipschitz continuous functions. 
	
	\begin{lemma} \label{lemma: Krein}
		For any $x\mapsto f(m,x),  f_1(m,x),  f_2(m,x)$ being real valued $2\pi$ periodic Lipschitz continuous functions. 
		Denote the Fourier coefficient $\hat{f}^{(m)}_k=\frac{1}{2\pi}\int_{-\pi}^{\pi}f(m,x)e^{-ikx}dx$ and the quantity
		\begin{align}\label{pseudovariance}
			\sigma^2(f)  = \sum_{r_1,r_2=1}^N\sum_{k=1}^\infty ke^{-|\tau_{r_1}-\tau_{r_2}|k}\hat{f}^{(r_1)}_k\hat{f}^{(r_2)}_{-k}.
		\end{align}
		Let $c_{Lip}=2N\pi^3$. Then we have the following bounds
		\begin{align}
			& \sigma^2(f) \leq  c_{Lip}\sum_{m=1}^NL(f(m,\cdot))^2, \label{pseudovariabcebound1}\\
			|\sigma^2(f_1)-\sigma^2(f_2)|\leq &   c_{Lip}\sum_{m=1}^N\left(L\left(f_1(m,\cdot)\right)+L\left(f_2(m,\cdot)\right)\right)\sum_{m=1}^NL\left(f_1(m,\cdot)-f_2(m,\cdot)\right).\label{pseudovariabcebound2}
		\end{align}
	\end{lemma}

	\begin{proof}
		By either inverse Fourier transformation of $e^{-k|\tau|}$ or residue theorem we have the identity
		\begin{align*}
			e^{-k|\tau|}=\frac{1}{\pi}\int_{-\infty}^\infty \frac{k}{k^2+\omega^2}e^{-i\omega \tau} d\omega.
		\end{align*}
		
		Now, \eqref{pseudovariance} can be rewritten as 
		\begin{align*}
			\sigma^2(f) =\frac{1}{\pi}\sum_{k=1}^\infty\int_{\mathbb{R}}\frac{k^2}{\omega^2+k^2}\left|\sum_{m=1}^Ne^{-i\tau_{m}\omega}\hat{f}_k^{(m)}\right|^2d\omega.
		\end{align*}
		By the Cauchy-Schwartz inequality for the sum over $m$ and $\int_{\mathbb{R}}\frac{k}{k^2+\omega^2}d\omega=\pi$, we have 
		\begin{align*}
			\sigma^2(f)  \leq N\sum_{m=1}^N\sum_{k=1}^\infty k\left|\hat{f}^{(m)}_k\right|^2.
		\end{align*}
		For all $m$ consider $f_{\delta}(m,x)\coloneqq f(m,x-\delta)$. Thus $\widehat{f_{\delta}}^{(m)}_k=e^{-ik\delta}\hat{f}^{(m)}_k$. By Parseval's identity, we have
		\begin{align*}
			\|f(m,\cdot)-f_{\delta}(m,\cdot)\|_2^2=\sum_{k=-\infty}^\infty|e^{-ik\delta}-1|^2\left|\hat{f}^{(m)}_k\right|^2.
		\end{align*}
		Now fix an integer $l$ and pick $\delta=\frac{\pi}{2^{l+1}}$ we have 
		\begin{align*}
        \sum_{k=2^l}^{2^{l+1}-1}\frac{k\delta}{\pi}\left|\hat{f}^{(m)}\right|^2\leq\sum_{k=2^l}^{2^{l+1}-1}|e^{-ik\delta}-1|^2\left|\hat{f}^{(m)}\right|^2\leq 	\|f(m,\cdot)-f(m,\cdot)_{\delta}\|_2^2\leq2\pi L\left(f(m,\cdot)\right)^2\delta^{2}, 
		\end{align*}
		where the first inequality is due to the fact that $k\delta\in[\frac{\pi}{2},\pi)$ in the chosen range and hence  the real part of $e^{-ik\delta}$ is non positive.  Consequently, 
		\begin{align*}
			\sum_{k=1}^\infty k\left|\hat{f}_k^{(m)}\right|^2=\sum_{l=0}^\infty\sum_{k=2^l}^{2^{l+1}-1}k\left|\hat{f}^{(m)}\right|^2\leq 2\pi^{3}L\left(f(m,\cdot)\right)^2,
		\end{align*}
		which proves inequality \eqref{pseudovariabcebound1}. To prove inequality \eqref{pseudovariabcebound2} we use the relation $|a|^2-|b|^2=a(\bar{a}-\bar{b})+\bar{b}(a-b)$ where $\bar{a}$ denotes the complex conjugation of $a$. This gives
		\begin{align*}
			\sigma^2(f_1)-\sigma^2(f_2)= & \frac{1}{\pi}\sum_{k=1}^\infty\int_{\mathbb{R}}\frac{k^2}{\omega^2+k^2}\left(\sum_{m=1}^Ne^{-i\tau_{m}\omega}\hat{f}_{1,k}^{(m)}\sum_{m=1}^N(e^{i\tau_{m}\omega}\hat{f}_{1,-k}^{(m)}-e^{i\tau_{m}\omega}\hat{f}_{2,-k}^{(m)})\right)d\omega \\
			& +\frac{1}{\pi}\sum_{k=1}^\infty\int_{\mathbb{R}}\frac{k^2}{\omega^2+k^2}\left(\sum_{m=1}^Ne^{i\tau_{m}\omega}\hat{f}_{2,-k}^{(m)}\sum_{m=1}^N(e^{-i\tau_{m}\omega}\hat{f}_{1,k}^{(m)}-e^{-i\tau_{m}\omega}\hat{f}_{2,k}^{(m)})\right)d\omega .
		\end{align*}
		Then, applying Cauchy-Schwartz inequality twice, following the same argument as above, we obtain
		\begin{multline*}
			 \left|\frac{1}{\pi}\sum_{k=1}^\infty\int_{\mathbb{R}}\frac{k^2}{\omega^2+k^2}\left(\sum_{m=1}^Ne^{-i\tau_{m}\omega}\hat{f}_{1,k}^{(m)}\sum_{m=1}^N e^{i\tau_{m}\omega}(\hat{f}_{1,-k}^{(m)}-\hat{f}_{2,-k}^{(m)})\right)d\omega\right| \\
			\leq  N\sum_{k=1}^\infty \sqrt{\sum_{m=1}^Nk\left|\hat{f}_{1,k}^{(m)}\right|^2\sum_{m=1}^Nk\left|\hat{f}_{1,-k}^{(m)}-\hat{f}_{2,-k}^{(m)})\right|^2}	
			\leq  N \sqrt{\sum_{m=1}^N\sum_{k=1}^\infty k\left|\hat{f}_{1,k}^{(m)}\right|^2}\sqrt{\sum_{m=1}^N\sum_{k=1}^\infty k\left|\hat{f}_{1,-k}^{(m)}-\hat{f}_{2,-k}^{(m)})\right|^2}\\	
			\leq  c_{Lip} \sqrt{\sum_{m=1}^NL\left(f_1(m,\cdot)\right)^2}\sqrt{\sum_{m=1}^NL\left(f_1(m,\cdot)-f_2(m,\cdot)\right)^2}.
		\end{multline*}  
		Similarly, we bound the second sum and apply the fact that for any positive number $x_n$, $\sum_{m=1}^Nx_m^2\leq (\sum_{m=1}^Nx_m)^2 $. This concludes the proof.
	\end{proof}
	
	\begin{lemma}\label{lemma:lipschtiz_poly} Assume that $\mu_n$ converges uniformly to $\mu$ uniformly  as in \eqref{ass:mu_sup}, and also assume that \eqref{ass:mu_lip} and \eqref{ass:mu} hold. Let $K^{(m)}$ be a set as described in Theorem~\ref{thm:fluctuation_theorem}. Then for any polynomial $x\mapsto p(m,x)$ following holds
		\begin{align}
			\lim_{n\to\infty}L_{K^{(m)}}\left(p(m, \mu \circ\mu_n^{-1,(m)})-p(m, \cdot)\right) =0, 
		\end{align}
		where $\mu \circ\mu_n^{-1,(m)}\coloneqq \mu (m,\mu_n^{-1,(m)})$ and $\mu_n^{-1,(m)}$ is the inverse of $x\mapsto \mu_n(m,x)$.
	\end{lemma}
	\begin{proof}
		Recall that $L$ defined by  \eqref{def: lipschitz_seminorm} is a seminorm. Let $f_n\coloneqq\mu \circ\mu_n^{-1,(m)}$  and $f_0(x)\coloneqq x$. It is then sufficient to show that for any $k\in \mathbb{N}$, we have  $\lim\limits_{n\to\infty}L_{K^{(m)}}\big(f_n^k-f_0^k \big)=0$. First note that 
		\begin{align}\label{eq:mu_poly}
			f_n^k-f_0^k= \left(f_n-f_0 \right)\left( \sum_{j=0}^{k-1} 	f_n^jf_0^{k-1-j}\right).
		\end{align}
		Also, it is easy to verify that for any two functions $f,g$, the following general statement about the Lipschitz constant holds 
		\begin{align}
			L(fg) \leq & \sup|f| L(g)+\sup|g|L(f)\label{eq:lipschitz1}.
		\end{align}
		Moreover, by \eqref{eq:lipschitz1} and induction one can also verify that for any $j,k \in \mathbb{N}_{\geq 1}$,
		\begin{align}
			L(f^k)\leq & k \sup|f|^{k-1}L(f) \label{eq:lipschitz2}\\
			L(f^kg^j)\leq & k \sup|f|^{k-1}\sup|g|^jL(f) + j\sup|g|^{j-1}\sup|f|^kL(g) .\label{eq:lipschitz3}
		\end{align}
		By assumption \eqref{ass:mu}, we can assume without loss of generality that there exists a constant $C>0$ with
		\begin{align}
			\sup_{x\in \frac{1}{n}\mathbb{Z} \cap K^{(m)}}|f_n|<C , & \quad L_{K^{(m)}}(f_n)<C .
		\end{align}  
		Then  by \eqref{eq:lipschitz2} and \eqref{eq:lipschitz3} we have
		\begin{align}
			L_{K^{(m)}} \left( \sum_{j=0}^{k-1} 	f_n^jf_0^{k-1-j} \right) \leq k(k-1)C^{k-1}.
		\end{align}
		Use \eqref{eq:lipschitz1} for \eqref{eq:mu_poly} and we get 
		\begin{align*}
			L_{K^{(m)}}\big( f_n^k-f_0^k \big) \leq &  \sup_{x\in \frac{1}{n}\mathbb{Z} \cap K^{(m)}}\left| f_n-f_0 \right|  k(k-1)C^{k-1} + L_{K^{(m)}}\big( f_n-f_0\big) kC^{k-1}.
		\end{align*}
		By assumption \eqref{ass:mu_sup} and \eqref{ass:mu_lip}, we then have $\lim\limits_{n\to\infty}	L_{K^{(m)}}\big( f_n^k-f_0^k \big)=0$.
	\end{proof}

	\begin{proof}[Proof of Theorem~\ref{thm:fluctuation_theorem}]
		
		We use the Stone–Weierstrass theorem to extend $f$ to any $C^1$ continuous function by approximating $f$ uniformly by polynomials. First, we recall some notations. $\mu^{-1,(m)}$ is  the inverse function of $x\mapsto\mu(m,x)$, $f\circ\mu^{-1}(m,x)=f(m,\mu^{-1,(m)}(x))$ and
		\begin{align*}
			\widehat{f\circ\mu^{-1}}_k^{(m)} & \equiv \frac{1}{2\pi}\int_0^{2\pi}f(m,\mu^{-1}(2a(1;m) \cos\theta+b(1;m))e^{-ik\theta}d\theta\,\\
			\sigma^2(f\circ\mu^{-1}) & = \sum_{r_1,r_2=1}^N\sum_{k=1}^\infty ke^{-|\tau_{r_1}-\tau_{r_2}|k}\widehat{f\circ \mu^{-1}}^{(r_1)}_k\widehat{f\circ \mu^{-1}}^{(r_2)}_k.
		\end{align*}
		With the same proof as Lemma 6.2 in \cite{duits2018global}, one can obtain the following bound 		
		\begin{align}
			\Var X_n(f)\leq \frac{N}{2}\sum_{m=1}^N\sup_{x,y\in \frac{1}{n}\mathbb{Z} \cap K^{(m)}}\left|\frac{f(m,x)-f(m,y)}{\mu_n(x)-\mu_n(y)}\right|^2\|[P_n,J_m]\|_2^2+o_n(1).
		\end{align}
		Moreover $\|[P_n,J_m]\|_2\leq |a_{n,n}^m|(|\frac{c_{n,m}}{c_{n+1,m}}|+|\frac{c_{n+1,m}}{c_{n,m}}|)$ which is bounded uniformly with respect to $n$ by the assumptions (\ref{ass:cltjaccobi},\ref{ass:cltc}). Combining \eqref{ass:mu_lip}, 
		
		\begin{align*}
			\sup_{x,y\in \frac{1}{n}\mathbb{Z} \cap K^{(m)}}\left|\frac{f(m,x)-f(m,y)}{\mu_n(x)-\mu_n(y)}\right|\leq & \sup_{x,y\in K^{(m)}}\left|\frac{f(m,x)-f(m,y)}{\mu(x)-\mu(y)}\right|\sup_{x,y\in \frac{1}{n}\mathbb{Z} \cap K^{(m)}}\left|\frac{\mu(x)-\mu(y)}{\mu_n(x)-\mu_n(y)}\right| \\
			\leq & L_{K^{(m)}}\left(f(m,\mu^{-1,(m)})\right)(1+o_n(1)).
		\end{align*}
		Therefore there exists a constant $c(N)$ only depends on $N$ such that 
		\begin{align}\label{varapprox2}
			\Var X_n(f)\leq & c(N)\sum_m L_{K^{(m)}}\left(f\circ\mu^{-1}(m,\cdot)\right)^2
		\end{align}
		by the inequality $|1-e^{ix}|\leq |x|$ for all real $x$, we have $|e^{ix}-e^{iy}|\leq|x-y|$ for any $x,y$ real. Using the Jensen's inequality, we have the following bound for any $f_1,f_2$ real-valued function and $\lambda$ real
		\begin{align}\label{varapprox3}
			\left|\mathbb{E}\left[e^{i \lambda X_n(f_1)}\right]e^{-i \lambda \mathbb{E}[X_n(f_1)]}-\mathbb{E}\left[e^{i \lambda X_n(f_2)}\right]e^{-i \lambda\mathbb{E}[X_n(f_2)]}\right|\leq \Var[X_n(f_1-f_2)]
		\end{align}
		
		Now, for any $f\circ\mu^{-1}(m,\cdot)$ being an $C^1$ function and for any $\epsilon>0$,  by the Stone–Weierstrass theorem, there exists a polynomial $p(m,\mu)$ in $\mu$ such that 
		\begin{align*}
			\max_{m=0,\dots,N}L_{K^{(m)}}\left(f\circ\mu^{-1}(m,\cdot)-p(m,\cdot)\right)^2<\epsilon.
		\end{align*}
		Thus Lemma~\ref{lemma: Krein} implies
		\begin{align}
			|\sigma^2(f\circ\mu^{-1})-\sigma^2(p)|\leq &   c(N)\sqrt{\epsilon}, \label{varapprox1}
		\end{align}
		for some constant $c(N)>0$ depending only on $N$. Then starting with the triangle inequality and applying \eqref{clt:polyfourier},\eqref{varapprox1}, \eqref{varapprox2} and \eqref{varapprox3}, we get
		\begin{align*}
			\left|\mathbb{E}[e^{i \lambda X_n(f)}]e^{-i \lambda \mathbb{E}[X_n(f)]}-e^{-\frac{\lambda^2}{2}\sigma_\mu^2(f)}\right|\leq &  \lambda^2\Var[X_n(f-p(\mu))]+\lambda^2\Var[X_n(p(\mu)-p(\mu_n))] \\
			& +\left|\mathbb{E}[e^{i \lambda X_n(p\circ \mu_n)}]e^{-i \lambda \mathbb{E}[X_n(p\circ \mu_n)]}-e^{-\frac{\lambda^2}{2}\sigma^2(p)}\right|+\left|e^{\frac{-\lambda^2}{2}\sigma^2(p)}-e^{\frac{-\lambda^2}{2}\sigma^2(f\circ\mu^{-1})} \right| \\
			\leq & c'(N)\epsilon + o_n(1),
		\end{align*}
		where $c'(N)>0$ is  a constant only depends on $N$, $f$ and $\mu$ and independent with $\epsilon$ and $n$. This shows that for $f\circ\mu^{-1}(m,\cdot)$ being a $C^1$ function, we have 
		\begin{align*}
			X_n(f)-\mathbb{E}[X_n(f)]\to\mathcal{N}\left(0,\sigma^2(f\circ\mu^{-1})\right),
		\end{align*}
   in distribution as $n\to \infty$.
	\end{proof}


 \section{Variational problem} \label{sec:variational}

It is well-established that random tilings of planar domains concentrate near a deterministic limit shape when the size of the domain grows large. For uniformly distributed random tilings, and more generally for tilings with periodic weight structure, there have been important works on characterizing the limit shape in terms of a variational problem, in particular \cite{CKP, CLP,DeS10, KeOk07,KOS,Ku}. Furthermore, a complete characterization of its regularity properties was achieved in \cite{ADPZ}. Random tilings from probability measures different from the uniform measure or measures arising from a periodic weight structure are much less understood and this is an interesting open problem. As we will discuss below, non-uniform weights can give rise to variational problems whose Euler-Lagrange equation is non-homogeneous, and whose regularity theory is much more complicated than for the homogeneous one.  The $q$-Racah weighting of the hexagon is an interesting case study from this point of view. Furthermore we associate a natural complex structure to the variational problem using a Beltrami equation approach, extending results in \cite{ADPZ}. In section \ref{subsection:ComplexS} we then verify that this complex structure is the same as the one occurring in the proof of the Gaussian free field in Theorem \ref{thm:GFF}. In addition we relate the complex structure to the complex slope solving a complex Burger equation. While the complex structure in the $q$-Racah weighting of the hexagon and similar trapezoid domains has also been studied in the works \cite{dimitrov2019log,gorin2022dynamical}, it has not before been connected to the variational problem.


\subsection{Variational problem for $q$-Racah weighted lozenge tilings of the hexagon}

We start with some notation. The independent variables will be relabeled according to 
\begin{align*}
t\mapsto x, \,\,\, x\mapsto y.
\end{align*}
The reason for this is twofold. Firstly, for general domains $t$ can no longer be interpreted as a time evolution parameter and secondly we want to use the same notation as in \cite{ADPZ} to facilitate comparison.
The height function $h$ defined previously by \eqref{eq:heightfunction} is not the same as the standard height function for the lozenges tiles. They are related through the formula 
\begin{align}
\label{eq:height_function_relation}
h_{tiles}(x,y)=-h(x,y)+y.
\end{align}
In what follows we will only consider the height function for the tiles which is the one used in the variational problem. Furthermore we will omit the subscript $h_{tiles}$. 
In particular we have the relation 
\begin{equation}\label{eq:GradDensity}
\nabla h(x,y)=\rho_{III}(y,x)e_1+\rho_{II}(y,x)e_2.
\end{equation}
Let $N\subset \mathbb C$ the triangle in the complex plane with corners at $0,1$ and $i$ (or any other triangle). Let $ P \subset \R^2\cong \C$ be the regular hexagon (or any  Lipschitz domain).  Define the space of admissible Lipschitz functions as
\begin{align}
	\mathscr{A}_N(P):=\{u\in C^{0,1}(\overline{P}): \nabla u(z)\in N \text{ for a.e. $z\in P$}\}
\end{align}
and 
\begin{align}
	\mathscr{A}_N(P,u_0):=\{u\in \mathscr{A}_N(P): u(z)=u_0(z) \text{ for  $z\in \dv P$ and $u_0\in \mathscr{A}_N(P)$}\}.
\end{align} 
The $u_0$ will be our boundary condition. 

For $u \in  \mathscr A(P,N)$ we define the energy functional by 
\begin{align}\label{Func1}
	I[u]=\int_{P}\sigma(\nabla u(z))dz+\int_{P}w(z)u(z)dz
\end{align}
where
\begin{align} \label{sigmaL}
	\sigma(s,t) =-\frac{1}{\pi^2}(\mathscr{L}(\pi s)+\mathscr{L}(\pi t)+\mathscr{L}(\pi (1-s-t))),  
\end{align}
and where 
\begin{align*} 
	\mathscr{L}(\theta)=-\int_{0}^{\theta}\log \vert 2\sin x\vert dx
\end{align*}
is the Lobachevsky function. Furthermore, the weight function arising from the non-local weight structure of the probability measure is given by
\begin{align}
	w(z)&=\ln(q)\frac{\kappa^2q^{-S+2y-x}+1}{\kappa^2q^{-S+2y-x}-1} \label{eq:Weight}
\end{align}
where $S,q$ and $\kappa$ are parameters. In \cite[Thm. 3.3]{MPT22Var} it is proven that for all $\eps>0$ the random height function associated to a tiling chosen with probability measure \eqref{eq:ellipticloz} and boundary value $u_n$ converging to $u_0$ stays with in $\eps$ of the minimizer $h$ of \eqref{Func1} with probability 1 as $n\to \infty$. Note however that the functional in \cite{MPT22Var} is $-$ \eqref{Func1} up to a constant, as can be seen integration by parts.

Let $\ell$ denote the line
\begin{align*}
\ell:=\{(x,y)\in \R^2:2y-x=S-2\ln(\kappa)(\ln(q))^{-1}\}
\end{align*}
Then $w\in C^\infty(\R^2\setminus \ell)$. Note also that $w\notin L^1(\R^2)$. In particular we will assume that $\ell\cap P=\varnothing$. It may however happen that $\dv P\cap \ell \neq \varnothing$. In this case $\ell$ will intersect either the corner $(1,0)$ or the corner $(1,2)$ of $P$, see for example figures \ref{Fig: density} (c) and (d). These singular cases occur precisely when either $S-2\ln(\kappa)(\ln(q))^{-1}=-1$ or $S-2\ln(\kappa)(\ln(q))^{-1}=3$. Furthermore, while $w\notin L^1(\R^2)$, $w\in L^1(P)$ in the singular cases, so the variational problem is well-defined.

It is important to note that the functional \eqref{Func1} is convex on $\mathscr{A}_N(P,u_0)$ but not \emph{strictly convex}. This is due to the fact that while $\sigma$ is strictly convex in the interior of $N$, $\sigma$ is affine on the boundary of $\dv N$, (in fact 0) and so is not strictly convex on $N$. Furthermore, the functional 
\begin{align*}
	u\mapsto \int_{P}w(z)u(z)dz
\end{align*}
is convex, but not strictly convex. Finally, since $\sigma\notin C^{1}(N)$, the functional \eqref{Func1} is \emph{not differentiable} (in either the Fréchet or Gâteaux sense). In particular 
$
|\nabla\sigma(p)\vert=+\infty,	 
$ for $ p \to \partial N$, 
which   causes significant challenges in studying the variational problem \ref{Func1}. Nonetheless, one can argue as in \cite{ADPZ} or \cite{DeS10} to show that \eqref{Func1} has a unique minimiser on $\mathscr{A}_N(P,u_0)$.

Having the minimizer at hand, the next question is about  its properties. In particular, how it divides the hexagon into disordered and frozen regions, as well as the regularity properties of the boundary of the disordered region.  For the case $w=0$ this has been carried out in full detail in \cite{ADPZ}, but the case $w\neq 0$ is still open. In fact, the case $w\neq 0$ can be expected to be significantly harder, as it is known that new effects will appear, in particular with regard to the geometry of the disordered region with the appearance of inflexion points and corners. In the rest of this section, inspired by the analysis in \cite{ADPZ}, we will make the first steps in the analysis of the variational problem and outline some of the difficulties. In particular, we will prove the equivalence of the Euler-Lagrange equation and a non-homogeneous Beltrami equation which we believe will be important ingredients in the analysis. Using the results on orthogonal polynomial ensembles we will then verify that these properties hold for the limiting behaviour of the height function, thus establishing important evidence in support of our methods.

\subsection{The liquid region}

Of particular importance for these questions are the following extremal functions called the \emph{upper and lower obstacles}
\begin{align*}
	M(z)&:=\sup_{z\in P}\{u(z):u\in \mathscr{A}_N(P,u_0)\},\\
	m(z)&:=\inf_{z\in P}\{u(z):u\in \mathscr{A}_N(P,u_0)\}
\end{align*}
which one can show belong to $\mathscr{A}_N(P,u_0)$.  In case of the hexagon, these obstacles correspond to the empty and fully packed room (when viewing a tiling as configuration of boxes stacked in the corner of a room), see the figure \ref{figPA-bv1}.
 
\begin{figure}[t]
\centering
\begin{tikzpicture}[xscale=2,yscale=2]
\draw[thick] (0,0)--(1,0)--(2,1)--(2,2)--(1,2)--(0,1)--(0,0);
\draw[thick,blue] (1,0)--(1,1);
\draw[thick,blue] (1,1)--(0,1);
\draw[thick,blue] (1,1)--(2,2);
\draw[xshift=4cm] ({sqrt(3)/2-2.7},2) node{$P$};
\draw[xshift=4cm] ({sqrt(3)/2-5.2},0.5) node{$a=1$};
\draw[xshift=4cm] ({sqrt(3)/2-4.4},-0.1) node{$b$};
\draw[xshift=4cm] ({sqrt(3)/2-3.3},-0.1) node{$c$};

\draw[xshift=3cm,thick,blue] (1,1)--(1,2);
\draw[xshift=3cm,thick,blue] (1,1)--(0,0);
\draw[xshift=3cm,thick,blue] (1,1)--(2,1);
\draw[xshift=3cm,thick] (0,0)--(1,0)--(2,1)--(2,2)--(1,2)--(0,1)--(0,0);
\end{tikzpicture}
\caption{Obstacles for the $abc$ hexagon}
\label{figPA-bv1}
\end{figure}
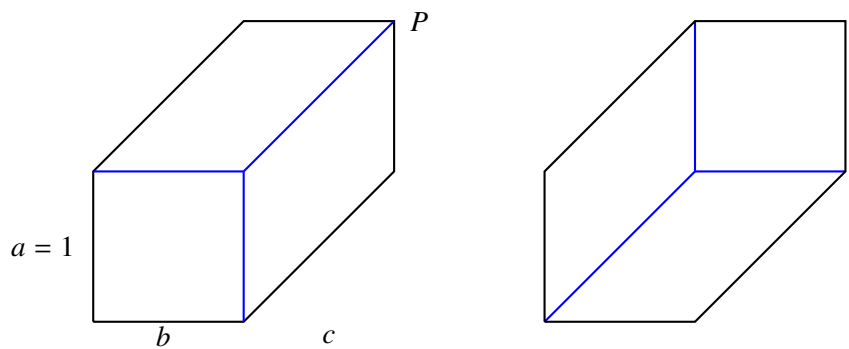

For the special case $w=0$,  the following (simplified version) fundamental initial regularity result is proven in \cite[Theorem 1.3]{DeS10}: 	The unique minimiser $h$ of \eqref{Func1} is $C^1$ on the set
	\begin{align*}
		\{z\in P: m(z)<h(z)<M(z)\}. 
	\end{align*} 
This allows one in particular to define the following set, called the \emph{liquid region} or \emph{disordered region} as
\begin{align}
	\LL:=\{z\in P: \text{ $h$ is $C^1$ in a neighbourhood of $z$ and $\nabla h(z)\in N^\circ$}\}. 
\end{align} 
We emphasize that this definition touches on a highly non-trivial issue. The minimiser is a Lipschitz function but will not be everywhere continuous differentiable (in fact, at the boundary of the disordered region it may fail to be continuously differentiable). Furthermore it is not even clear a priori that $\LL$ is non-empty.

Theorem 1.3 in \cite {DeS10} allowed in \cite{ADPZ} for the use of complex analytic methods to derive a complete classification of regularity of $h$ as well as the boundary of the liquid region (also called \emph{the frozen boundary or arctic curve}) or more precisely $\dv \LL \cap P$. Therefore, to initiate a rigorous study of the variational problem \ref{Func1} with $w\neq 0$, an extension of Theorem 1.3 is required.

On the other hand, in \cite[Thm. 1.2 c)]{ADPZ}, it is proven that frozen boundary is locally convex for any smooth point of the frozen boundary, see also \cite[Thm. 4.2]{DeS10}. Based on the explicit formulas for the liquid region in \cite{dimitrov2019log} that we will rederive in Section \ref{sec:tiling}, this result \emph{does not hold} in general when $w\neq 0$. This should be taken as an indication that the analysis will be much harder and that many useful properties in the case $w=0$ do not carry over to this case.

\subsection{Beltrami equation}

Central to the analysis of the variational problem and the regularity of its minimizer, is the first order non-linear  Beltrami equation of the form 
\begin{align} \label{eq:beltrami}
	f_{\overline{z}}(z)=\mathcal{H}'_\sigma(f(z))f_z(z)+w(z)
\end{align}
where 
\begin{align}\label{eq:Cayley}
	\mathcal{H}_\sigma(z)=(I-\nabla \sigma)\circ (I+\nabla \sigma)^{-1}(\overline{z})
\end{align}

Since  $\sigma$ satisfies a Monge-Ampère equation, we know by \cite[Lem. 3.1]{ADPZ} that
\begin{align*}
	\mathcal{H}_\sigma: \text{Dom}(\mathcal{H}_\sigma)=\{z\in \C: \overline{z}\in (I+\nabla \sigma)(N^\circ)\}\to \C
\end{align*}
is holomorphic and $\mathcal{H}_{\sigma}'(z)=\dv_z\mathcal{H}_\sigma(z): \text{Dom}(\mathcal{H}_\sigma)\to \Di$ is a \emph{proper holomorphic map}! 

\subsubsection{Euler-Lagrange conditions}

The Euler-Lagrange equations for the variational problem amount to:
\begin{align}\label{eq:EL2}
	\text{div}\,\nabla \sigma(\nabla h(z))=w(z),
\end{align}
for $ z \in \mathscr L$.  We prove that this equation is equivalent to the non-homogeneous Beltrami equation \eqref{eq:beltrami}. This theorem was derived jointly with Kari Astala.

\begin{theorem}\label{thm:B}
	Let $\mathcal{U}\subset \C$ be an open domain and assume that $u\in C^\infty(\mathcal{U})$ is a solution to \eqref{eq:EL2}
	where $\nabla u(\mathcal{U})\subset N^\circ$ and $\sigma$ as in  \eqref{sigmaL}. Furthermore let $w$ be as in \eqref{eq:Weight}.
	Then $f$ defined by  
	\begin{equation} \label{eq:relftoh}
		f(z)=\overline{(I+\nabla \sigma)(\nabla u(z))},
	\end{equation} solves the inhomogeneous $\C$-quasilinear Beltrami equation  \eqref{eq:beltrami}.
\end{theorem}

For $w=0$ this  is Theorem 3.2 in \cite{ADPZ}. Also, in our proof we only need that $\sigma$ is as in \cite[eq. (2.1)]{ADPZ} so that the result applies also to other situations than the context of this paper.

\begin{proof}
	Let $G$ be a smooth vector field in $\mathcal L$ such that $\text{div}\,G(z)=w(z)$. The existence of $G$ in $P$ follows from Lemma 3.5.5. in \cite{Sw}. Set $\A=\nabla \sigma$. Define the divergence free vector field in $\LL$ 
	\begin{align*}
		B(x):=\A(\nabla u(x))-G(x). 
	\end{align*}
	
	We introduce that div-curl couple $\F=(B,E)=(\A (\nabla u)-G,\nabla u)$. We define the function $v$ (locally, i.e. on simply connected subsets of $\mathscr H$) through 
	\begin{align*}
		\nabla v=\star B=\star (\A (\nabla u)-G),
	\end{align*}
where $\star$ is the Hodge star operator, equivalent in the plane to multiplication by the matrix
\begin{align*}
\begin{bmatrix}
0 & -1\\ 1 & 0
\end{bmatrix}.
\end{align*}
	We define
	\begin{align*}
		F(z)&=2(u(z)+iv(z)).
	\end{align*}
	Let 
	\begin{align*}
		\F^+(z)&=\frac{1}{2}(E(z)+B(z)), \quad \F^-(z)=\frac{1}{2}(E(z)-B(z)). 
	\end{align*}
	
	A direct computation gives
	\begin{align*}
		F_{\overline{z}}&=u_x-v_y+i(u_y+v_x),\\
		\overline{F_{z}(z)}&=u_x+v_y+i(u_y-v_x).\\
	\end{align*}
	Using that $\nabla v=(v_x,v_y)=(\A_2(\nabla u)-(G)_2,-\A_1(\nabla u)+(G)_1)$, where $\A=(\A_1,\A_2)$ and $G=(G_1,G_2)$, we get
	\begin{align*}
		2\F^+(z)&=(u_x-v_y,u_y+v_x),\\
		2\F^-(z)&=(u_x+v_y,u_y-v_x).
	\end{align*}

	Thus we have the identifications 
	\begin{align}\label{eq:ComplexGrad}
		\F^+(z)=\frac{1}{2}F_{\overline{z}}(z),\, \, \F^-(z)=\frac{1}{2}\overline{F_{z}(z)},
	\end{align}
	Since 
	\begin{align*}
		2\F^\pm&=E\pm B=E\pm \A(E)\mp G=(I\pm\A)(E)\mp G
	\end{align*}
	and 
	\begin{align*}
		E=(I+\A)^{-1}(2\F^++G)
	\end{align*}
	we get the algebraic relation 
	\begin{align*}
		2\F^--G=(I-\A)\circ (I+\A)^{-1}(2\F^++G) 
	\end{align*}
	which after identifications is equivalent to 
	\begin{align*}
		F_{\overline{z}}=(I-\A)\circ (I+\A)^{-1}(\overline{F_z}+G) +G
	\end{align*}
	Using the structure field $\LL$ this can be written as
	\begin{align}\label{eq:VolB}
		F_{\overline{z}}=\mathcal{H}_\sigma(F_z+\overline{G}) +G.
	\end{align}
	Set
	\begin{align*}
		f(z)=F_z(z)+\overline{G(z)}.
	\end{align*}
	and note that even though $F$ is only locally defined, $f$ is well-defined everywhere in $\LL$. 
	Then
	\begin{align*}
		f_{\overline{z}}(z)&=F_{\overline{z}z}(z)+\overline{G_{z}(z)}\\
		&=(\mathcal{H}_\sigma(f(z)) +G(z))_z+\overline{G_{z}(z)}\\
		&=\mathcal{H}_\sigma'(f(z))f_z(z)+G_z(z)+\overline{G_{z}(z)}\\
		&=\mathcal{H}_\sigma'(f(z))f_z(z)+\text{div}\,G(z)\\
		&=\mathcal{H}_\sigma'(f(z))f_z(z)+w(z).
	\end{align*}
	Finally we verify that $f(z)=\overline{L_\sigma(\nabla u(z))}$.
	
	We have 
	\begin{align*}
		F_z=u_x+v_y+i(v_x-u_y)
	\end{align*}
	
	In addition 
	\begin{align*}
		\nabla v(z)&=\ast (\A(\nabla u(z))-G(z))\\&=\begin{bmatrix}
			0 & -1 \\
			1& 0 
		\end{bmatrix}
		\begin{bmatrix}
			\A_1(\nabla u(z))-G_1(z) \\
			\A_2(\nabla u(z))-G_2(z)
		\end{bmatrix}
		=
		\begin{bmatrix}
			-\A_2(\nabla u(z))+G_2(z)\\
			\A_1(\nabla u(z))-G_1(z)
		\end{bmatrix}
	\end{align*}

	Hence 
	\begin{align*}
		\mathfrak{R}e[f(z)]=u_x+v_y+G_1(z) =u_x+\A_1(\nabla u(z))-G_1(z)+G_1(z)=u_x+\A_1(\nabla u(z))
	\end{align*}
	and 
	\begin{align*}
		\mathfrak{I}m[f(z)]=v_x-u_y-G_2(z)=-\A_2(\nabla u(z))+G_2(z)-u_y-G_2(z)=-\A_2(\nabla u(z))-u_y
	\end{align*}
	Thus $f(z)=\overline{(I+\nabla \sigma)(\nabla u(z))}$.
\end{proof}

\subsubsection{Complex structure}

Next we discuss an intrinsic coordinate system, called \emph{complex structure}, that plays an important role. Consider a homeomorphic solution $G:\mathcal{U} \to \mathcal{D}$ to the $\C$-linear Beltrami equation
\begin{align}\label{eq:Conf}
	G_{\dbb}(z)=\mathcal{H}_\sigma'(f(z))G_z(z).
\end{align}
for some suitable uniformisation domain $\mathcal{D}$. Since equation \eqref{eq:Conf} is not uniformly elliptic its existence is not automatically guaranteed, however repeating the argument in \cite[Thm. 4.3]{ADPZ} proves the existence of such $G$, which in this case will be a real analytic map. The complex structure will denote interchangeably \emph{both} the diffeomorphism $G$ endowing the liquid domain with a complex atlas and the Beltrami coefficient $\mu(z)=\mathcal{H}_\sigma'(f(z))$ which generates the diffeomorphism (which is unique up to suitable normalization). 

In the case $w=0$, this complex structure that generates the diffeomorphism $G$ turns out to be a natural object and has been discussed intensively in \cite{ADPZ}. In particular, these are the coordinates in which the fluctuations of the height function are conjectured to be described by the Gaussian Free Field (on $\mathcal D$), something that has been rigorously proven in a number of important cases, see for instance \cite[Ch. 11]{Vadim} and \cite{BuK16,Boutillier-Li}. We believe that the generalization to non-zero $w$ is still given by the same equation (with a dependence on $w$ only through $f$), including the  Gaussian Free Field fluctuations for the height function in the new coordinates. Evidence for our conjecture is supported by explicit computations that we do using recurrence coefficients for orthogonal polynomials.  

To illustrate the fact that \eqref{eq:Conf} still gives the complex structure even for $w\neq0$  is somewhat remarkable, we mention that the approach to solving the Beltrami equation \eqref{eq:beltrami} in \cite{ADPZ} for the case $w=0$ does not straightforwardly generalize. Following \cite{ADPZ} one can try the ansatz $f(z)=\eta(G(z))$ and simultaneously solve for the unknown function $\eta$ and homeomorphism $G$. Performing a hodograph transformation and after some manipulations one can show that \eqref{eq:beltrami} leads to the system 
\begin{equation}\label{eq:HN}
	\begin{cases}
		g_{\dbb}(z)=-\mathcal{H}_\sigma'(\eta(z))\overline{g_z(z)}, \\ 
		\eta_{\dbb}(z)=\overline{g_z(z)}w(g(z)) 
	\end{cases},
\end{equation}
where  $g=G^{-1}$. Note that when $w\equiv 0$, we have that $\eta$ is a holomorphic function and, more importantly, $\mathcal{H}_\sigma'(f(z))$ is a proper holomorphic.  This is in fact the Stoilow factorization theorem, that precisely says that solution space to the homogeneous Beltrami equation is given by holomorphic maps composed with a homeomorphism. In this case under the assumption that the minimiser $h$ satisfies that $\nabla h: \LL\to N^\circ$ is \emph{proper map}, we can deduce, as in the work \cite{ADPZ}, that $u$ and therefore $\mathcal{H}_\sigma' \circ \eta$ is a \emph{proper holomorphic function}, which allows one to identify the structure of $\mathcal{H}_\sigma'\circ \eta$. In this particular case, we can after post-composition with a  conformal map identify $\mathcal{H}_\sigma'\circ \eta$ with a Blaschke product. In the general non-homogeneous case when $w\neq 0$ this is no longer the case and studying \eqref{eq:HN} is very challenging. 

In the case when $w=\lambda$, where $\lambda$ is a constant all is not lost however. In this case \eqref{eq:HN} becomes 
\begin{equation}\label{eq:HN1}
    \begin{cases}
      g_{\dbb}(z)=-\mathcal{H}'(\eta(z))\overline{g_z(z)} \\
      \eta_{\dbb}(z)=\lambda \overline{g_z(z)} ,
     \end{cases}
\end{equation}

This gives
\begin{align*}
g_{\dbb}(z)&=-\frac{1}{\lambda}\mathcal{H}'(\eta(z))\eta_{\dbb}(z)\\
&=-\frac{1}{\lambda}(\mathcal{H}(\eta(z)))_{\dbb}
\end{align*}

Thus \eqref{eq:HN1} is a equivalent to 

\begin{equation}\label{eq:HN2}
    \begin{cases}
      (g(z)+\frac{1}{\lambda}\mathcal{H}(\eta(z)))_{\dbb}=0, \\
     (\eta(z)-\lambda \overline{g(z)})_{\dbb}=0.
     \end{cases}
\end{equation}

Thus there exists two holomorphic functions $\psi$ and $\phi$ such that

\begin{equation}\label{eq:HN3}
    \begin{cases}
      \lambda g(z)+\mathcal{H}(\eta(z))=\lambda \psi(z), \\
     \eta(z)-\lambda \overline{g(z)}=\phi(z).
     \end{cases}
\end{equation}

Decoupling the equations gives

\begin{equation}\label{eq:HN4}
    \begin{cases}
      \lambda g(z)+\mathcal{H}(\phi(z)+\lambda  \overline{g(z)})=\lambda \psi(z), \\
    \overline{\eta(z)}+\mathcal{H}(\eta(z))=\overline{\phi(z)}+\lambda \psi(z).
     \end{cases}
\end{equation}

We now solve for $\eta$ and $g$ in terms of $\psi$ and $\phi$. Using the Lewy transform 
\begin{align*}
L_\sigma(z)=\overline{z+\nabla \sigma(z)}
\end{align*}
we have by \cite[eq. (3.14)]{ADPZ} 
\begin{align*}
L_\sigma^{-1}(w)=\frac{1}{2}\bigg(\overline{w}+\mathcal{H}(w)\bigg)\in N^\circ, \quad \overline{w}\in \text{dom}\, (\mathcal{H}). 
\end{align*}
and so the equation $L_\sigma^{-1}(w)=\zeta$ has solution $w=L_{\sigma}(\zeta)$, which gives 
\begin{align*}
\eta(z)&=L_\sigma\bigg(\frac{1}{2}(\overline{\phi(z)}+\lambda \psi(z))\bigg)\\
g(z)&=\frac{\overline{\eta(z)}-\overline{\phi(z)}}{\lambda}=\frac{1}{\lambda}(\overline{L_{\sigma}((\overline{\phi(z)}+\lambda \psi(z)))/2}-\overline{\phi(z)}).
\end{align*}

It is tempting to use the parametrisation of $\eta$ and $g$ in terms of the pair of holomorphic functions $(\phi,\psi)$ to deduce regularity of the height function $h$ and the boundary of the liquid domain $\dv \LL$. Unfortunately, it seems very difficult to determine which holomorphic functions $\psi$ and $\phi$ that gives rise to \emph{homeomorphic} $g$ and \emph{proper} $\eta$. Furthermore it seems difficult to connect to the case when $\lambda=0$ and $g$ is teleomorphic. What we want is to view $g$ in the case $\lambda\neq 0$ as a natural deformation of a teleomorphic $g$ as in \cite{ADPZ}. Luckily it seems that a certain combination of $\phi$ and $\psi$ can be determined. We now define
\begin{align*}
\omega(z):=\frac{1}{2}(\overline{\phi(z)}+\lambda \psi(z)).
\end{align*}
Then $\Delta \omega=0$ so $\omega$ is a harmonic map. Furthermore, since $L_\sigma^{-1}$ is a homeomorphism and $u$ is proper and $\omega=L_\sigma(u)$ it follows that $\omega: \Di \to N$ is a proper harmonic map. Hence $\omega$ is $k:1$ for some $k$ and an orientation preserving map. Furthermore, since $\mathcal{H}'$ is a proper holomorphic map it follows that $\mathcal{H}'(\eta)$ is proper as well. The case of \emph{harmonic homeomorphisms} $\omega:\Di \to N$, where $N$ is a convex compact polygon in the plane is well-understod, see \cite[Ch. 4]{Dur}, although the authors are unaware of a complete classification of $k:1$ harmonic mappings $\omega: \Di \to N$. Assuming one has a good understanding of this class of harmonic maps, instead of studying \eqref{eq:HN1} one could study the \emph{linear equation}
\begin{align*}
g_{\dbb}(z)=-\mathcal{H}'(L_\sigma(\omega))\overline{g_z(z)},
\end{align*} 
where $\omega$ is now \emph{any proper harmonic map} $\omega: \Di \to N$ and determine the boundary regularity for $g$. A key idea in \cite{ADPZ} for studying the boundary regularity in the case when $\lambda=0$ is to consider the map
\begin{align*}
\gamma(z)=g(z)+\mathcal{H}'(\eta(z))\overline{g(z)}, 
\end{align*}
which is holomorphic in this case. This is no longer the case when $\lambda\neq 0$. In this case one can show that $\gamma$ solves the equation
\begin{align*}
\gamma_{\dbb}(z)&=\alpha(z)\gamma(z)+\beta(z)\overline{\gamma(z)},
\end{align*}
where 
\begin{align*}
\alpha(z)&=\frac{\mathcal{H}''(\eta(z))\eta_{\dbb}(z)}{1-\vert \mathcal{H}'(\eta(z))\vert^2}\overline{\mathcal{H}'(\eta(z))},\\
\beta(z)&=\frac{\mathcal{H}''(\eta(z))\eta_{\dbb}(z)}{1-\vert \mathcal{H}'(\eta(z))\vert^2}.
\end{align*}
Thus $\gamma$ is a generalized analytic function in the sense of Vekua. Then Vekua's factorization theorem implies that 
\begin{align*}
\gamma(z)=\Phi(z)e^{\kappa(z)}. 
\end{align*}
for some holomorphic $\Phi$ and Hölder continuous $\kappa$. Solving for $g$ gives 
\begin{align*}
g(z)&=\frac{1}{1-\vert \mathcal{H}'\eta(z))\vert^2}(\gamma(z)-\mathcal{H}'(\eta(z))\overline{\gamma(z)}).
\end{align*}
To understand the geometry of the frozen boundary one must then study the singular limit
\begin{align*}
\lim_{z\in \Di\to z_0\in \dv \Di}\frac{1}{1-\vert \mathcal{H}'(\eta(z))\vert^2}(\gamma(z)-\mathcal{H}'(\eta(z))\overline{\gamma(z)}).
\end{align*}

If one can fully understand the case when the weight function $w$ is constant, it is conceivable that the case for general non-singular $w$ can be reduced to that for a constant $w$ by approximation and localization arguments. The case with singular weights $w$ however would require completely different methods.

\subsubsection{Complex structure and complex Burger's equation}

We next show that the complex structure  $\mathcal H'(f)$  is directly related to the \emph{complex slope} which solves an inhomogeneous complex Burgers equation (as expected). To this end, we first  define the complex valued functions $\Upsilon,\Om$ on $\LL$ through
\begin{align}\label{eq:UO}
	\log \Upsilon&=\pi \sigma_1(\nabla h)-i\pi h_y\\
	\log \Om&= \pi \sigma_2(\nabla h)+i\pi h_x.\label{eq:UOp}
\end{align}
Note that our $\sigma$ is normalized differently than in \cite{KeOk07}. As in \cite{KeOk07} one verifies that 
\begin{align} \label{eq:UO2}
	\Upsilon+\Om\equiv 1,
\end{align}
and that $(\Upsilon,\Om)$ solves the complex Burger equation
\begin{align}\label{eq:CB}
	\frac{\Upsilon_x(z)}{\Upsilon(z)}+\frac{\Om_y(z)}{\Om(z)}=w(z),\quad z\in \LL
\end{align}
whenever $h$ solves the Euler-Lagrange equation \eqref{eq:EL2} in $\LL$. 

From \eqref{eq:UO}, \eqref{eq:UOp} and \eqref{eq:UO2} we find that the limiting density for the lozenges can be read off from  the triangle  formed by $\Omega$, $0$ and $1$. We  relate $\Upsilon$ and $\Om$ to the complex structure $\mathcal{H}'(f)$. 

\begin{proposition}\label{prop:ComplexS}
	Let $(\Upsilon,\Om)$ be given by \eqref{eq:UO} and \eqref{eq:UO2}. Then 
	\begin{align}\label{eq:ComplexS}
		\mathcal{H}'(f)=-\frac{i+(1-i)\Upsilon}{i-(1+i)\Upsilon}=-\frac{1+(i-1)\Om}{-1+(1+i)\Om}. 
	\end{align}
	Furthermore, $\Upsilon$ and $\Om$ are related to $f$ in Theorem \ref{thm:B} through 
	\begin{align*}
		\Upsilon&=e^{\frac{\pi}{2}(f-\mathcal{H}(f))},\\
		\Om&=e^{i\frac{\pi}{2}(f+\mathcal{H}(f))}.
	\end{align*}
\end{proposition}

\begin{proof}
	Identifying $\nabla \sigma(z)$ with $\sigma_x+i\sigma_y$, we have 
	\begin{align*}
		\pi \sigma_x(z)-i\pi y&=\frac{\pi}{2}(\nabla \sigma(z)+\overline{\nabla \sigma(z)})-\frac{\pi}{2}(z-\overline{z})\\
		&=-\frac{\pi}{2}(z-\nabla \sigma(z))+\frac{\pi}{2}\overline{(z+\nabla \sigma(z))}\\
		\pi \sigma_y(z)+i\pi x&=i\frac{\pi}{2}(\overline{\nabla \sigma(z)}-\nabla \sigma(z))+i\frac{\pi}{2}(z+\overline{z})\\
		&=i\frac{\pi}{2}(z-\nabla \sigma(z))+i\frac{\pi}{2}\overline{(z+\nabla \sigma(z))}
	\end{align*}
	We recall that 
	\begin{align*}
		\mathcal{H}(z)=(I-\nabla \sigma)(I+\nabla \sigma)^{-1}(\overline{z}),
	\end{align*}
	and that by Theorem \ref{thm:B},
	\begin{align*}
		\nabla h(z)=U(\overline{f(z)}),
	\end{align*}
	where $U=((I+\nabla \sigma)^{-1})$ is the inverse of the Lewy transform $L(z)=z+\nabla \sigma(z)$. Set $z=U(\overline{w})$. Then 
	\begin{align*}
		\log \gamma(w)&:=-\frac{\pi}{2}((U(\overline{w}))-\nabla \sigma(U(\overline{w})))+\frac{\pi}{2}\overline{L(U(\overline{w}))}\\
		&=\frac{\pi}{2}(-\mathcal{H}(w)+w)\\
		\log \omega(w)&:=i\frac{\pi}{2}(U(\overline{w})-\nabla \sigma(U(\overline{w})))+i\frac{\pi}{2}\overline{L(U(\overline{w}))}\\
		&=i\frac{\pi}{2}(\mathcal{H}(w)+w),
	\end{align*}
	and $\Upsilon(z)=\gamma(f(z))$ and $\Om(z)=\omega(f(z))$. Since $\gamma(w)+\omega(w)\equiv 1$ differentiating this relation with respect to $\frac{\dv}{\dv w}$ gives 
	\begin{align*}
		\frac{\pi}{2}(1-\mathcal{H}'(w))e^{\frac{\pi}{2}(w-\mathcal{H}(w))}+i\frac{\pi}{2}(1+\mathcal{H}'(w))e^{i\frac{\pi}{2}(w+\mathcal{H}(w))}=0
	\end{align*}
	or equivalently 
	\begin{align*}
		(1-\mathcal{H}'(w))\zeta+i(1+\mathcal{H}'(w))\omega=0.
	\end{align*}
	Solving for $\mathcal{H}'(w)$ gives using that $\omega=1-\zeta$
	\begin{align*}
		\mathcal{H}'(w)=-\frac{\zeta(w)+i\omega(w)}{-\zeta(w)+i\omega(w)}=-\frac{i+(1-i)\zeta(w)}{i-(1+i)\zeta(w)}=-\frac{1+(i-1)\omega(w)}{-1+(1+i)\omega(w)}
	\end{align*}
	This completes the proof. 
\end{proof}


\section{Lozenge Tilings of a hexagon: Asymptotic Analysis} \label{sec:tiling}
In this section, we return to the lozenge tilings of the hexagon with weight \eqref{eq:ellipticloz}.  As discussed in the Introduction, the position of the white points in Figure~\ref{fig:sample_color} from the lozenge tiling with the $q$-Racah weight can be described as a dynamic $q$-Racah ensemble. This is shown in Theorem 7.5. in \cite{borodin2010q}, and we will briefly recall their definition. We then use Theorems~\ref{thm:LLNlpfunctions}, \ref{limit_shape_qracah} and \ref{thm:fluctuation_theorem} to study the asymptotic behavior of this point process.

\subsection{Extended $q$-Racah ensemble} \label{subsec:tiling model}

Let $\mathbf a, \mathbf b, \mathbf c\in \mathbb N$ and consider the probability measure on all possible tilings of the hexagon defined by the weight \eqref{eq:ellipticloz}. 
Let  $\{(\unscaledtime,y_j(\unscaledtime))\}_{\unscaledtime=1,j=1}^{\mathbf{b}+\mathbf{c}-1,\mathbf{a}}$ be the point process for the lozenge tiling model, which describes the positions of the white points in Figure~\ref{fig:sample_color} shifted by $\frac 12$.   As mentioned in the Introduction, for fixed $1\leq s \leq {\bf b}+{\bf c}-1$, the points $\{y_j(s)\}_{j=1}^{{\mathbf a}}$ form a $q$-Racah ensemble. Following \cite{borodin2010q}, the parameters for the ensemble, $\alpha$, $\beta$, $\gamma$ and $\delta$ can be expressed in terms of the parameters $\kappa, {\bf a}, {\bf b}$ and ${\bf c}$ as indicated in Table~\ref{table:q-racah-tiling}.  There are four cases in Table~\ref{table:q-racah-tiling} depending on $s$ and the dimensions of the hexagon ${\bf a},{\bf b}$ and ${\bf c}$.  If one assumes $\mathbf b<\mathbf c$, only cases i) iii) and iv) are involved, whereas if $\mathbf b>\mathbf c$, only cases iii) and iv) are involved. If $\mathbf b=\mathbf c$, then only cases i) and 4) occur.   Important to note is that we always count from the bottom of the hexagon. This means that for $s>{\bf b}$, we shift the $q$-Racah ensemble by $s-{\bf b}$, which happens for cases iii) and iv). 

In \cite{borodin2010q} the probability distribution for the extended process $\{(\unscaledtime,y_j(\unscaledtime))\}_{\unscaledtime=1,j=1}^{\mathbf{b}+\mathbf{c}-1,\mathbf{a}}$ was also computed and we will recall their definitions.
The starting point are the transition matrices, for $s=0,\ldots,{\mathbf b}+{\mathbf c}-1$,
     \begin{align*}
	T_{\unscaledtime}(y_i,y_j)  \coloneqq\sum_{k=0}^{{\bf a}-1}\frac{c_{k,\unscaledtime+1}}{c_{k,\unscaledtime}}r_{k,\unscaledtime}(y_i)r_{k,\unscaledtime+1}(y_j),
\end{align*}
where, for $s=0,\ldots, {\mathbf b}+{\mathbf c}$,  the function $r_{k,\unscaledtime}(y) $ is the $q$-Racah orthonormal polynomial with parameters specified in Table~\ref{table:q-racah-tiling}.  So  the $r_{k,s}$ are  polynomials in $\nu_s(y)$ orthogonal with respect to the $q$-Racah weight $w_{n,\unscaledtime}(y)$ as defined in \eqref{eq:qRac}. Important to note is that $s > {\bf b}$ we shifted the  $q$-Racah polynomial and its weight $w$ by $s-{\bf b}$. To avoid cumbersome notation, we do not indicate this shift explicitly, and we trust this will not lead to confusion. The coefficients $c_{k,s}$ are determined by 
\begin{align} \label{eq:expressionsck}
	c_{k,0}  \coloneqq \sqrt{\prod_{r=0}^k\frac{1-\mathbf{q}^{-(\mathbf{a}+\mathbf{b}+\mathbf{c}-r)}}{1-\mathbf{q}^{-(\mathbf{a}-r)}}},\qquad
	c_{k,\unscaledtime} \coloneqq c_{k,0}\sqrt{\prod_{r=0}^{\unscaledtime-1}(1-\mathbf{q}^{-\mathbf{a}-r+k})(\mathbf{q}^{(\mathbf{a}+\mathbf{b}+\mathbf{c}-r-k-1)}-1)} .
\end{align}

The probability measure on the points becomes the so-called dynamic $q$-Racah polynomial ensemble, which is proportional to 
\begin{equation} \label{eq: tiling dpp}
\sim \prod_{\unscaledtime=0}^{\mathbf{b}+\mathbf{c}-1}\det(T_{\unscaledtime}(y_{i}(\unscaledtime),y_{j}(\unscaledtime+1)))_{i,j=1}^{\mathbf{a}} \prod_{\unscaledtime=1}^{\textbf{b}+\mathbf{c}-1}\prod_{j=1}^{\mathbf{a}} w_{\unscaledtime}(y_j(\unscaledtime)),
	\end{equation}
 with $y_j(0)=j-1$ and $y_j({\mathbf b}+{\mathbf c})={\mathbf c}+j-1$ fixed for $j=1,\ldots,{\mathbf a}$.

\begin{table}[t]
\centering
\begin{tabular}{ c c c c c c c}
	\hline
		& & $\alpha$ & $\beta$ & $\gamma$ & $\delta$ & $y$\\ \hline
		i)& $\unscaledtime<\min\{\mathbf{b},\mathbf{c}\}$, \\& $0\leq y+\min\{0,\mathbf{b}-\unscaledtime\}\leq \unscaledtime+\mathbf{a}-1$ &  $\mathbf{q}^{-\mathbf{c}-\mathbf{a}}$ & $\mathbf{q}^{-\mathbf{b}-\mathbf{a}}$ & $\mathbf{q}^{-\unscaledtime-\mathbf{a}}$ &  $\kappa^2\mathbf{q}^{-\mathbf{c}+\mathbf{a}}$ & $y$ \\ 
	\hline 
		ii)& if $\mathbf{b}>\mathbf{c}$,   $\mathbf{c}-1<\unscaledtime<\mathbf{b}-1$, \\& $0\leq y+\min\{0,\mathbf{b}-\unscaledtime\}\leq \mathbf{c}+\mathbf{a}-1$ &  $\mathbf{q}^{-\unscaledtime-\mathbf{a}}$ & $\mathbf{q}^{\unscaledtime-\mathbf{a}-\mathbf{b}-\mathbf{c}}$ & $\mathbf{q}^{-\mathbf{a}-\mathbf{c}}$ &  $\kappa^2\mathbf{q}^{-\unscaledtime+\mathbf{a}}$ & $y$ \\ 
	\hline 
		iii)& if $\mathbf{b}\leq \mathbf{c}$, $\mathbf{b}-1<\unscaledtime<\mathbf{c}$, \\ & $0\leq y+\min\{0,\mathbf{b}-\unscaledtime\}\leq \mathbf{b}+\mathbf{a}-1$ &  $\mathbf{q}^{\unscaledtime-\mathbf{a}-\mathbf{b}-\mathbf{c}}$ & $\mathbf{q}^{-\unscaledtime-\mathbf{a}}$ & $\mathbf{q}^{-\mathbf{b}-\mathbf{a}}$ &  $\kappa^2\mathbf{q}^{\unscaledtime-\mathbf{b}-\mathbf{c}+\mathbf{a}}$ & $\mathbf{b}-\unscaledtime+y$\\ \hline 
		iv) & $\max\{\mathbf{b},\mathbf{c}\}<\unscaledtime+1$,\\ & $0\leq y+\min\{0,\mathbf{b}-\unscaledtime\}\leq \mathbf{b}+\mathbf{c}-\unscaledtime+\mathbf{a}-1$ &  $\mathbf{q}^{-\mathbf{b}-\mathbf{a}}$ & $\mathbf{q}^{-\mathbf{c}-\mathbf{a}}$ & $\mathbf{q}^{\unscaledtime-\mathbf{a}-\mathbf{b}-\mathbf{c}}$ &  $\kappa^2\mathbf{q}^{-\mathbf{b}+\mathbf{a}}$ & $\mathbf{b}-\unscaledtime+y$\\ \hline 
	\end{tabular}
	\caption{Parameters for Lozenge Tiling with $q$-Racah weight }
	\label{table:q-racah-tiling}
\end{table}	

To establish a probability measure, one also needs a configuration of $({\mathbf q}, \kappa)$. Here we consider 
	\begin{enumerate}
		\item Real case: ${\mathbf q}>0, \kappa\in \mathbb{R}$ and $\kappa$ can not lie in $[{\mathbf q}^{-{\mathbf a}+\frac{1}{2}},{\mathbf q}^{\frac{{\mathbf b}+{\mathbf c}}{2}-\frac{1}{2}}]$ if $\mathbf q>1$ or $[{\mathbf q}^{\frac{\mathbf b+{\mathbf c}}{2}-\frac{1}{2}}, {\mathbf q}^{-\mathbf a+\frac{1}{2}}]$ if ${\mathbf q}<1$. 
		\item Imaginary case : ${\mathbf q}>0, \kappa\in i \mathbb{R}$. 
		\item Trigonometric case : ${\mathbf q}=e^{i \lambda}, \kappa=e^{i \eta}$. To establish a probability measure,  it is enough to let $\lambda,\eta$ be such that $\eta-\frac{\lambda({\mathbf b}+ {\mathbf c }-1)}{2}, \eta+\lambda(\mathbf{a}-\frac{1}{2}) \in(k\pi,(k+1)\pi)$ for some $k$ to be an integer.
        \item Singular trigonometric case: similarly in the trigonometric case with $k\pi \leq \eta-\frac{\lambda({\mathbf b}+ {\mathbf c }-1)}{2}, \eta+\lambda(\mathbf{a}-\frac{1}{2})\leq(k+1)\pi$ and at least one of the inequality is an equality.
		\item Singular real case:   $\mathbf q>0$ and $\kappa = {\mathbf q}^{-\mathbf{a}+\frac{1}{2}}$ or ${\mathbf q}^{\frac{{\mathbf b}+{\mathbf c}-1}{2}}$. 
	\end{enumerate}


\subsection{Large hexagon}
To study the asymptotic behavior, we need to scale the parameters. The limit regime we are considering are as such 
\begin{align}\label{eq: scale parameters}
    \mathbf{a}=n, \quad \mathbf{b}=nb, \quad\mathbf{c}=nc, \quad\mathbf{q}=q^{\frac{1}{n}}
\end{align}
for some $b,c>0$ and $q>0$ or $q$ on the unite circle. We further scale the lattice by $\frac{1}{n}$ via 
\begin{align}
    y =nx,  &\qquad x\in\frac{1}{n}\mathbb{Z}\\
    \unscaledtime=\lfloor nt \rfloor, &\qquad t\in[0,b+c].
\end{align}
Instead of looking at the $(s,y)$ plane, we are looking at the $(t,x)$ plane. Set $$x_j(t)=\frac{1}{n}y_j(\unscaledtime).$$ We are going to study the process $\{(t,x_j(t))\}_{t,j}$.

Following the analysis of Section~\ref{sec:generalOPE}, the limiting behavior of a random tiling is determined by the asymptotic behavior of the recurrence coefficients for the $q$-Racah polynomials and the coefficients $c_{k,s}$. We will summarize the relevant results first and apply them to find the limiting behavior of the height function and its fluctuations in later sections.  We will only work out the case of real and imaginary parameters listed above, and we leave the trigonometric case to the readers.

The recurrence coefficients are already studied in Section~\ref{sec:q-Racah}. We will recall their limits and express them in the parameters of the tiling model. Define  
	\begin{align} \label{eq:defaballcases}
		a(\xi;t) = \sqrt{A(\xi)C(\xi)}, \qquad \qquad b(\xi;t) = -A(\xi;t)-C(\xi;t)+1+\gamma\delta,
	\end{align}
	where $A(\xi)$ and $C(\xi)$ are defined by
	\begin{align}
		A(\xi;t) & =\frac{\left(1-q^{-c+\xi-1}\right) \left(1-q^{\xi-t-1}\right) \left(1-q^{-b-c+\xi-2}\right) \left(1-\kappa ^2 q^{-b-c+\xi}\right)}{\left(1-q^{-b-c+2 \xi-2}\right)^2} \label{eq:recurrance A}\\
		C(\xi;t) & = \frac{\left(1-q^\xi\right) \left(1-q^{-b+\xi-1}\right) \left(\kappa ^2 q^{1-c}-q^{-c+\xi-1}\right) \left(q^{-t-1}-q^{-b-c+\xi-2}\right)}{\left(1-q^{-b-c+2 \xi-2}\right)^2}.\label{eq:recurrance CC}
	\end{align}
 These expressions are the limits of the coefficients $A_{k,n}$ and $C_{k,n}$, but only in case i) and ii). For cases iii) and iv), we find different limits, but somewhat remarkably, the expressions above still give the right limiting behavior for $a(\xi;t)$ and $b(\xi;t)$ (up to a constant) as stated in the following lemma.

\begin{lemma} \label{lem:casessame} Let $a(\xi;t)$ and $b(\xi;t)$ be as in \eqref{eq:defaballcases}. Then
    \begin{equation}
        \lim_{\frac{j}{n}\to\xi} a_{k,n}^{(s)} =\begin{cases}
        a(\xi;t), & 0< t\leq b,\\
        q^{t-b} a(\xi;t), &b<t<b+c,
        \end{cases}
    \end{equation}
   and    \begin{equation}
        \lim_{\frac{j}{n}\to\xi} b_{k,n}^{(s)} =\begin{cases}
        b(\xi;t), & 0< t\leq b,\\
        q^{t-b} b(\xi;t), &b<t<b+c.
        \end{cases}
        \end{equation}
        Moreover, with $\mu(t,x)=q^{-x}+\kappa^2q^{-t-c+x}$  we have
          \begin{equation}
    \lim_{n\to\infty}\mu_n(x) =\begin{cases}
        \mu(x;t), & 0< t\leq b,\\
           q^{t-b} \mu(x;t), &b<t<b+c.
        \end{cases}
    \end{equation}
\end{lemma}
\begin{proof}
    This follows by a straightforward computation using \eqref{eq:recurrance A}, \eqref{eq:recurrence C} and Table \eqref{table:q-racah-tiling}. The details are left to the reader.
\end{proof}
The following lemma deals with the limiting behavior of the coefficients $c_{k,s}$.
\begin{lemma} \label{lem:cks}
Let $t \in (0,b+c)$ and $k,l \in \mathbb N$. Then, with $s=\lfloor tn\rfloor $ we have
\begin{equation}
    \lim_{n\to \infty}\frac{c_{n+k,t}}{c_{n+l,t}}=e^{\tau(t)(l-k)}
\end{equation}
where
   \begin{equation}
		\tau(t)  = \frac{1}{2}\log\left|\frac{q^{t}-1}{1-q^{t-b-c}}\right| \label{eq: tau}.
    \end{equation}
\end{lemma}
\begin{proof}
 For $k<l$ and $\unscaledtime>l-k+1$ we obtain, from \eqref{eq:expressionsck}
	\begin{align} 
		\frac{c_{k,\unscaledtime}}{c_{l,\unscaledtime}} & =\frac{c_{k,0}}{c_{l,0}} \sqrt{\prod_{r=k+1}^{l}\frac{(1-\mathbf{q}^{-(\mathbf{a}+\unscaledtime-r)})(\mathbf{q}^{\mathbf{a}+\mathbf{b}+\mathbf{c}-r}-1)}{(1-\mathbf{q}^{-(\mathbf{a}-r)})(\mathbf{q}^{\mathbf{a}+\mathbf{b}+\mathbf{c}-\unscaledtime-r}-1)}} = \sqrt{\mathbf{q}^{\unscaledtime(l-k)}\prod_{r=k+1}^{l}\frac{1-\mathbf{q}^{-(\mathbf{a}+\unscaledtime-r)}}{1-\mathbf{q}^{-(\mathbf{a}+\mathbf{b}+\mathbf{c}-s-r)}}} 
	\end{align}
 In the first equality, we used the second expression in \eqref{eq:expressionsck} and the fact that we have a telescoping product. In the second, we used the explicit expression for $c_{k,0}$ and $c_{l,0}$.
	Therefore, with $k<l$, $t\in[0,b+c]$ and $\unscaledtime \coloneqq \lfloor nt\rfloor$,  we have that
	\begin{align}
		\frac{c_{n+k,\unscaledtime}}{c_{n+l,\unscaledtime}} & = \sqrt{q^{\unscaledtime (l-k)/n}\prod_{r=k+1}^{l}\frac{1-q^{-(\unscaledtime/n-r/n)}}{1-q^{-(b+c-\unscaledtime/n-r/n)}}} \to \sqrt{\left|\frac{q^{t}-1}{1-q^{-(b+c-t)}}\right|^{l-k}},
	\end{align}
	as $n \to \infty$, and this proves the statement.	
\end{proof}

	\begin{figure}[t]
		\begin{subfigure}[b]{0.4\textwidth}
			\includegraphics[scale=0.2]{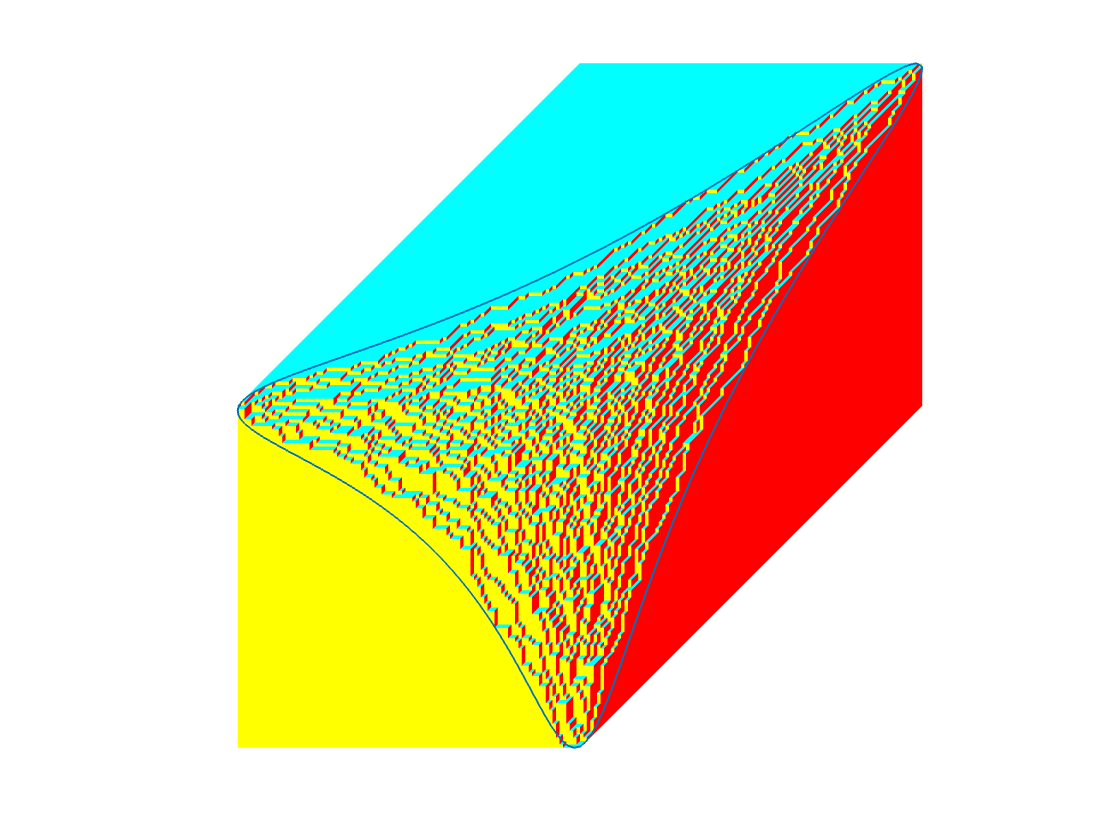}
			\caption{Real Case $q=0.05, \kappa=0.01$}
		\end{subfigure}
		\hfill
		\begin{subfigure}[b]{0.4\textwidth}
			\includegraphics[scale=0.2]{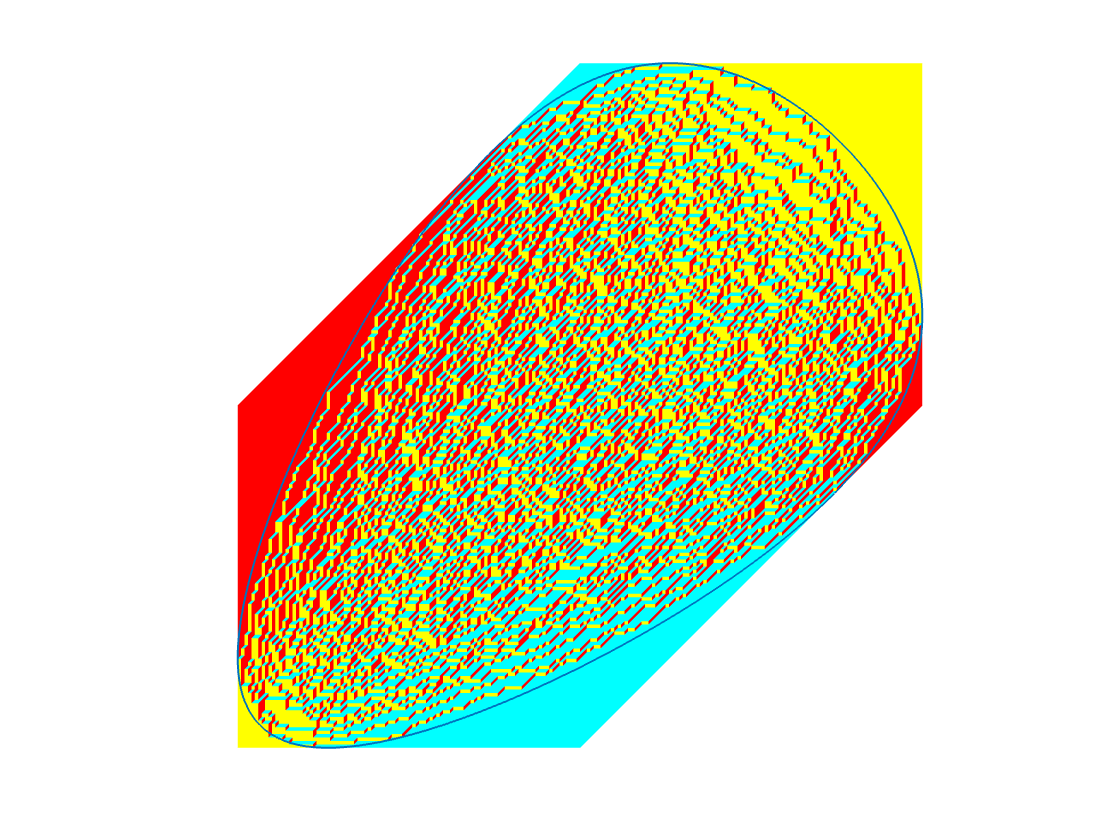}
			\caption{Imaginary Case $q=0.5, \kappa=29i$}
		\end{subfigure}
		\hfill
		\begin{subfigure}[b]{0.4\textwidth}
			\includegraphics[scale=0.2]{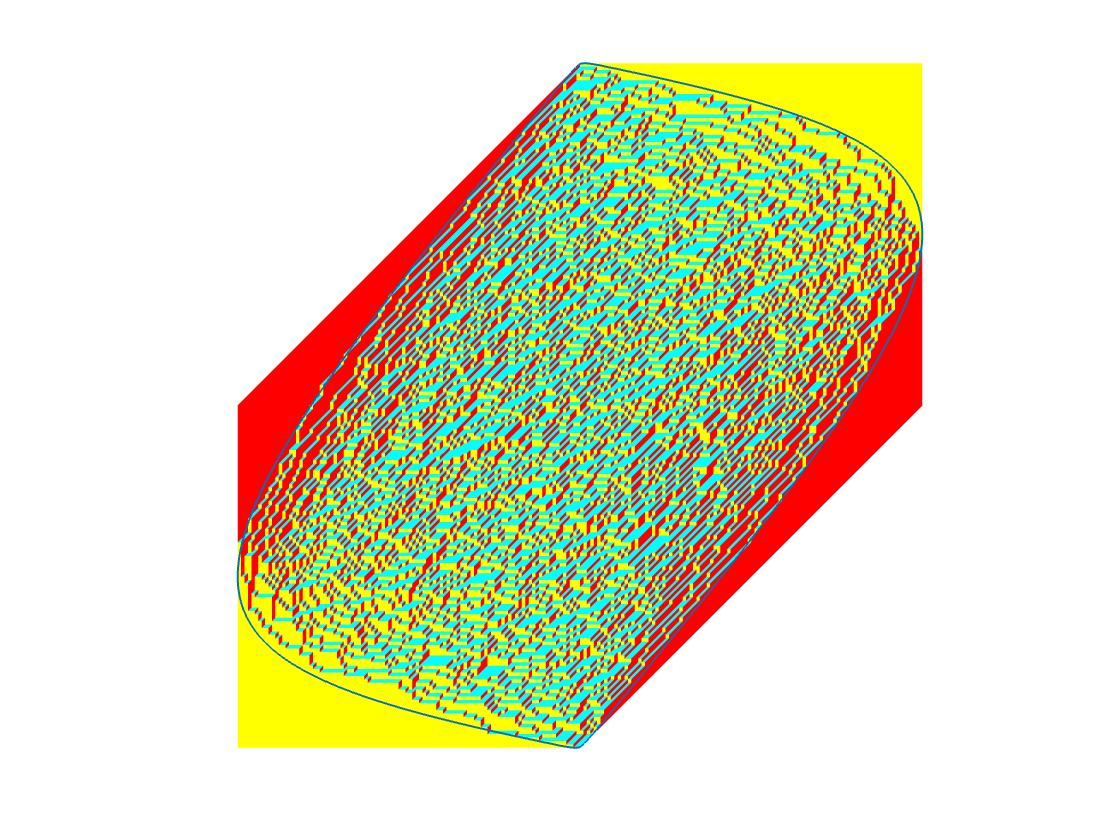}
			\caption{Singular Trigonometric Case \\$q=e^{i(\frac{\pi}{2}-\frac{1}{n})}, \kappa=e^{i\frac{\pi}{2}}$}
		\end{subfigure}
		\hfill
		\begin{subfigure}[b]{0.4\textwidth}
			\includegraphics[scale=0.2]{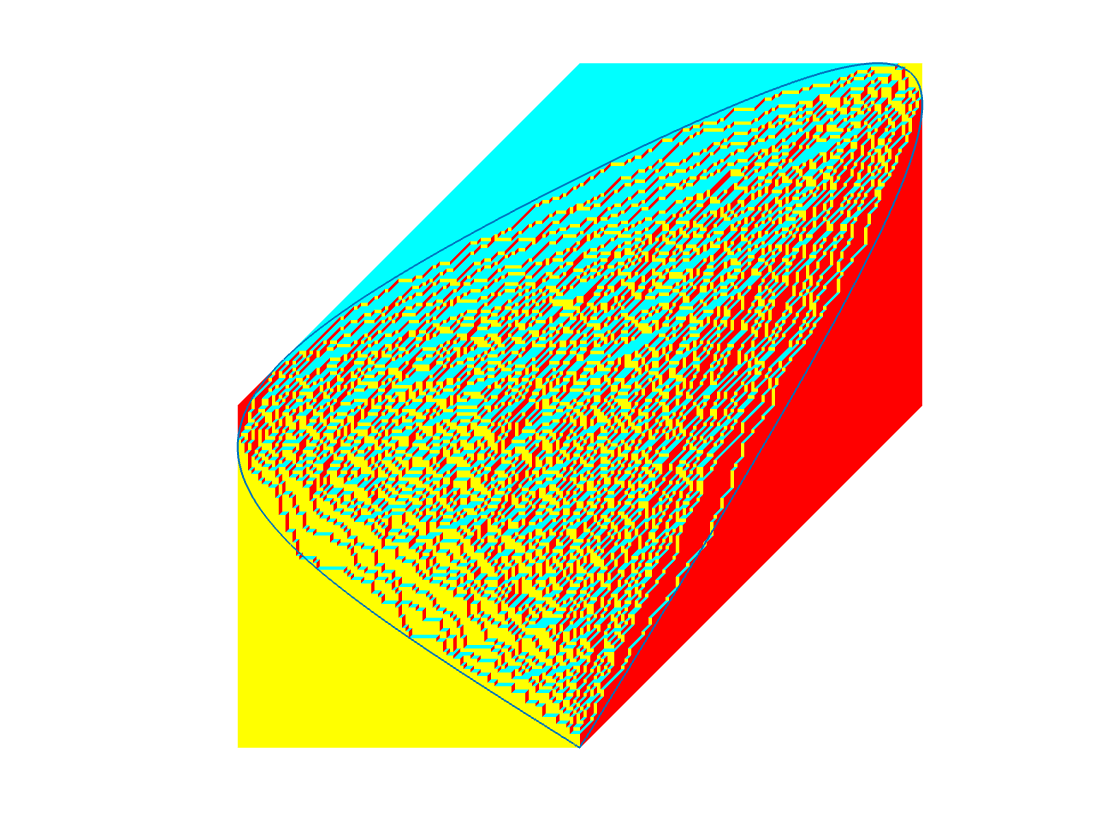}
			\caption{Singular Real Case $q=e, \kappa=e$}
		\end{subfigure}
		\caption{Different Cases of $q$-Racah Weights with size $n=100$} \label{fig:morexamples}
	\end{figure}


\subsection{Limit Shape}
We will now compute the limiting height function. We recall the definition of the height function $h_n$ in \eqref{eq:heightfunction}. Then for any fixed $(t,x)$, we can rewrite the height function as 
    \begin{align*}
    h_n(t,x)=\sum_{j=1}^{\mathbf{a}}\chi_{(-\infty, x]}(x_j(t)) =X_n(\chi_{(-\infty, x]}),
    \end{align*}
   where  $\chi_{A}$ stands for the characteristic function of the set $A$. Hence, $h_n$ is the linear statistic for the step function with the position of the step at $x$.

   We recall from Section~\ref{sec:q-Racah} that the limiting density of points can be computed by solving 
   \begin{equation}\label{eq:centraleqall}
      a(\xi;t)^2-(\mu(x;t)-b(\xi;t))^2=0,
   \end{equation}
   where $a(\xi;t)$ and $b(\xi;t)$ are defined as in \eqref{eq:defaballcases}. We recall that these expressions are only valid for $0<t\leq b$, but by Lemma~\ref{lem:casessame}, we see that for $t>b$ equation \eqref{eq:centraleqall} has an overall factor $q^{t-b}$ which we can divide out. Thus, \eqref{eq:centraleqall} is the correct equation for all $t$.
   
 \begin{theorem}[Limit Shape for the Lozenge Tiling]  \label{thm: lln tiling} Let $h_n$ be the height function for the rescaled $q$-Racah lozenge tiling in the hexagon defined in \eqref{eq:heightfunction}. For any given point $(t,x)$, we have the almost sure convergence of the height function 
 \begin{align}
     \frac{1}{n}h_n(t,x)\to \int_{\max\{0,t-b\}}^x\rho(t,u)du
 \end{align}
    as $n\to\infty$
    where $\rho(t,x)$ is the density defined by
    \begin{align}\label{eq:limitshapetiling}
			\rho(x)= \begin{cases} 1 & \quad \xi_-(x)<\xi_+(x)\leq 1\\ \frac{1}{\pi}\arccos\left(\frac{2y(1)-y(\xi_-(x))-y(\xi_+(x))}{\left|y(\xi_-(x))-y(\xi_+(x))\right|}\right)& \quad \xi_-(x)<1<\xi_+(x) \\ 0 &\quad 1\leq \xi_-(x) <\xi_+(x),
			\end{cases}
		\end{align}
  where $\xi_-(x)\leq \xi_+(x)$ are the solutions to \eqref{eq:centraleqall} and $y(\xi)=q^{-\xi}+q^{\xi-b-c-2}$.
 \end{theorem}
 \begin{proof}
Since the hexagon is a compact domain, the moment condition \eqref{ass: moment} is satisfied.   
    Hence, the conditions in Theorem~\ref{thm:LLNcontinuousfunctions} are satisfied. Note that $a(\xi;t)$ is analytic in $\xi$ and \eqref{eq:conda} is satisfied. Note also that a step function is Riemann integrable. From here, we can apply our Theorem~\ref{thm:LLNlpfunctions}, and we use Proposition~\ref{limit_shape_qracah} to further compute the limiting density $\rho(t,x)$. 
 \end{proof}

 Examples of $\rho(t,x)$ are graphed in Figure~\ref{Fig: density}.  As expected, the corners are frozen in all cases, and their densities are either zero or one. The arctic curve is characterized by 
    \begin{align}
        \{(t,x): 4a(1;t)^2-(\mu(t,x)-b(1;t))^2=0\}.
    \end{align}
    Note that this is not an algebraic curve in the variables $t$ and $x$, but it is a degree two curve in the variables $\mu(t,x)$ and $q^{-t}$.  Furthermore, the arctic curve has 6 points tangent to the hexagon in both real and imaginary cases. In the singular trigonometric case, there are two tangential points and two non-differential points on the arctic curve. In contrast, in the singular real case, there are five tangential points and one non-differential point on the arctic curve. The tangential points $(t,x )$ can be classified in two ways. 
\begin{corollary} Intersection points of the arctic curve and the boundary of the hexagon are given by
 \begin{align}
    t=0 \text{ and }&  x= \mu^{-1,(0)}(b(1;0)) \label{eq: intersection 1} \\
     t=b+c \text{ and }&  x= \mu^{-1,(b+c)}(b(1;b+c)) \label{eq: intersection 2}
 \end{align}
and 
\begin{align}
    t= & \log_q \left(\frac{\kappa ^2 q^{b+c+1}-\kappa ^2 q^{b+c+2}-q^{b+c}+q^{b+2 c+1}-\kappa ^2 q^{c+1}+\kappa ^2 q}{-q^{b+c+1}+q^{b+2 c+1}-\kappa ^2 q^{c+2}+q^{c+1}-q^c+\kappa ^2 q}\right) \text{ and }  x= 0 \\
    t= & \log_q \left(-\frac{q^c \left(-q^{b+c}+q^{b+c+1}-\kappa ^2 q^{b+2}+q^b+\kappa ^2 q-1\right)}{\kappa ^2 q^{b+c+2}-q^{b+c+1}-\kappa ^2 q^{c+2}+q^c+\kappa ^2 q^2-\kappa ^2 q}\right) \text{ and }  x= 1+t \\
    t= & \log_q \left(\frac{q^{b+c} \left(-q^{b+c+1}+q^{2 b+c+1}-\kappa ^2 q^{b+2}+q^{b+1}-q^b+\kappa ^2 q\right)}{\kappa ^2 q^{b+c+1}-\kappa ^2 q^{b+c+2}-q^{b+c}+q^{2 b+c+1}-\kappa ^2 q^{b+1}+\kappa ^2 q}\right) \text{ and }  x= t-b \\
    t= & \log_q \left(-\frac{q^c \left(\kappa ^2 q^{b+c+2}-q^{b+c+1}-\kappa ^2 q^{b+2}+q^b+\kappa ^2 q^2-\kappa ^2 q\right)}{-q^{b+c}+q^{b+c+1}-\kappa ^2 q^{c+2}+q^c+\kappa ^2 q-1}\right) \text{ and }  x= 1+c.
\end{align}
\end{corollary} 
\begin{proof}
 The desired points are the intersection of the arctic curve with the boundary of the hexagon. For $t=0$  we have $A(1)=0$. For  $t=b+c$ we have $C(1)=0$. Hence in both cases we have $a(1;0)=a(1;b+c)=0$. Thus two points as \eqref{eq: intersection 1} and \eqref{eq: intersection 2} are found and given by $b(1;0)$ and $b(1;b+c)$. For the other points, we note that on the boundary of the hexagon, the discriminant defined in \eqref{eq: discriminant} is always zero. Recall $\xi_-(x)$ and $\xi_+(x)$ defined in Proposition~\ref{proof:qracahLimitShape}. This means at these points $\xi_-(x)=\xi_+(x)$. Moreover, the point $(t,x)$ lying on the arctic curve also means either $y(\xi_-(x))=y(1)$ or $y(\xi_+(x))=y(1)$. Then solve the equations $y(\xi_+(x_0))=y(1)$ with respect to $t$ for $x_0=0,1+t,t-b, 1+c$ respectively, and we will find the rest of the four points. 
\end{proof}
	\begin{figure}[t]
		\begin{subfigure}[b]{0.4\textwidth}
			\includegraphics[scale=0.2]{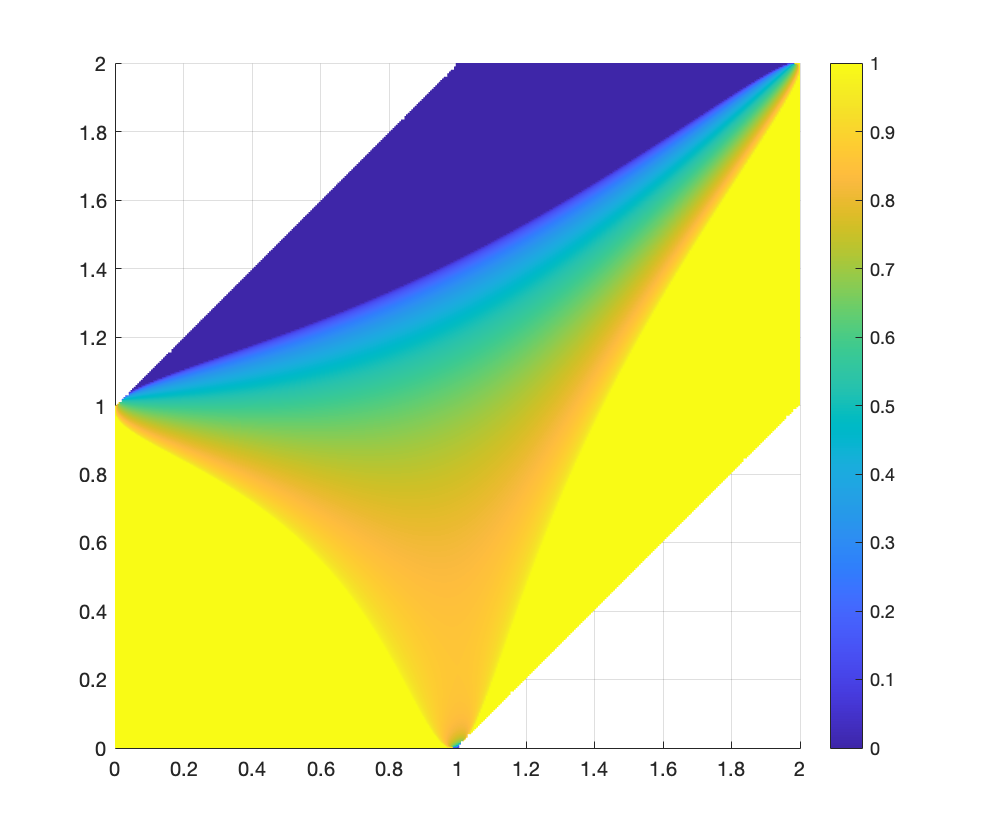}
			\caption{Real Case with size $q=0.05$, $\kappa=0.01$}
		\end{subfigure}
		\hfill
		\begin{subfigure}[b]{0.4\textwidth}
			\includegraphics[scale=0.2]{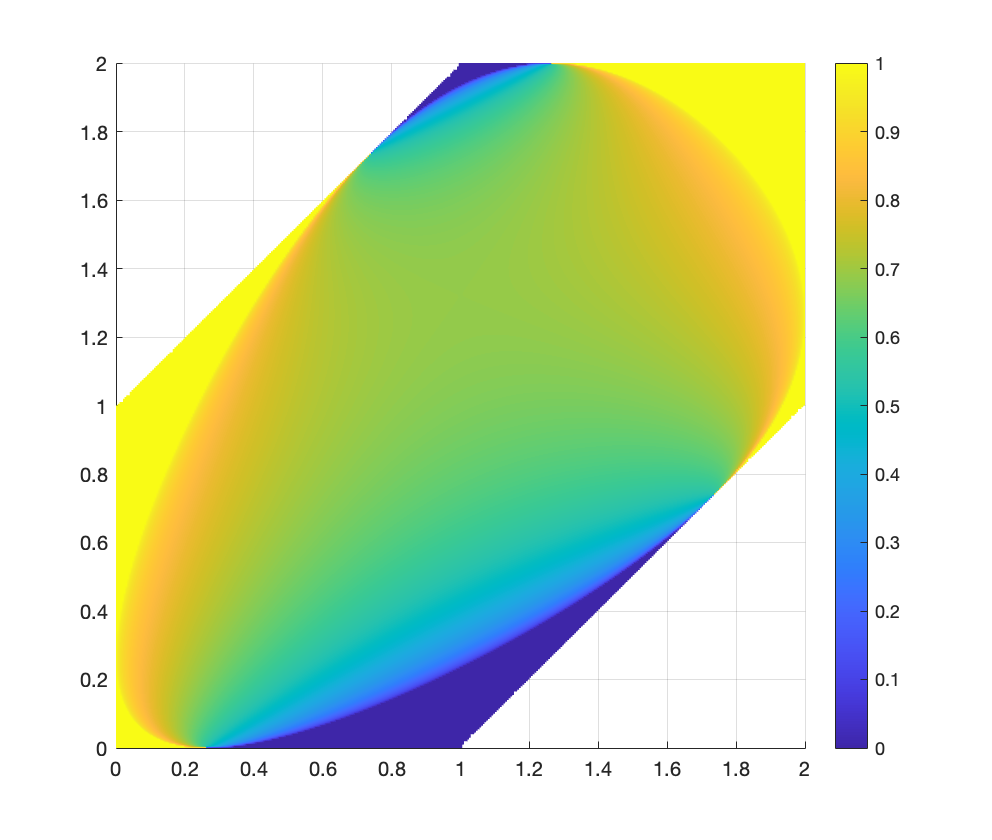}
			\caption{Imaginary Case with $q=0.5$, $\kappa=29i$}
		\end{subfigure}
		\hfill
		\begin{subfigure}[b]{0.4\textwidth}
			\includegraphics[scale=0.2]{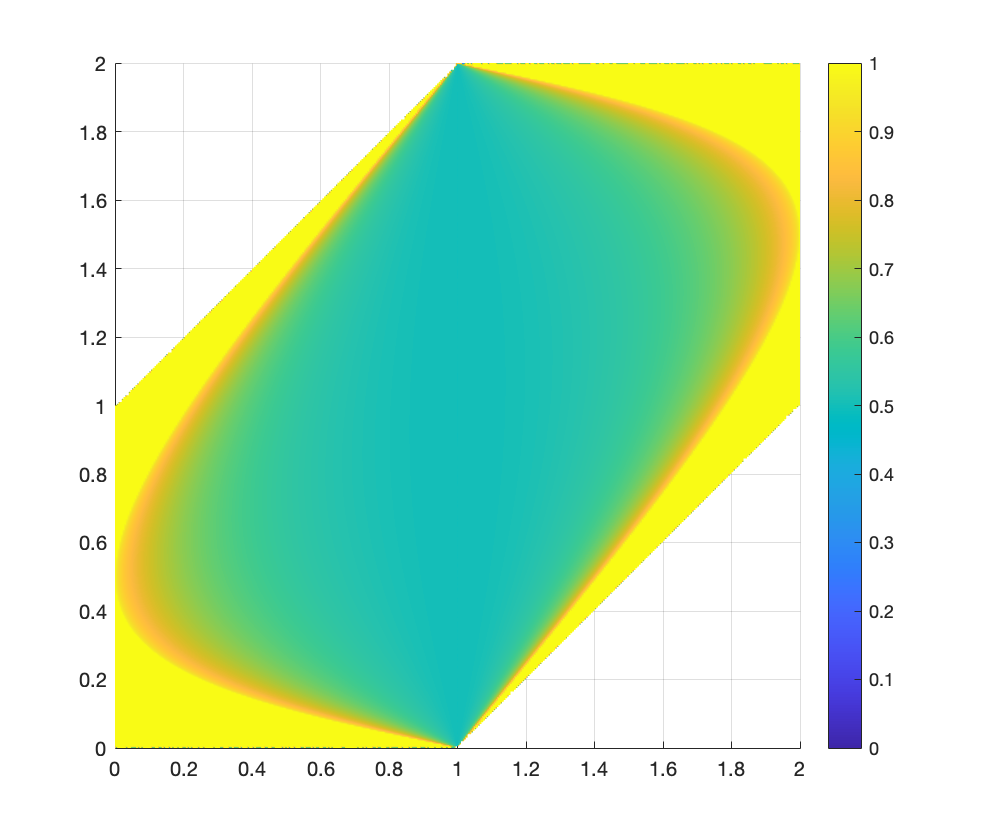}
			\caption{Singular Trigonometric Case with $q=e^{i\pi/2}$, $\kappa=e^{i\pi/2}$}
		\end{subfigure}
		\hfill
		\begin{subfigure}[b]{0.4\textwidth}
			\includegraphics[scale=0.2]{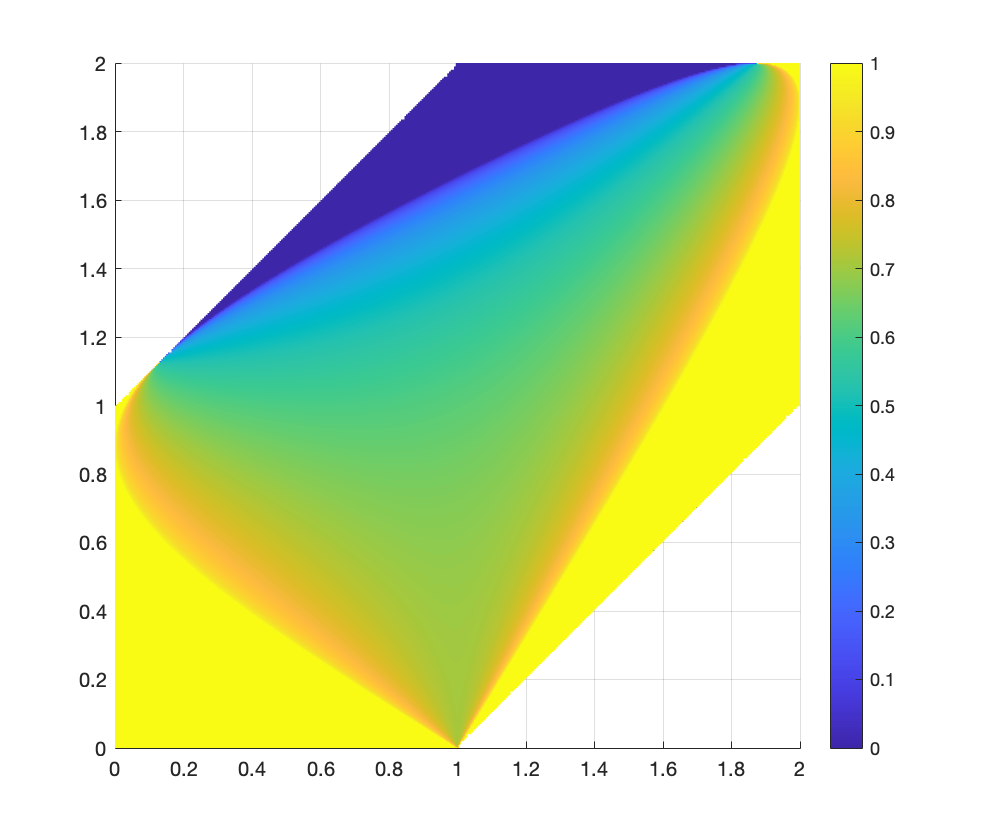}
			\caption{Singular Real Case $q=e$, $\kappa=e$}
		\end{subfigure}
		\caption{Limiting densities  \eqref{eq:limitshape} for $q$-Racah weight with $b=c=1$}
		\label{Fig: density}
	\end{figure}


	\subsection{Height Fluctuation and Gaussian Free Field} 
In this subsection, we will show how the fluctuation of the height function in the liquid region can be described by  Gaussian Free Field on the strip $S= \mathbb{R}\times (0,\pi)$, with Dirichlet boundary conditions.  
Recall that on the Sobolev space $W^{1,2}_0(S)$ we have a bilinear pairing 
\begin{align*}
\langle f,g\rangle_\nabla=\pi \int_S\langle \nabla f(x),\nabla g(x)\rangle dx. 
\end{align*}
This bilinear paring can be extended to a well-defined bilinear pairing between the space of distributions $W^{-1,2}(S)\subset \mathcal{D}'(S)$ (where $\mathcal{D}'(S)$ denotes the space of distributions on $S$ in the sense of L. Schwartz) and $W^{1,2}_0(S)$ which we still denote by $\langle \cdot,\cdot \rangle_\nabla$, (see \cite{Adams75} for definitions on negative order Sobolev spaces). 
The Gaussian Free Field (we refer to \cite{PowellWerner21} for more details) is a random variable $\Phi$ with values in $W^{-1,2}(S)$ such that for every $\phi\in W^{1,2}_0(S)$ the \emph{random variable}
\begin{equation} \label{eq:dirichlet}
    \langle \Phi,\phi\rangle_\nabla \sim N\bigg(0,\pi\int |\nabla \phi|^2\bigg).
\end{equation}

For our statement, we need to introduce new coordinates for the liquid region 
  \begin{align*}
      \mathcal{E} \coloneqq \left\{(t,x):4a(1; t)^2-\left(\mu(t,x)-b(1; t)\right)^2>0\right\}.
  \end{align*}
  To this end, with $S= \mathbb{R}\times (0,\pi)$, we define
\begin{equation}
		F  :\mathcal{E}\to
		S : (t,x) \mapsto(\tau,\theta) =\left(\tau(t),\arccos\left(\frac{\mu(t, x)-b(1; t)}{2a(1; t)}\right)\right), \label{pullback_qracah}
   \end{equation}
   where $\tau$ is as defined in \eqref{eq: tau}. One readily verifies that $F$ is a diffeomorphism  with 
   inverse
   \begin{equation}
		F^{-1}  :S\to\mathcal{E}: 
		 (\tau,\theta) \mapsto (t,x)=(t(\tau),\mu^{-1, (t(\tau))}\left(b\left(1; t(\tau)\right)+2a\left(1; t(\tau))\cos\theta\right)\right),
	\end{equation}
where $\mu^{-1,(t)}(x)$ is the inverse of $x\mapsto \mu(t,x)$. 

We will show that $h_n\circ F^{-1} -\mathbb E[ h_n \circ F^{-1}]$ converges to the Gaussian Free Field on $S$ (with Dirichlet boundary conditions). More precisely, we will prove that $$\langle h_n\circ F^{-1},\phi \rangle_\nabla -\mathbb E [\langle h_n \circ F^{-1}, \phi \rangle_\nabla] \to \langle \Phi, \phi\rangle_\nabla,$$
for all $C^2$ functions $\phi$ on $S$ such that $\Delta \phi$ has compact support and is piecewise constant in the vertical direction. That is, we assume there exists 
$$
0<t_0 <t_1 < \ldots < t_N<b+c,
$$
such that 
\begin{equation}\label{eq:piece}
    \Delta  \phi(\tau, \theta)=
\begin{cases}
\Delta \phi(\tau(t_m),\theta), & \tau \in (\tau(t_{m-1}),\tau(t_m)], \qquad m=1,\dots,N\\
0,& \text{ otherwise.}
\end{cases}
\end{equation}
For such $\phi$, the pairing \eqref{eq:dirichlet} can be written as
\begin{align*}
		\langle \Phi,\phi\rangle_\nabla  \coloneqq - \pi\sum_{m=1}^N(\tau(t_{m})-\tau(t_{m-1}))\int_0^{\pi}  \Phi(\tau(t_m),\theta) \Delta\phi(\tau(t_m),\theta)d\theta.
	\end{align*}
 Functions $\phi$ that satisfy \eqref{eq:piece} are dense in the space of test functions with the Dirichlet norm.

\begin{theorem}\label{thm:GFF}
	Consider a nonsingular $q$-Racah tiling model, i.e.,  $\mu_x(t,x)\neq 0$ for any $(t,x)$ inside the hexagon. Let $\phi$ be any twice differentiable function with compact support in $\mathbb{R}\times (0,\pi)$ such that $\Delta \phi$ is piecewise constant (as in \eqref{eq:piece}). Then
	\begin{align*}
		\langle h_n\circ F^{-1}, \phi \rangle_\nabla-\mathbb{E}[\langle h_n\circ F^{-1}, \phi\rangle] \to \mathcal{N}\left(0,\pi\int |\nabla \phi|^2 \right)
	\end{align*}
	 as $n\to\infty$ in distribution.
\end{theorem}
\begin{proof}
The proof is based on Theorem~\ref{thm:fluctuation_theorem}.

	To apply Theorem~\ref{thm:fluctuation_theorem} we first show that $\langle h_n \circ F^{-1}, \Delta \phi\rangle$ is a linear statistic. To this end, note that $h_n(t,x)=\sum_{j=0}^{\mathbf{a}-1}\chi_{\{x_j(t)\leq x\}}$ where we used the notation $\chi_A$ for the characteristic function of the set $A$. Then,
 \begin{align*}
		\langle h_n\circ F^{-1}, \phi \rangle =-\pi\sum_{m=1}^N\sum_{j=1}^{n}(\tau(t_m)-\tau(t_{m-1}))\int_{\arccos \left(\frac{\mu(t_m,x_j(t_m))-b(1;t)}{2 a(1;t)}\right)}^\infty  \Delta  \phi(\tau(t_m),\theta) d \theta.
	\end{align*}
	
	If we thus define $$f(t_m,y)=-\pi(\tau(t_m)-\tau(t_{m-1}))\int_{\arccos \left(\frac{\mu(t_m,x_j(t_m))-b(1;t)}{2 a(1;t)}\right)}^\infty  \Delta  \phi(\tau(t_m),\theta) d \theta,$$ then the above can be rewritten as a linear statistic for the point process $\{(t_m,x_j(t_m))\}_{m=1,j=0}^{N,\mathbf{a}-1}$. Note that this is a marginal of \eqref{eq: tiling dpp}, and the marginal density can be computed using repeated application of the Cauchy-Binet identity and using orthogonality of the polynomials. The result is exactly of the form \eqref{ensemble}
    and \eqref{transition} with the $q$-Racah polynomials $r_{k,t_m}$ and coefficients $c_{k,t_m}$ as in \eqref{eq:expressionsck}.  
    
	Now, we are to check the conditions of Theorem~\ref{thm:fluctuation_theorem}. First, the hexagon forms a compact domain; hence \eqref{ass: moment} is satisfied. Recall that $\mu_n(t,x)=q^{-x}+ \kappa^2 q^{-c-t+x+\frac{1}{n}}$, which converges uniformly to $\mu(t,x)=q^{-x}+ \kappa^2 q^{-c-t+x}$. Hence \eqref{ass:mu_sup} and \eqref{ass:mu} is satisfied. Moreover 
	\begin{align}
	\mu(t,\mu_n^{-1,(t)}(y))-y = \frac{1}{2} q^{-1/n} \left(q^{1/n}-1\right) \left(\sqrt{y^2-4 \kappa^2q^{-c-t+1/n}}-y\right),
	\end{align}
	and for any $y_1,y_2\in\{\mu_n(t,x): (t,x)\in\left(\frac{1}{n}\mathbb{Z}\times(\frac{1}{n}\mathbb{Z}+\frac{1}{2n})\right) \cap Hexagon\}$, we have $|y_1+y_2|$ is uniformly bounded. Moreover, for non singular $q$-Racah, we have $\mu_x(t,x)\neq 0$ for any $(t,x)$ inside the hexagon. Since $\mu(t,x)^2-4\kappa^2q^{-c-t}=(\log q)^{-2} \mu_x (t,x)^2$ and compactness, this implies that $$\inf_{(t,x)}\left(\mu(t,x)^2-4\kappa^2q^{-c-t}\right)>0.$$
 Hence, 
	\begin{align} \label{eq:lip bound y1 y2}
	    \left|\frac{\sqrt{y_1^2-4\kappa^2q^{-c-t-1/n}}-\sqrt{y_2^2-4\kappa^2q^{-c-t-1/n}}}{y_1-y_2} \right| &=  \left|\frac{y_1+y_2}{\sqrt{y_1^2-4\kappa^2q^{-c-t-1/n}}+\sqrt{y_2^2-4\kappa^2q^{-c-t-1/n}}} \right|\\
     &\leq \frac{\sup_{(t,x)}|\mu(t,x)|}{\inf_{(t,x)}\sqrt{\mu(t,x)^2-4\kappa^2q^{-t-c}}}+O(n^{-1}), \nonumber
	\end{align}	
 as $n\to \infty$,  which is uniformly bounded. In the last step, we used that $y_j=\mu(t,x_j) $ for some $x_j$ and $j=1,2$. Thus $L(\mu(t,\mu_n^{-1,(t)}))-id)=O(n^{-1})$, which shows the condition \eqref{ass:mu_lip}. Moreover, \eqref{ass:cltjaccobi} and \eqref{ass:cltc} hold by Lemmas~\ref{lem:casessame} and \ref{lem:cks}. 
	
	We have now verified all conditions of Theorem~\ref{thm:fluctuation_theorem}, and we get the following convergence in distribution,
	\begin{align*}
		X_n(f)- \mathbb{E}[X_n(f)] \to \mathcal{N}\left(0,\sigma^2(f\circ\mu^{-1})\right),
	\end{align*}
 as $n\to \infty,$
	where $\sigma^2(f\circ\mu^{-1})$ is the variance in the limit as defined as \eqref{eq: CLT Gaussian} and \eqref{eq: CLT Gaussian Variance}.  It remains to show that $\sigma^2(f\circ\mu^{-1})$ matches the desired variance. 

First recall the identity
\begin{align*}
	e^{-k|\tau|}=\frac{1}{\pi}\int_{-\infty}^\infty \frac{k}{k^2+\omega^2}e^{-i\omega \tau} d\omega.
\end{align*}
This allows us to rewrite the variance  as
\begin{align}\label{eq: variance Fourier}
	\sigma^2(f\circ\mu^{-1}) = \frac{1}{\pi}\sum_{k=1}^\infty\int_{\mathbb{R}}\frac{k^2}{\omega^2+k^2}\left|\sum_{m=1}^Ne^{-i\tau(t_m)\omega}\widehat{f\circ \mu^{-1}}_k^{(m)}\right|^2d\omega .
\end{align}
Next, note that by construction, we have
\begin{align*}
	\theta\left(t,\mu^{-1,(t)}(b(1; t)+2a(1; t)\cos(\theta))\right)=\theta,
\end{align*}
which implies
\begin{align}
	f(t,\mu^{-1,(t_m)}(b(1; t_m)+2a(1; t_m)\cos\theta)) & =- \pi \left(\tau(t_m)-\tau(t_{m-1})\right)\int_\theta^{\pi}\Delta \phi(\tau(t_m),\theta')d\theta',\nonumber
\end{align}
and thus 
\begin{align*}
	\widehat{f\circ \mu^{-1}}_k^{(m)} & = -\left(\tau(t_m)-\tau(t_{m-1})\right)\int_0^{\pi}\int_{\theta}^\pi \Delta \phi(\tau(t_m),\theta_q)d\theta_q\cos k\theta \ d\theta. 
\end{align*}
Using integration by parts, we get 
\begin{align} \label{eq: fourier 1}
	\widehat{f\circ \mu^{-1}}_k^{(m)}& = -\frac{\left(\tau(t_m)-\tau(t_{m-1})\right)}{ k}\int_0^\pi \Delta \phi(\tau(t_m),\theta)\sin k\theta \ d\theta.
\end{align}
By inserting \eqref{eq: fourier 1} into \eqref{eq: variance Fourier}, we find
\begin{multline} \label{eq:variance fourier last}
\sigma^2(f\circ\mu^{-1}) = \frac{1}{\pi}\sum_{k=1}^\infty\int_{\mathbb{R}}\frac{1}{\omega^2+k^2}\left|\sum_{m=1}^N \left(\tau(t_m)-\tau(t_{m-1})\right) e^{-i\tau(t_m)\omega}\int_0^\pi \Delta \phi(\tau(t_m),\theta)\sin k\theta \ d\theta\right|^2d\omega\\
= \frac{1}{\pi}\sum_{k=1}^\infty\int_{\mathbb{R}}\frac{1}{\omega^2+k^2}\left| \int_{\mathbb R} \int_0^\pi  e^{-i\tau\omega} \Delta \phi(\tau ,\theta)\sin k\theta \ d\theta d\tau\right|^2d\omega,
\end{multline}
where in the last step we used that $\Delta \phi$ is piecewise constant in $\tau$, and thus the sum over $m$ can be replaced by an integral. Finally, the innermost double integral at the right-hand side of \eqref{eq:variance fourier last} is the Fourier transform on the strip S for functions that vanish on the boundary. Moreover, multiplication by  $(\omega^2+k^2)^{-1}$ 
on the Fourier transform is equivalent to applying $-\Delta ^{-1}$, and thus one of the Laplace operator cancels. By Plancherel, we thus find that 
$$\sigma^2(f\circ\mu^{-1})=-\pi \int_{\mathbb R} \int_0^\pi  \phi(\tau,\theta) \Delta \phi(\tau,\theta)d\theta d\tau$$ and this proves the statement.
\end{proof}

 For the singular case, the CLT theorem fails since \eqref{eq:lip bound y1 y2} is no longer bounded. The singular case allows a situation where there exist two different points $y_1, y_2$ such that $y_i = 2\sqrt{4\kappa^2q^{-c-t-1/n}}+O(n^{-1})$ for $i=1,2$. Hence 
 \begin{align*}
     \frac{\sqrt{y_1^2-4\kappa^2q^{-c-t-1/n}}-\sqrt{y_2^2-4\kappa^2q^{-c-t-1/n}}}{y_1-y_2} =\frac{y_1+y_2} {\sqrt{y_1^2-4\kappa^2q^{-c-t-1/n}}+\sqrt{y_2^2-4\kappa^2q^{-c-t-1/n}}}
 \end{align*}
 diverges linearly with $n$.  Hence $L(\mu(t,\mu_n^{-1,(t)}))-id)$ is bounded below away from zero.


	\subsection{Complex Structure}\label{subsection:ComplexS}
	We now return to the complex structure of the liquid region. With the limiting height function at hand from Theorem~\ref{thm: lln tiling}, we mention  that, for $(t,x)$ in the liquid region, the densities of the lozenges can be expressed as: 
 \begin{align}\label{eq:def_rhoi} 
		\rho_{I}(t,x) =  -\partial_t h(t,x)\pi, \qquad
		\rho_{II}(t,x) = \partial_t h(t,x)\pi +\partial_x h(t,x)\pi, \qquad
		\rho_{III}(t,x) = \pi-\partial_x h(t,x)\pi, 
	\end{align}
 where $\rho_{I,II,III}$ stand for the limiting densities (up to factor $\pi$) of the lozenges of type $I$, $II$ and $III$ respectively (this is based on a general principle and, for example, explained in  \cite{charlier2020periodic}). We included the factor $\pi$ for convenience in the upcoming analysis. Now let $\Omega \in \mathbb C$ such that the angles of triangle $0,1$ and $\Omega$ has angle $\rho_I$ at $0$, $\rho_{II}$ at $\Omega$ and $\rho_{III}$ at $1$.  Note that we can write $\Omega$ explicitly as 
	\begin{align} \label{complexslope}
		\Omega \coloneqq  \frac{\tan(\rho_{III})}{\tan(\rho_{III})+\tan(\rho_{I})}+i\frac{\tan(\rho_{III})\tan(\rho_{I})}{\tan(\rho_{III})+\tan(\rho_{I})}
	\end{align}
We recall that $\Omega$ already appears in \eqref{eq:UOp} (in which $h=h_{tiles}$ and not the $h$ above, but they are related \eqref{eq:height_function_relation}). 	
In the following, we will show that this complex slope can be derived directly from the diffeomorphism $F$ in \eqref{pullback_qracah}. More precisely, we will show that the relation between $F$ and $\Omega$ for this particular model matches the prediction from the combination of \eqref{eq:Conf} and \eqref{eq:ComplexS} by explicit computations on the recurrence coefficients.
\begin{theorem}\label{thm: triangle picture}
  Let $F$ be the diffeomorphism defined in  \eqref{eq: working F} and $\Omega$ as in \eqref{complexslope}. Then 
 \begin{align}\label{eq:complex structure}
\frac{\partial_{\bar{z}}F}{\partial_{z}F}=-\frac{1+(i-1)\Om}{-1+(1+i)\Om} \quad \text{ for } q>1, \qquad \frac{\partial_{\bar{z}}F}{\partial_{z}F}=-\frac{1+(i-1)\overline{\Om}}{-1+(1+i)\overline{\Om}}, \quad \text{ for } 0<q<1, 
\end{align}
where $\partial_z=(\partial_t-i\partial_x)/2$ and $\partial_{\bar{z}}=(\partial_t+i\partial_x)/2$ are the Wirtinger derivatives. 
\end{theorem}

We will also show that the $\Omega$ satisfies the complex Burgers equation. 
 \begin{theorem} \label{thm: complex structure} The function $\Omega$ satisfies the inhomogeneous complex Burger's equation:
	\begin{align}\label{eq:Burgers}
			\frac{\Omega_x}{\Omega}-\frac{\Omega_t}{1-\Omega}=-\log (q)\frac{ q^{c+t}+\kappa ^2 q^{2 x}}{q^{c+t}-\kappa ^2 q^{2 x}}.
		\end{align}
	\end{theorem}

 We will prove Theorems~\ref{thm: triangle picture} and \ref{thm: complex structure} in Appendix~\ref{app:complex}.  The proofs are based on computations involving the explicit expressions for the recurrence coefficients for $F$ and $\rho_{III}$  in \eqref{pullback_qracah} and \eqref{eq:limitshapetiling}. Indeed, parts of the proof are straightforward but cumbersome computations. However, the main obstacle in the proof is that  $\rho_{III}$ itself does not define $\Omega$ and we will need either $\rho_{I}$ and $\rho_{II}$ (we will work with $\rho_I$). Note that we have not derived simple expressions for the height function other than the integral over $\rho_{III}$ and from \eqref{eq:def_rhoi}. It is an implicit formula by first integrating $\rho_{III}$ over $x$ and then differentiating with respect to $t$ that gives $\rho_I$. We solve this issue by using the consistency relation for $\rho_I$ and $\rho_{III}$. That is, given $\rho_{III}$, the angle $\rho_I$ is the unique solution to the initial value problem 
$$
\begin{cases}
    \partial_x {\rho}_{I}(t,x)=\partial_t \rho_{III}(t,x)\\
    \tan {\rho}_{I}(t,x^*(t))=0
\end{cases},$$
where $x^*(t)=\mu^{(-1),t}(b(1;t)-2a(1;t))$ is the boundary of the liquid region.  Once we proved Theorem~\ref{thm: triangle picture}, the proof of Theorem~\ref{thm: complex structure} follows by expressing $\Omega$ in terms of $F$, computing the derivatives with respect to $t$ and $x$, and substituting the resulting expressions in the left-hand side of \eqref{eq:Burgers}. As mentioned, the details are given in Appendix~\ref{app:complex}.

\subsection*{Acknowledgements}
We thank Kari Astala for allowing us to present Theorem~\ref{thm:B} in this work. We also thank Christophe Charlier for sharing his code for generating pictures for random tilings. We thank István Prause for pointing out reference \cite{MPT22Var}.

\appendix

\section{Limiting densities for the $q$-Racah ensemble}

In this section, we prove Theorem~\ref{limit_shape_qracah} and \ref{limit_shape_qracah tri}.

\subsection{Proof of Theorem~\ref{limit_shape_qracah}}\label{proof:qracahLimitShape}

 Without loss of generality, we discuss the case when $q>1$ only. For $q<1$, the proof follows from similar arguments. First recall that Theorem~\ref{thm:LLNlpfunctions} and \eqref{eq:limit_shape_density} give the following
		\begin{align*}
			\rho(x) & = \frac{\mu'(x)}{\pi }\int_{I(x)}\frac{1}{\sqrt{4a(\xi)^2-(\mu(x)-b(\xi))^2}}d\xi.
		\end{align*} 
		We compute this integral by a change of variable. Let  $A(\xi)\coloneqq\lim_{\frac{j}{n}\to\xi}A_j$ and $C(\xi)\coloneqq\lim_{\frac{j}{n}\to\xi}C_j$, where $A_j, C_j$ are from \eqref{eq:recurrence A} and \eqref{eq:recurrence C}. Then, one can compute that
		\begin{align}
			A(\xi)-C(\xi)= & \frac{(1+\alpha  \beta  q^{2 \xi }) \left(\gamma  \delta  -1\right)+q^{\xi } \left(\alpha  \left(-\beta  \delta  +\beta -\gamma +1\right)-(\beta +1) \gamma  \delta  +\beta  \delta  +\gamma \right)}{\alpha  \beta  q^{2 \xi }-1}\label{eq:A-C real}\\
			-A(\xi)-C(\xi)= &-1 -\gamma  \delta  -\frac{2 q^{2 \xi } \left( \alpha\left(\alpha \beta + \beta  \gamma + \beta +\gamma \right)+ \beta  \delta  \left(\alpha  \beta +\alpha  \gamma +\alpha +\gamma \right) \right)}{\left(\alpha  \beta  q^{2 \xi }-1\right)^2}\label{eq:-A-C real}\\
			& +\frac{\left(\alpha  \beta  q^{3 \xi }+q^{\xi }\right) \left( \alpha  (\beta +\gamma +1)+\gamma + \delta  (\beta  (\alpha +\gamma +1)+\gamma ) \right) }{\left(\alpha  \beta  q^{2 \xi }-1\right)^2}\nonumber
		\end{align}
		 Then with $y(\xi)$ as defined in \eqref{eq:def:y} we have   $y^2-4\alpha\beta=(q^{-\xi}-\alpha\beta q^{\xi})^2> 0$. By our assumptions on $\alpha$ and $\beta$, we have that $y$ is decreasing in $\xi$, for all $\xi\in (0,-\log_q \gamma)$ (it $q<1$ then $y$ is increasing in $\xi$). We make the following change of variable 
		
		\begin{align*}
			\xi (y)  =\frac{1}{\log (q)}\log \left(\frac{y-\sqrt{y^2-4 \alpha  \beta }}{2 \alpha  \beta }\right), \quad  d\xi =-\frac{1}{\ln q\sqrt{y^2-4\alpha\beta}}dy .
		\end{align*}
		
		In the new variable 
		\begin{align*}
			A(\xi(y))-C(\xi(y))= & \frac{y\left(1- \gamma  \delta  \right)+\alpha  (-\beta +\gamma -1)-\gamma + \delta  (\beta  (\alpha +\gamma -1)+\gamma )}{\sqrt{y^2-4\alpha\beta}}\\
			-A(\xi(y))-C(\xi(y))= & -1 -\gamma  \delta  - \frac{2 \left( \alpha\left(\alpha \beta + \beta  \gamma + \beta +\gamma \right)+ \beta  \delta  \left(\alpha  \beta +\alpha  \gamma +\alpha +\gamma \right) \right)}{y^2-4 \alpha  \beta } \\
			& + \frac{y \left( \alpha  (\beta +\gamma +1)+\gamma + \delta  (\beta  (\alpha +\gamma +1)+\gamma )  \right)}{y^2-4 \alpha  \beta } .
		\end{align*}
		With $b(\xi(y))=-A(\xi(y))-C(\xi(y))+1+\gamma \delta$ and $a(\xi(y))^2=A(\xi(y))C(\xi(y))$  we have
		\begin{align*}
			& 4a(\xi(y))^2-(\mu(x)-b(\xi(y)))^2 \\
			= & -(\mu(x)-1-\gamma\delta)^2-2(\mu(x)-1-\gamma\delta)(A(\xi(y))+C(\xi(y)))  -(A(\xi(y))-C(\xi(y)))^2,
		\end{align*}
  which can be written as
  \begin{align*}
      & 4a(\xi(y))^2-(\mu(x)-b(\xi(y)))^2 = \frac{\left( 4 \gamma  \delta -\mu(x) ^2\right)y^2+h(x)y+p(x)}{y^2-4\alpha\beta},
  \end{align*}
  where $h(x)$ and $p(x)$ are given by 
\begin{align}
h(x) =  2 \mu(x) (\alpha  (\beta +\gamma +1) & +\gamma )-4 \alpha  \gamma -4 \beta  \gamma \delta^2 \label{eq:h}\\
&  +\delta \left( 2 \mu(x) (\beta  (\alpha +\gamma +1)+\gamma ) -4 (\alpha  (\beta  \gamma +\beta +\gamma )+\gamma  (\beta +\gamma +1))\right),\nonumber \\
p(x) =  2 \alpha  (-2 \mu(x) (\beta  \gamma +\beta & +\gamma )+\gamma  (\beta +\gamma +1)+2 \beta  \mu(x)^2) \label{eq:p}  
-\alpha ^2 \left(\beta ^2-2 \beta  (\gamma -2 \mu(x)+1)\right) \\
-\alpha ^2 (\gamma -1)^2  -\gamma ^2 & + \delta\left( 2 (\alpha  (\beta +\gamma +1)+\gamma ) (\beta  (\alpha +\gamma -2 \mu(x)+1)+\gamma ) \right) \nonumber\\
& \qquad\qquad + \delta^2\left( -(\alpha -1)^2 \beta ^2+2 (\alpha +1) (\beta +1) \beta  \gamma -(\beta -1)^2 \gamma ^2 \right).\nonumber
\end{align}
Now, since
		\begin{align*}
			\mu(x) = q^{-x}+ \gamma\delta q^{x}, \quad \frac{\mu'(x)}{\log(q)}  =  -(q^{-x}-q^{x}\gamma\delta),
		\end{align*}
		the coefficient for $y^2$ can be written as \begin{align*}
			4\gamma\delta -  \mu(x)^2 & =-\left( \frac{\mu'(x)}{\log(q)}\right)^2,
		\end{align*}
		 and thus
		\begin{align}\label{calculation_density_qracah}
			(4a(\xi(y))^2-(\mu(x)-b(\xi(y)))^2)(y^2-4\alpha\beta) & = -\left(\frac{\mu'(x)}{\log(q)}\right)^2y^2+h(x)y+p(x),
		\end{align}
  which is \eqref{eq:def:y:ext}. Then, since our assumption on $\alpha$ and $\beta$  make $y$ injective, we indeed have  two real solutions $\xi_-(x)$ and $\xi_+(x)$ with $\xi_-(x) \leq \xi_+(x)$ of \eqref{eq: 4a2-mub2}.  Furthermore, $I(x)$ can be rewritten as 
  \begin{align*}
      I(x) = [0,1]\cap [\xi_-(x),\xi_+(x))].
  \end{align*}
		This finally allows us to compute the limiting density for $\xi_-(x)\neq \xi_+(x)$ to be
		\begin{multline*}
			\rho(x)  =  -\frac{1}{\pi}\int_{q^{-\min(\xi_-(x),1)}+\alpha\beta q^{\min(\xi_-(x),1)}}^{q^{-\min(\xi_+(x),1)}+\alpha\beta q^{\min(\xi_+(x),1)}}\frac{\mu'(x)}{\log(q)\sqrt{-\left(\frac{\mu'(x)}{\log(q)}\right)^2 y^2+h(x)y+p(x)}}dy \\
			=  \frac{1}{\pi}\arccos\left(\frac{ 2\left(\frac{\mu'(x)}{\log(q)}\right)^2(q^{-\min(\xi_+(x),1)}+\alpha\beta q^{\min(\xi_+(x),1)})-h(x)}{\sqrt{h^2(x)+4\left( \frac{\mu'(x)}{\log(q)}\right)^2p(x)}}\right)\\
			 -\frac{1}{\pi}\arccos\left(\frac{2\left(\frac{\mu'(x)}{\log(q)}\right)^2(q^{-\min(\xi_-(x),1)}+\alpha\beta q^{\min(\xi_-(x),1)})-h(x)}{\sqrt{h^2(x)+4\left( \frac{\mu'(x)}{\log(q)}\right)^2p(x)}}\right),
		\end{multline*}
which can be rewritten in the following compact form:	
	\begin{align}
	\rho(x)= \begin{cases} 1 & \quad \xi_-(x)<\xi_+(x)\leq 1,\\ \frac{1}{\pi}\arccos\left(\frac{2\left( \frac{\mu'(x)}{\log(q)}\right)^2(q^{-1}+\alpha\beta q)-h(x)}{\sqrt{h^2(x)+4\left( \frac{\mu'(x)}{\log(q)}\right)^2p(x)}}\right)& \quad \xi_+(x)>1>\xi_-(x), \\ 0 &\quad \xi_+(x)>\xi_-(x)\geq 1.\end{cases} \label{eq:q-racah density hp2}
	\end{align}
  
  To further simplify this formula, we recall that $\xi_-(x)$ and $\xi_+(x)$ are solutions to  $-\left(\frac{\mu'(x)}{\log(q)}\right)^2y^2+h(x)y+p(x)=0$, and therefore
		\begin{equation} \label{eq:hiny}
		h(x) = \left(\frac{\mu'(x)}{\log(q)}\right)^2\left(y(\xi_-(x))+y(\xi_+(x))\right), \qquad p(x) = -\left(\frac{\mu'(x)}  {\log(q)}\right)^2  y(\xi_-(x)) y(\xi_+(x)).
  \end{equation}
    
  Finally, after inserting \eqref{eq:hiny} into \eqref{eq:q-racah density hp2}, we find the statement.

\subsection{Proof of Theorem~\ref{limit_shape_qracah tri}}
	Note that in the trigonometric setting, $\nu_n(y)=\mathbf{q}^{-y}+\gamma\delta \mathbf{q}^{y}$ is not a real-valued function. However,  $	\tilde{\nu}_n(y) \coloneqq (\gamma\delta\mathbf{q})^{-\frac{1}{2}}\nu_n(y)$ is real-valued and strictly monotone under our assumptions. 
	To use Theorem~\ref{thm:LLNlpfunctions} we need to write the recurrence relations of $q$-Racah polynomial \eqref{eq:nuRacah}  in terms of $\tilde{\nu}_n$. That is before scaling 
	\begin{align}
		\tilde{\nu}_n(y)r_j(y) & =\sqrt{\tilde{A}_j\tilde{C}_{j+1}}r_{j+1}(y)+((\delta\gamma \mathbf{q})^{-\frac{1}{2}}+(\delta\gamma \mathbf{q})^{\frac{1}{2}}-\tilde{A}_j-\tilde{C}_j)r_j(y)+\sqrt{\tilde{A}_{j-1}\tilde{C}_j}r_{j-1}(y).
	\end{align}
	where
	\begin{align}
		\tilde{\nu}_n(y)=  (\gamma\delta\mathbf{q})^{-\frac{1}{2}}\nu_n(y), \qquad 	\tilde{A}_j = (\gamma\delta\mathbf{q})^{-\frac{1}{2}}A_j, \qquad \tilde{C}_j = (\gamma\delta\mathbf{q})^{-\frac{1}{2}}C_j, 
	\end{align}
	Then after scaling $\mathbf{q}=q^{\frac{1}{n}}$, $x=\frac{y}{n}$ and $\tilde{\mu}_n(x)=\tilde{\nu}_n(nx)$ we have the following limit 
	\begin{align}
		\lim_{n\to \infty}\mu_n(x) = 2\cos\left( g_q\left(x+\frac{g_\gamma+g_\delta}{2}\right) \right) \eqqcolon \tilde{\mu},
	\end{align}
which is strictly monotone under our assumptions. Additionally,
\begin{align}
	\tilde{A}(\xi)\coloneqq &  \lim_{\frac{j}{n}\to\xi}\tilde{A}_j=\frac{-4\sin\left( g_q\left(\frac{g_\gamma+\xi}{2}\right) \right)\sin\left( g_q\left(\frac{g_\alpha+\xi}{2}\right) \right)\sin\left( g_q\left(\frac{g_\alpha+g_\beta+\xi}{2}\right) \right)\sin\left( g_q\left(\frac{g_\delta+g_\beta+\xi}{2}\right) \right)}{\left(\sin\left( g_q\left(\frac{g_\alpha+g_\beta}{2}+\xi\right) \right)\right)^2}, \\
		\tilde{C}(\xi)\coloneqq &  \lim_{\frac{j}{n}\to\xi}\tilde{A}_j=\frac{-4\sin\left( \frac{g_q\xi}{2} \right)\sin\left( g_q\left(\frac{g_\beta+\xi}{2}\right) \right)\sin\left( g_q\left(\frac{g_\alpha-g_\delta+\xi}{2}\right) \right)\sin\left( g_q\left(\frac{g_\alpha+g_\beta-g_\gamma+\xi}{2}\right) \right)}{\left(\sin\left( g_q\left(\frac{g_\alpha+g_\beta}{2}+\xi\right) \right)\right)^2} . 
\end{align}
Note that $a(\xi)=\sqrt{\gamma\delta \tilde{A}(\xi)\tilde{C}(\xi)}$ and $b(\xi)=1+\delta\gamma-(\delta\gamma)^{\frac{1}{2}}\tilde{A}(\xi)-(\delta\gamma)^{\frac{1}{2}}\tilde{C}(\xi)$. Now we are ready to apply Theorem~\ref{thm:LLNlpfunctions} and obtain the limit shape density defined in \eqref{eq:limit_shape_density} to be 
\begin{align}
	\rho(x) & =  \frac{\tilde{\mu}'(x)}{\pi }\int_{I(x)}\frac{d\xi  }{\sqrt{4\tilde{A}(\xi)\tilde{C}(\xi)-\left(\tilde{\mu}(x)-2\cos\left( g_q\left(\frac{g_\gamma+g_\delta}{2}\right)\right)+\tilde{A}(\xi)+\tilde{C}(\xi)\right)^2}}. \label{eq:tri-density 1}
\end{align}
To find its primitive function, we will follow the same strategy as in the proof of Theorem~\ref{limit_shape_qracah}. Note that formulas \eqref{eq:A-C real} and \eqref{eq:-A-C real} are still valid for the trigonometric case.  Since 

\begin{equation*}
	\tilde{A}(\xi)-\tilde{C}(\xi) = \left( \gamma\delta \right)^{-\frac{1}{2}}\left( A(\xi)-C(\xi) \right), \qquad 	-\tilde{A}(\xi)-\tilde{C}(\xi) = \left( \gamma\delta \right)^{-\frac{1}{2}}\left( -A(\xi)-C(\xi) \right),
\end{equation*}
use Euler's equation, and we can write 

\begin{align}
	\tilde{A}(\xi)-\tilde{C}(\xi) = & \frac{ 2\cos\left( g_q\left( \frac{g_\alpha+g_\beta}{2}+\xi \right) \right) \sin\left(g_q\left(  \frac{g_\gamma +g_\delta}{2} \right) \right)}{\sin\left(g_q\left(  \frac{g_\alpha+g_\beta}{2} +\xi\right) \right) } 	+ \frac{c_1}{\sin\left(g_q\left(  \frac{g_\alpha+g_\beta}{2} +\xi\right) \right) } \\
		-\tilde{A}(\xi)-\tilde{C}(\xi)  = & -2\cos\left( g_q\left( \frac{g_\gamma+g_\delta}{2}\right) \right) - \frac{c_2}{\left( \sin\left( g_q\left( \frac{g_\alpha+g_\beta}{2}+\xi \right) \right) \right)^2} 	+ \frac{\cos\left( g_q\left( \frac{g_\alpha+g_\beta}{2}+\xi \right) \right)c_3}{\left( \sin\left( g_q\left( \frac{g_\alpha+g_\beta}{2}+\xi \right) \right) \right)^2},
\end{align}
where 
\begin{multline*}
	c_1 =  \sin\left(g_q\left(  \frac{g_\gamma -g_\alpha-g_\beta-g_\delta}{2} \right) \right)+ \sin\left(g_q\left(  \frac{g_\alpha+g_\beta-g_\gamma-g_\delta}{2} \right) \right) \\
	 +\sin\left(g_q\left(  \frac{g_\beta+g_\delta-g_\alpha-g_\gamma}{2} \right) \right) \sin\left(g_q\left(  \frac{g_\alpha -g_\beta-g_\gamma-g_\delta}{2} \right) \right)  \\
	c_2 =  \cos\left( g_q\left(g_\alpha- \frac{g_\gamma+g_\delta}{2}\right) \right) +\cos\left( g_q\left( \frac{g_\gamma-g_\delta}{2}\right) \right)+\cos\left(- g_q\left( \frac{g_\gamma+g_\delta}{2}\right) \right)  \cos\left( g_q\left(g_\beta+ \frac{g_\gamma-g_\delta}{2}\right) \right) ,
\end{multline*}
\begin{multline*}
		c_3 =  \cos\left(g_q\left(  \frac{g_\gamma -g_\alpha-g_\beta-g_\delta}{2} \right) \right) + \cos\left(g_q\left(  \frac{g_\alpha+g_\beta-g_\gamma-g_\delta}{2} \right) \right) \\
	+\cos\left(g_q\left(  \frac{g_\beta+g_\delta-g_\alpha-g_\gamma}{2} \right) \right) \cos\left(g_q\left(  \frac{g_\alpha -g_\beta-g_\gamma-g_\delta}{2} \right) \right)  .
\end{multline*}

 Let us define $y(\xi)=q^{-\xi}+\alpha\beta q^{\xi}$ and 
\begin{align}
	\tilde{y}(\xi) \coloneqq \left( \alpha\beta \right)^{-\frac{1}{2}} y(\xi)= 2\cos\left( g_q\left(\frac{g_\alpha+g_\beta}{2}+\xi\right) \right)
\end{align}
which is monotone in $\xi$ under our assumptions. We make the following change of variable
\begin{align*}
	\xi(\tilde{y}) = \arccos\left( \frac{\tilde{y}}{2} \right)\frac{1}{g_q}-\frac{g_\alpha+g_\beta}{2}, \qquad d\xi = -\frac{1}{2g_q\sqrt{4-\tilde{y}^2}}dy.
\end{align*} 
In the new variable $\tilde{y}$, we have 
\begin{align}
	\tilde{A}(\xi(\tilde{y}))-\tilde{C}(\xi(\tilde{y})) = & \frac{ 2\sin\left(g_q\left(  \frac{g_\gamma +g_\delta}{2} \right) \right)\tilde{y}+c_1}{\sqrt{4-\tilde{y}^2}}  \label{eq: change A-C tri}\\
	-\tilde{A}(\xi(\tilde{y}))-\tilde{C}(\xi(\tilde{y}))  = & -2\cos\left( g_q\left( \frac{g_\gamma+g_\delta}{2}\right) \right) + \frac{2c_3\tilde{y}-4c_2}{4-\tilde{y}^2} \label{eq: change -A-C tri}. 
\end{align}
Note that we can rewrite 
\begin{multline*}
	4\tilde{A}(\xi)\tilde{C}(\xi)-\left(\tilde{\mu}(x)-2\cos\left( g_q\left(\frac{g_\gamma+g_\delta}{2}\right)\right)+\tilde{A}(\xi)+\tilde{C}(\xi)\right)^2 \\
	= -\left( \tilde{\mu}(x)-2\cos\left( g_q\left(\frac{g_\gamma+g_\delta}{2}\right)\right) \right)^2 -2\left( \tilde{\mu}(x)-2\cos\left( g_q\left(\frac{g_\gamma+g_\delta}{2}\right)\right) \right)\left( \tilde{A}(\xi)+\tilde{C}(\xi) \right)-\left( \tilde{A}(\xi)-\tilde{C}(\xi) \right)^2
\end{multline*}
After inserting \eqref{eq: change A-C tri} and \eqref{eq: change -A-C tri} we find
\begin{align*}
		4\tilde{A}(\xi)\tilde{C}(\xi)-\left(\tilde{\mu}(x)-2\cos\left( g_q\left(\frac{g_\gamma+g_\delta}{2}\right)\right)+\tilde{A}(\xi)+\tilde{C}(\xi)\right)^2 = \frac{\left( \tilde{\mu}(x)^2-4 \right)\tilde{y}^2+\tilde{h}(x)\tilde{y}+\tilde{p}(x)}{4-\tilde{y}^2}, 
\end{align*}
where $\tilde{h}(x), \tilde{p}(x)$ are given by 
\begin{equation*}
		\tilde{h}(x)=  4c_3\left( \tilde{\mu}(x) -2\cos\left( g_q\left(\frac{g_\gamma+g_\delta}{2}\right)\right)\right)+4c_1\sin\left( g_q\left(\frac{g_\gamma+g_\delta}{2}\right)\right), 
\end{equation*}
\begin{multline*}
	 \tilde{p}(x) = -4 \left(\tilde{\mu}(x)-2\cos\left( g_q\left(\frac{g_\gamma+g_\delta}{2}\right)\right)\right)^2-16 \left(\tilde{\mu}(x)-2\cos\left( g_q\left(\frac{g_\gamma+g_\delta}{2}\right)\right)\right)\cos\left( g_q\left(\frac{g_\gamma+g_\delta}{2}\right)\right)\\
	 -8c_2 \left(\tilde{\mu}(x)-2\cos\left( g_q\left(\frac{g_\gamma+g_\delta}{2}\right)\right)\right)-4c_1\sin\left( g_q\left(\frac{g_\gamma+g_\delta}{2}\right)\right)-c_1^2.
\end{multline*}
Also, note that 
\begin{align*}
	4-\tilde{\mu}(x)^2=\left( \frac{\tilde{\mu}'(x)}{g_q} \right)^2.
\end{align*}
Then we have 
	\begin{align*}
		\left( 4\tilde{A}(\xi)\tilde{C}(\xi)-\left(\tilde{\mu}(x)-2\cos\left( g_q\left(\frac{g_\gamma+g_\delta}{2}\right)\right)+\tilde{A}(\xi)+\tilde{C}(\xi)\right)^2 \right) \left( 4-\tilde{y}^2 \right)=-\left( \frac{\tilde{\mu}'(x)}{g_q} \right)^2\tilde{y}^2+\tilde{h}(x)\tilde{y}+\tilde{p}(x).
	\end{align*}

Since our assumption on $\alpha$ and $\beta$ makes $y$ injective, we have two solutions $\xi_-(x)$ and $\xi_+(x)$ with $\xi_-(x) \leq \xi_+(x)$ such that 
 \begin{align*}
 	-\left( \frac{\tilde{\mu}'(x)}{g_q} \right)^2\tilde{y}\left(\xi_\pm(x)\right)^2+\tilde{h}(x)\tilde{y}(\xi_\pm(x))+\tilde{p}(x) =0.
 \end{align*} 
 That is to say $4a(\xi_\pm(x))^2-(\mu(x)-b(\xi_\pm(x)))^2=0$. Furthermore, $I(x)$ can be rewritten as 
\begin{align*}
	I(x) = [0,1]\cap [\xi_-(x),\xi_+(x))].
\end{align*}
This finally allows us to compute the limiting density for $\xi_-(x)\neq \xi_+(x)$ to be
\begin{multline*}
	\rho(x)  = -\frac{1}{\pi}\int_{2\cos\left( g_q\left(\frac{g_\alpha+g_\beta}{2}+\min(\xi_-(x),1)\right) \right)}^{2\cos\left( g_q\left(\frac{g_\alpha+g_\beta}{2}+\min(\xi_+(x),1)\right) \right)}\frac{\tilde{\mu}'(x)}{g_q\sqrt{-\left(\frac{\tilde{\mu}'(x)}{g_q}\right)^2 \tilde{y}^2+\tilde{h}(x)\tilde{y}+\tilde{p}(x)}}d\tilde{y} \\
	= \frac{1}{\pi}\arccos\left(\frac{ 4\left(\frac{\tilde{\mu}'(x)}{g_q}\right)^2\cos\left( g_q\left(\frac{g_\alpha+g_\beta}{2}+\min(\xi_+(x),1)\right) \right)-\tilde{h}(x)}{\sqrt{h^2(x)+4\left(\frac{\tilde{\mu}'(x)}{g_q}\right)^2\tilde{p}(x)}}\right)\\
	-\frac{1}{\pi}\arccos\left(\frac{4\left(\frac{\tilde{\mu}'(x)}{g_q}\right)^2\cos\left( g_q\left(\frac{g_\alpha+g_\beta}{2}+\min(\xi_-(x),1)\right) \right)-\tilde{h}(x)}{\sqrt{\tilde{h}^2(x)+4\left(\frac{\tilde{\mu}'(x)}{g_q}\right)^2 \tilde{p}(x)}}\right),
\end{multline*}
which can be rewritten in the following compact form:

\begin{align}
	\rho(x)= \begin{cases} 1 & \quad \xi_-(x)<\xi_+(x)\leq 1,\\ \frac{1}{\pi}\arccos\left(\frac{2\left(\frac{\tilde{\mu}'(x)}{g_q}\right)^2\cos\left( g_q\left(\frac{g_\alpha+g_\beta}{2}+1\right) \right)-\tilde{h}(x)}{\sqrt{\tilde{h}^2(x)+\left(\frac{\tilde{\mu}'(x)}{g_q}\right)^2\tilde{p}(x)}}\right)& \quad \xi_+(x)>1>\xi_-(x), \\ 0 &\quad \xi_+(x)>\xi_-(x)\geq 1.
	\end{cases} \label{eq:q-racah density hp}
\end{align}

To further simplify this formula, we recall that $\xi_-(x)$ and $\xi_+(x)$ give solutions to  $-\left(\frac{\tilde{\mu}'(x)}{g_q}\right)^2\tilde{y}^2+\tilde{h}(x)\tilde{y}+\tilde{p}(x)=0$, and therefore
\begin{equation} \label{eq:hiny2}
	\tilde{h}(x) = \left(\frac{\tilde{\mu}'(x)}{g_q}\right)^2\left(\tilde{y}(\xi_-(x))+\tilde{y}(\xi_+(x))\right) ,\qquad \tilde{p}(x) = -\left(\frac{\tilde{\mu}'(x)}{g_q}\right)^2\tilde{y}(\xi_-(x)) \tilde{y}(\xi_+(x)).
\end{equation}
Finally, after inserting \eqref{eq:hiny2} into \eqref{eq:q-racah density hp}  and we find the statement.


\section{Complex structure} \label{app:complex}

We prove Theorems~\ref{thm: triangle picture} and \ref{thm: complex structure}.
Before we come to the proofs, we start with some preliminaries.

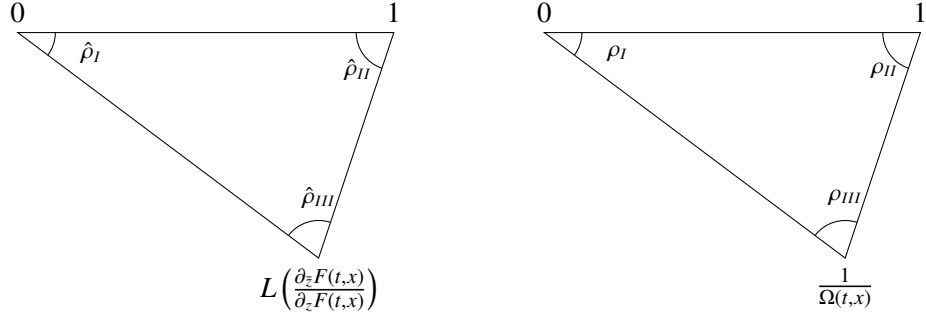
\begin{figure}
\centering
\begin{tikzpicture}[
	my angle/.style = {draw,
	angle radius=5mm, 
	angle eccentricity=1.7,
	right, inner sep=1pt,
					font=\footnotesize} 
				]
				\draw   (0,0) coordinate[label=above:$0$] (a) --
				(5,0) coordinate[label=above:$1$] (c) --
				(4,-3) coordinate[label=below:$	L\left(\frac{\partial_{\bar{z}}F(t,x)}{\partial_{z}F(t,x)}\right)$] (b) -- cycle;
				\pic[my angle, "$\hat\rho_{I}$"] {angle = b--a--c};
				\pic[my angle, "$\hat\rho_{II}$"] {angle = a--c--b};
				\pic[my angle, "$\hat\rho_{III}$"] {angle = c--b--a};
                \draw   (7,0) coordinate[label=above:$0$] (a) --
				(12,0) coordinate[label=above:$1$] (c) --
				(11,-3) coordinate[label=below:$ \frac{1}{\Omega(t,x)}$] (b) -- cycle;
				\pic[my angle, "$\rho_{I}$"] {angle = b--a--c};
				\pic[my angle, "$\rho_{II}$"] {angle = a--c--b};
				\pic[my angle, "$\rho_{III}$"] {angle = c--b--a};
	\end{tikzpicture} 
   \caption{The triangles in the proof of Theorem~\ref{thm: triangle picture} for $q>1$. In the proof we show $\hat \rho_j=\rho_j$ and thus $L\left(\frac{\partial_{\bar{z}}F}{\partial_{z}F}\right)=1/\Omega$.}
   \label{fig:triangle updown}
  \end{figure}
   
\subsection{Preliminaries}
 First, consider the M\"obius transformation from the complement of the unit disc to the upper half plane given by 
	\begin{align}\label{mobius}
			L(z)=\frac{(1-i) z+(1+i)}{i-i z}.
	\end{align}
Then, to show \eqref{eq:complex structure} is equivalent to show 
\begin{align}\label{eq: complex structure L}
L\left(\frac{\partial_{\bar{z}}F}{\partial_{z}F}\right)=\Omega^{-1}, \quad \text{ for $q>1$}, \qquad  L\left(\frac{\partial_{\bar{z}}F}{\partial_{z}F}\right)=\overline{\Omega}^{-1}, \quad \text{for $0<q<1$}.
\end{align}
By  definition of $\partial_z, \partial_{\bar{z}}$ and $L$  we find
	\begin{align}
L\left(\frac{\partial_{\bar{z}}F(t,x)}{\partial_{z}F(t,x)}\right)= \frac{F_t(t,x)}{F_x(t,x)}+1,
	\end{align}
 and this will helps us finding an expression for $L\left(\frac{\partial_{\bar{z}}F(t,x)}{\partial_{z}F(t,x)}\right)$.  To this end, we start by recalling that the limiting recurrence coefficients of the $q$-Racah polynomials at $\xi=1$ are (see also \eqref{eq:recurrance A} and \eqref{eq:recurrance CC})
	\begin{align} \label{eq:At}
		 A(t) =A(\xi;t)|_{\xi=1} =\frac{\left(1-q^{-c}\right) \left(1-q^{-t}\right) \left(1-q^{-b-c-1}\right) \left(1-\kappa ^2 q^{-b-c+1}\right)}{\left(1-q^{-b-c}\right)^2}, \\
		C(t)   =C(\xi;t)|_{\xi=1}  = \frac{\left(1-\frac{1}{q}\right) \left(1-q^{-b}\right) \left(1-\kappa ^2 q\right) q^{-c-t} \left(1-q^{-b-c+t}\right)}{\left(1-q^{-b-c}\right)^2}.  \label{eq:Ct}
	\end{align}
Since from now on, we will always set $\xi=1$ in the expression for $A(\xi,t)$, we drop the $\xi$-dependence from the notation and view $A$ as a function of $t$ only.
We will also use the notation
\begin{align}
    Y(t,x)\coloneqq\mu(t,x)+1+\kappa^2 q^{-t-c}-A(t)-C(t).
\end{align}
 We will use the subscript to denote the partial derivatives with respect to $x$ or $t$ and $'$ to denote the derivatives of single variable functions. For example,
\begin{multline}
    Y_x \coloneqq \partial_x Y, \qquad Y_{xx} \coloneqq \partial_x^2Y, \qquad Y_{tx} \coloneqq \partial_x\partial_tY, \\
    A'\coloneqq\frac{d}{dt}, \qquad  A''\coloneqq\frac{d^2}{dt^2}A, \qquad C'\coloneqq\frac{d}{dt}, \qquad  C'' \coloneqq\frac{d^2}{dt^2}C.
\end{multline} 
The following equations are straightforward to verify but will be important in our proofs. 
\begin{multline}\label{diff:mutx}
    Y_{xx}+2Y_{tx}+	Y_{tt} =-\log(q) (Y_x+Y_t),\\
			Y_{tt}  =-\log(q) Y_t, \qquad
			A'' =-\log(q)	\ A' , \qquad
			C'' =-\log(q)\ C'.
\end{multline}
			
 \begin{lemma}\label{lemma: F lower upper plane}
	Let $F$ be the diffeomorphism as defined in \eqref{pullback_qracah}. Then, its partial derivatives are
  \begin{align}
        F_t = & \frac{1}{2} \left(\frac{A'}{A}-\frac{C'}{C(1)}\right)+i\frac{ \left(A' C+AC'\right)Y-2 AC Y_t}{2 (A C\sqrt{4A(1) C-Y^2}}, \\
        F_x = & -\frac{i Y_x}{ \sqrt{4A C-Y^2}}.
    \end{align}
    Hence, $\frac{\partial_tF(t,x)}{\partial_xF(t,x)}$ lies in the lower (or upper) half plane for $q>1$ (or $0<q<1$) whenever $(t,x)$ is inside the liquid region.
	\end{lemma}
 \begin{proof}
 By \eqref{eq: tau}, \eqref{eq:At} and \eqref{eq:Ct} we find
\begin{align*}
   \tau(t) -\frac{1}{2}\log\left(\frac{A}{C}\right) = -\frac{1}{2}\log\left|\frac{(q^c-1)(1-q^{-b-c-1})(1-\kappa^2q^{-b-c+1})}{(1-q^{-1})(1-q^{-b})(1-\kappa^2q)}\right|\eqqcolon constant. 
\end{align*}
Then $F$ can be rewritten as 
		\begin{align}
			F=\frac{1}{2}\log\left(\frac{A}{C}\right)+i\arccos{\frac{Y}{2\sqrt{AC}}} + \text{constant} \label{eq: working F}.
		\end{align}
Strictly speaking, this only holds for $0<t\leq b$, but by Lemma~\ref{lem:casessame}, we see that the extra factors in the argument of the arccosine cancel out. The first statement comes from a direct computation. To show the second statement, we note that
\begin{align*}
    \Im\left(\frac{\partial_tF(t,x)}{\partial_xF(t,x)}\right) = -\frac{\log(q)} {\partial_x\mu(t,x)}\frac{q^t\left(q^{b+c}-1\right)}{2 \left(q^t-1\right) \left(q^{b+c}-q^t\right)}\frac{\sqrt{4A(t)C(t)-Y(t,x)^2} }{4\sqrt{A(t)C(t)}}.
\end{align*} 
Recall that $0<t<b+c$ and hence  $\left(q^t-1\right) \left(q^{b+c}-q^t\right)>0$. Moreover, $\mu(t,x)=q^{x}+\kappa^2 q^{-t-c} q^{-x}$. By our choice of parameters, we have that $\frac{\partial_x\mu(t,x)}{\log(q)}$ is positive. Hence, $\frac{\partial_tF}{\partial_xF}$ is in the upper half plane for $0<q<1$ and in the lower half plane for $q>1$. 
 \end{proof}

We will also need the following.
\begin{lemma}\label{lemma:inverse omega}
    \begin{align}
        \frac{1}{\Omega(t,x) }=	\frac{\tan(\rho_{II})}{\tan(\rho_{II})+\tan(\rho_{I})}-i\frac{\tan(\rho_{II})\tan(\rho_{I})}{\tan(\rho_{II})+\tan(\rho_{I})}.
    \end{align}
\end{lemma}
\begin{proof}
    Recall the definition of $\Omega$ in \eqref{complexslope}. Then the statement holds by applying the following sum formula for tangent 
\begin{align}\label{eq: sum tangent}
\tan(\rho_{I}+\rho_{II})=\frac{\tan (\rho_{I})+\tan( \rho_{II})}{1-\tan (\rho_{I})\tan (\rho_{II})}, 
\end{align}
and $\tan(\rho_{I}+\rho_{II})=-\tan(\rho_{III})$.
\end{proof}
 Now we are ready for the proofs of Theorems~\ref{thm: triangle picture} and \ref{thm: complex structure}.
 
\subsection{Proof of Theorem~\ref{thm: triangle picture}}
By \eqref{eq: complex structure L} and Lemma~\ref{lemma:inverse omega} it is sufficient to show that 
  \begin{align}\label{mobiustoslope}
L\left(\frac{\partial_{\bar{z}}F}{\partial_{z}F}\right)=\frac{\tan(\rho_{II})}{\tan(\rho_{II})+\tan(\rho_{I})}-i\frac{\tan(\rho_{II})\tan(\rho_{I})}{\tan(\rho_{II})+\tan(\rho_{I})}, \quad \text{ for } q>1, \\
L\left(\frac{\partial_{\bar{z}}F}{\partial_{z}F}\right)=\frac{\tan(\rho_{II})}{\tan(\rho_{II})+\tan(\rho_{I})}+i\frac{\tan(\rho_{II})\tan(\rho_{I})}{\tan(\rho_{II})+\tan(\rho_{I})}, \quad \text{ for } 0<q<1. \nonumber
\end{align}
We denote the angles of the triangle formed by $0$,  $L\left(\frac{\partial_{\bar{z}}F}{\partial_{z}F}\right)$ and $1$, by $\hat \rho_I$ at 0, $\hat \rho_{II}$ at $L\left(\frac{\partial_{\bar{z}}F}{\partial_{z}F}\right)$ and $\hat \rho_{III}$ at $1$. See also Figure~\ref{fig:triangle updown}. To prove the theorem, it is sufficient to show that $\hat \rho_I=\rho_I$ and $\hat \rho_{III}=\rho_{III}$ for either of the cases. 
\bigskip

\textbf{STEP 1 : $\hat \rho_{III}=\rho_{III}$ }

Recall the density of the type III lozenge is given by \eqref{eq:def_rhoi}. It can be further computed via Theorem~\ref{limit_shape_qracah} to be
	\begin{align*}
		\rho_{III}(t,x)= & \pi-\arccos\left(\frac{2y(1)-y(\xi_-(x))-y(\xi_+(x))}{\left|y(\xi_-(x))-y(\xi_+(x))\right|}\right).
    \end{align*}
       More explicitly, by \eqref{eq:q-racah density hp}, we can rewrite the above to be 
       \begin{align*}
			\rho_{III}(t,x)= & \arccos\left(\frac{\alpha_1}{\alpha_2}\right),
   \end{align*}
   where 
   \begin{align*}
   \alpha_1 \coloneqq  -\left(2\left( \frac{\mu'(x)}{\log(q)}\right)^2(q^{-1}+\alpha\beta q)-h(x)\right), \quad 
   \alpha_2 \coloneqq \sqrt{h^2(x)+4\left( \frac{\mu'(x)}{\log(q)}\right)^2p(x)},
       \end{align*}
       and $h(x)$ and $p(x)$ are defined by \eqref{eq:h} and \eqref{eq:p}. By the trigonometric identity, we have
       \begin{align*}
           \tan \rho_{III}(t,x)=\frac{\sqrt{\alpha_2^2-\alpha_1^2}}{\alpha_1} .
       \end{align*}
       By the relation \eqref{calculation_density_qracah} we have $\alpha_2^2-\alpha_1^2=4\left(\frac{\partial_x\mu(t,x)}{\log(q)}\right)^2(q^{-1}-\alpha\beta q)^2(4 A(1) C(1)-Y(t,x)^2)$, where $(q^{-1}-\alpha\beta q)=(q^{-1}-q^{-b-c-1})$ and
       \begin{align*}
           \alpha_1= & 2 q^{-b-c-t-2 x-2} \left(-q^{b+c+t+1}+q^{b+c+x+1}+q^{b+t+x+1}+q^{b+x}-2 q^{b+2 x}+q^{t+x}-q^{t+1}\right) \\
           & + \kappa^2 2 q^{-b-2 c-2 t-1} \left(\left(q^{b+1}+q^{t+1}+1+q^{t-c}\right) q^{c+t-x}+2 q^t \left(q^{b+c}+1\right)+q^{x-1} \left(q^{b+t+1}+q^b+q^t+q^{b+c+1}\right)\right) \\
           & + \kappa^2 2 q^{-b-2 c-2 t-1} \left( -2 q^{t-1} \left(q^{b+1}+q^{t+1}+1\right)-2 q^c \left(q^{b+t+1}+q^b+q^t\right)\right) \\
           & + \kappa^4 2 q^{-b-3 c-2 t-1} \left(q^{b+c+x+1}-q^{b+c+2 x}+q^{c+t+x+1}-2 q^{c+t+1}+q^{c+x}+q^{t+x}-q^{2 x}\right).
       \end{align*}

Next, we turn to $\hat \rho_{III}$. A simple computation shows that 
		\begin{align} \label{eq:rhoiii tri upsidedown}
			\tan \hat \rho_{III}   = & \frac{\left|\Im\left(L\left(\frac{\partial_{\bar{z}}F(t,x)}{\partial_{z}F(t,x)}\right)\right)\right|}{\Im\left(L\left(\frac{\partial_{\bar{z}}F(t,x)}{\partial_{z}F(t,x)}\right)\right)^2+\Re\left(L\left(\frac{\partial_{\bar{z}}F(t,x)}{\partial_{z}F(t,x)}\right)\right)(\Re\left(L\left(\frac{\partial_{\bar{z}}F(t,x)}{\partial_{z}F(t,x)}\right)\right)-1)}.
		\end{align}
  Now note that $L\left(\frac{\partial_{\bar{z}}F}{\partial_{z}F}\right)=\frac{\partial_{t}F}{\partial_{x}F}+1$, and thus  $\tan \hat \rho_{III}$, is explicitly given in Lemma~\ref{lemma: F lower upper plane}. Then, by a cumbersome but straightforward computation, we find that $\tan \rho_{III}= \tan \hat \rho_{III}$. 
\bigskip

\textbf{STEP 2 : $\hat \rho_I=\rho_I$.}

We define two differential operators 
	\begin{align}\label{eq:def: d1 d2}
\mathcal{D}_1\coloneqq\partial_t+\partial_x \quad \mathcal{D}_2\coloneqq\partial_t.
	\end{align}
It is straightforward from \eqref{eq:def_rhoi} that $\rho_I$ satisfies the following initial value problem 
$$
\begin{cases}
    \partial_x \rho_{I}(t,x)=\mathcal D_2\rho_{III}\\
    \tan \rho_{I}(t,x^*(t))=0
\end{cases},$$
where $x^*(t)=\mu^{(-1),t}(b(1;t)-2a(1;t))$ is the boundary of the liquid region. In fact, this initial value problem defines $\rho_{I}$ uniquely. We will complete the proof by showing that $\hat\rho_I$ also solves this equation. 
 Since $A$ and $C$ do not depend on $x$, we have $\mathcal{D}_1(AC=\mathcal{D}_2(AC)$.  
 
Define 
\begin{align} 
	R_j\coloneqq & 2AC\mathcal{D}_j(Y)-\mathcal{D}_j(AC) Y,  \quad \text{for }j=1,2 ,\label{eq:def R} \\
	I\coloneqq & (A'C-C'A)\sqrt{4AC-Y^2}. \label{eq:def I}
\end{align}
Then $L$ can be rewritten as
\begin{align}\label{eq: L working}
L\left(\frac{\partial_{\bar{z}}F}{\partial_{z}F}\right) =\frac{R_1}{2ACY_x}+i\frac{I}{2ACY_x} ,\quad
	L\left(\frac{\partial_{\bar{z}}F}{\partial_{z}F}\right)-1 =\frac{R_2}{2ACY_x}+i\frac{I}{2AC Y_x}.
\end{align}
Then $ \mathcal{D}_1\arg\left(L\left(\frac{\partial_{\bar{z}}F}{\partial_{z}F}\right) \right) =  \mathcal{D}_1\arctan \frac{I}{R_1}$ and  $
\mathcal{D}_2\arg\left(L\left(\frac{\partial_{\bar{z}}F}{\partial_{z}F}\right) -1\right) =  \mathcal{D}_2\arctan \frac{I}{R_2}$. Using the for differential equations \eqref{diff:mutx} (somewhat surprisingly, these differential equations are the only ingredients of the rest of the proof), we have the following handy relations  
\begin{align}
	\mathcal{D}_j(R_j)  = &-\log(q) R_j -\mathcal{D}_j(AC)\mathcal{D}_j(Y)+2A'C'Y, \label{eq: handy1}\\
	\mathcal{D}_j(I) = & I\left(-\log(q)  +\frac{2\mathcal{D}_j(AC)-Y\mathcal{D}_j(Y)}{4AC-Y^2}\right).\label{eq: handy2}
\end{align}

Also, note that
\begin{equation}
\frac{I^2+R_j^2}{4AC}=
A'C'Y^2-\mathcal{D}_j(AC)\mathcal{D}_j(Y)Y+AC\mathcal{D}_j(Y)^2+\mathcal{D}_j(AC)^2-4AA'CC'. \label{eq: handy3}
\end{equation} 
Then combining \eqref{eq: handy1}, \eqref{eq: handy2}  and \eqref{eq: handy3},  one can directly compute, for $j=1,2$
\begin{align}\label{eq:angle equation}
	\mathcal{D}_j\arctan \frac{I}{R_j} = \frac{YI}{2AC(4AC-Y^2)},
\end{align}
which shows
\begin{align}\label{eq:Burgers0 imaginary}
	\mathcal{D}_2\arg\left(L\left(\frac{\partial_{\bar{z}}F}{\partial_{z}F}\right) -1\right) =  \mathcal{D}_1\arg\left(L\left(\frac{\partial_{\bar{z}}F}{\partial_{z}F}\right) \right).
\end{align}
Since $\arg\left(L\left(\frac{\partial_{\bar{z}}F}{\partial_{z}F}\right) -1\right)=\pi-\hat \rho_{III}+\arg\left(L\left(\frac{\partial_{\bar{z}}F}{\partial_{z}F}\right) \right)$. Hence $\hat \rho_1=-\arg\left(L\left(\frac{\partial_{\bar{z}}F}{\partial_{z}F}\right) \right)$ we have 
\begin{align}		\mathcal{D}_1\hat{\rho_I}+\mathcal{D}_2(\pi-\hat {\rho_I}-\hat \rho_{III}) = 0,
\end{align}
and since we already proved that $\hat \rho_{III}= \rho_{III}$  and $\mathcal{D}_1-\mathcal{D}_2=\partial_x$, we find that 
 $\partial_x \hat \rho_{I}(t,x)=\mathcal D_1\rho_{III}.$
 Moreover, $I$ vanishes so long as $x=x^*(t)$ and thus the boundary condition is satisfied. We conclude that $\hat \rho_I=\rho_I$ and this finishes the proof.

\subsection{Proof of Theorem~\ref{thm: complex structure}}
We continue to work with $L$ in the form \eqref{eq: L working} with $R_j$ and $I$ defined as in \eqref{eq:def R} and \eqref{eq:def I}. Now by \eqref{eq: handy1}, \eqref{eq: handy2}, we get
\begin{multline*}
	\frac{I\mathcal{D}_j(I)+R_j\mathcal{D}_j(R_j)}{I^2+R_j^2} = -\log(q) \\+ \frac{I^2\left(2\mathcal{D}_j(AC)-Y\mathcal{D}_j(Y)\right)+R_j(4AC-Y^2)\left(-\mathcal{D}_j(AC)\mathcal{D}_j(Y)+2A'C'Y\right)}{(I^2+R_j^2)(4AC-Y^2)}.
\end{multline*}
By further simplifying the numerator using the definition of $I, R_j$ using   \eqref{eq: handy3} for the denominator, we find
\begin{align*}
	\frac{I\mathcal{D}_j(I)+R_j\mathcal{D}_j(R_j)}{I^2+R_j^2} = -\log(q) + \frac{\mathcal{D}_j(AC)}{2A C}.
\end{align*}
Hence 
\begin{align*}
	\frac{1}{2}\mathcal{D}_j\log \left|\frac{I^2+R_j^2}{(2ACY_x)^2}\right| =  & \frac{I\mathcal{D}_j(I)+R_j\mathcal{D}_j(R_j)}{I^2+R_j^2}-\mathcal{D}_j\log(2AC)-\frac{\mathcal{D}_j(Y_x)}{Y_x}\\
	= & -\log(q) - \frac{\mathcal{D}_j(AC)}{2AC}-\frac{\mathcal{D}_j(Y_x)}{Y_x}. 
\end{align*}
Since $A $ and $C$ are functions in $t$ only, we have $\mathcal{D}_1\log(AC)=\mathcal{D}_2\log(AC)$. Additionally, $\mathcal{D}_1(Y_x)-\mathcal{D}_2(Y_x)=Y_{xx}$ and hence 
\begin{align}
	-\mathcal{D}_1\log \left|L\left(\frac{\partial_{\bar{z}}F}{\partial_{z}F}\right)\right|+\mathcal{D}_2\log \left|L\left(\frac{\partial_{\bar{z}}F}{\partial_{z}F}\right)-1\right| = & \frac{\mu_{xx}}{\mu_x}. \label{eq:Burgers0 real} 
\end{align}
Summing up \eqref{eq:Burgers0 real} and \eqref{eq:Burgers0 imaginary} from the proof of Theorem~\ref{thm: triangle picture}, we obtain the following,
\begin{align}\label{Burgers0}
	-\frac{\mathcal{D}_1L\left(\frac{\partial_{\bar{z}}F}{\partial_{z}F}\right)}{L\left(\frac{\partial_{\bar{z}}F}{\partial_{z}F}\right)}+\frac{\mathcal{D}_2L\left(\frac{\partial_{\bar{z}}F}{\partial_{z}F}\right)}{L\left(\frac{\partial_{\bar{z}}F}{\partial_{z}F}\right)-1}=  \frac{\mu_{xx}}{\mu_x}. 
\end{align}
For $q>1$,
Theorem~\ref{thm: triangle picture} shows that $L\left(\frac{\partial_{\bar{z}}F}{\partial_{z}F}\right)=\frac{1}{\Omega}$. Then, by the chain rule, \eqref{Burgers0} is equivalent to \eqref{eq:Burgers}. 

For $0<q<1$, Theorem~\ref{thm: triangle picture} shows that $L\left(\frac{\partial_{\bar{z}}F}{\partial_{z}F}\right)=\frac{1}{\overline{\Omega}}$. Note that \eqref{Burgers0} is a real equation. Taking the conjugate of \eqref{Burgers0} and using the chain rule, we also obtain \eqref{eq:Burgers}. This finishes the proof.

\bibliographystyle{amsplain}

\end{document}